\documentclass[11pt]{article}
 \usepackage{hyperref}
 \usepackage{benstyle}
 
\usepackage{enumerate}
\usepackage{longtable}

\DefineNamedColor{named}{forestgreen} {rgb}{0.13,0.54,0.13}

\DefineNamedColor{named}{oceanfoam} {rgb}{0.34, 0.52, 0.54}

\DefineNamedColor{named}{burgundy}{rgb}{0.61, 0, 0}

\def\bnu{\bm{\nu}}

\def\tkappa{\tilde{\kappa}}

\usepackage{cite}

\title{Optimal approximation of infinite-dimensional holomorphic functions II: recovery from i.i.d.\ pointwise samples}
\author{Ben Adcock\thanks{Department of Mathematics, Simon Fraser University, 8888 University Drive, Burnaby BC, Canada, V5A 1S6.} \and Nick Dexter\thanks{Department of Scientific Computing, Florida State University, 400 Dirac Science Library, Tallahassee, Florida,
USA, 32306-4120.} \and Sebastian Moraga\footnotemark[1]}

\begin{document}

\maketitle

\begin{abstract}
Infinite-dimensional, holomorphic functions have been studied in detail over the last several decades, due to their relevance to parametric differential equations and computational uncertainty quantification. The approximation of such functions from finitely-many samples is of particular interest, due to the practical importance of constructing surrogate models to complex mathematical models of physical processes. In a previous work, \cite{adcock2024optimal} we studied the approximation of so-called Banach-valued, $(\bm{b},\varepsilon)$-holomorphic functions on the infinite-dimensional hypercube $[-1,1]^{\bbN}$ from $m$ (potentially adaptive) samples. In particular, we derived lower bounds for the adaptive $m$-widths for classes of such functions, which showed that certain algebraic rates of the form $m^{1/2-1/p}$ are the best possible regardless of the sampling-recovery pair. In this work, we continue this investigation by focusing on the practical case where the samples are pointwise evaluations drawn identically and independently from the underlying probability measure for the problem. Specifically, for Hilbert-valued $(\bm{b},\varepsilon)$-holomorphic functions, we show that the same rates can be achieved (up to a small polylogarithmic or algebraic factor) for tensor-product Jacobi measures. Our reconstruction maps are based on least squares and compressed sensing procedures using the corresponding orthonormal Jacobi polynomials. In doing so, we strengthen and generalize past work that has derived weaker nonuniform guarantees for the uniform and Chebyshev measures (and corresponding polynomials) only. We also extend various best $s$-term polynomial approximation error bounds to arbitrary Jacobi polynomial expansions. Overall, we demonstrate that i.i.d.\ pointwise samples drawn from an underlying probability measure are near-optimal for the recovery of infinite-dimensional, holomorphic functions.
\end{abstract}


\pbk

\pbk
\textbf{Corresponding author:} {\tt smoragas@sfu.ca}

\section{Introduction}

In many applications in biology, chemistry, economics, engineering and elsewhere,  a mathematical model is represented as the solution of a (system of) parametric Differential Equations (DEs), where the input variable $\bm{y} \in \cU$  belongs to a subset of $\bbR^{\bbN}$ instead of a finite-dimensional space and the output is an element of an infinite-dimensional function space, e.g.,  a Hilbert or Banach space $\cV$. Notably, this is the case in parametric DEs when infinite-dimensional random fields are represented by their Karhunen--Lo\`eve expansions \cite{beck2012optimal,cohen2010convergence,schwab2006karhunen,hansen2013analytic}.

Working directly with such a model is often infeasible, since each model evaluation requires a costly DE solve \cite{sullivan2015introduction,gunzburger2014stochastic}. Therefore, a standard approach is to construct \textit{surrogate model}, i.e., an approximation $\hat{f}$ to the true target function $f: \bm{y} \in \cU \mapsto f(\bm{y})\in \cV$. Since it is often undesirable to alter the existing numerical DE code, one typically strives to do this via a \textit{nonintrusive} method, where the approximation $\hat{f}$ is computed from a collection of pointwise samples of $f$ (also known as \textit{snapshots}). However, this is a challenging task. The number of samples available to produce such approximations is often highly limited, and since the underlying function is infinite dimensional, the so-called \textit{curse of dimensionality} is an ever-present issue.

Fortunately, solution maps $\bm{y} \mapsto f(\bm{y})$ arising from parametric DEs are often smooth functions. A long series of works have shown that solution maps of various different parametric DEs belong to the class of \textit{$(\bm{b},\varepsilon)$-holomorphic} functions (see \cite{cohen2015approximation} and \cite[Chpt.\ 4]{adcock2022sparse} for reviews). Here $\bm{b} \in [0,\infty)^{\bbN}$ is a nonnegative sequence, whose $i$th entry is associated with the smoothness in the $i$th variable, and $\varepsilon > 0$ is a scalar. Because of the relevance of this class to surrogate models, it is crucial to understand the fundamental limits of how well one can approximate such functions from finitely-many samples. In our previous work \cite{adcock2024optimal}, we derived lower bounds for the (adaptive) $m$-widths for classes of $(\bm{b},\varepsilon)$-holomorphic functions. In other words, we examined how well one can approximate such functions using $m$ arbitrary (adaptive) linear measurements and arbitrary reconstruction maps. Our results considered the case where $\bm{b}$ (or its monotone majorant) is $\ell^p$-summable for some $0 < p < 1$, this being a standard assumption in the above literature. In particular, we showed that, when the $m$-width decays, it does so with rate no faster than the algebraic power $m^{1/2 - 1/p}$.

Unfortunately, \cite{adcock2024optimal}  does not directly apply to the surrogate model construction problem, where the measurements are pointwise samples of the target function.
The goal of this paper is to continue the investigations of \cite{adcock2024optimal} by focusing on the case of pointwise samples (also known as \textit{standard information} \cite{novak2008trac,novak2010trac}). We show that these lower bounds can be achieved (up to either logarithmic or arbitrarily small algebraic factors) from i.i.d.\ pointwise samples (also known as \textit{random information}) drawn from arbitrary Jacobi measures. In other words, random information constitutes nearly optimal information for recovering $(\bm{b},\varepsilon)$-holomorphic functions. We also derive explicit reconstruction procedures that attain these rates, thus demonstrating practical schemes for surrogate model construction from pointwise samples.

\subsection{Contributions and related work}\label{S:contr}

In this work, we study the approximation of $(\bm{b},\varepsilon)$-holomorphic functions $f: \cU \rightarrow \cV$, where $\cU=[-1,1]^{\bbN}$ and $\cV$ is a Hilbert space. See \S \ref{S:spaces} for the full definition. 
As discussed in \cite{adcock2025near,adcock2022sparse,adcock2024optimal}, $\bm{b}$ controls the \textit{anisotropy} of such a function. The size of the entries of $\bm{b}$ relate to the smoothness of $f$ with respect to each of its variables. Following our previous work \cite{adcock2024optimal}, we consider two scenarios. First, the \textit{known anisotropy} (known $\bm{b}$) case, where the reconstruction map is allowed to depend on $\bm{b}$. Second, the more challenging \textit{unknown anisotropy} (unknown $\bm{b}$) case, where the reconstruction map is required to be independent of $\bm{b}$. For samples, we focus on pointwise samples that are drawn i.i.d.\ from tensor-product Jacobi probability measures. Besides a mild boundedness assumption (see Assumption \ref{main-ass}), these can be arbitrary.

In each of the known and unknown anisotropy cases, we derive reconstruction maps that obtain algebraic convergence rates
of the form 
\begin{equation*}
\cO((m/\text{polylog}(m))^{1/2-1/p})
\end{equation*}
in expectation and probability, uniformly for functions in the given class, for any fixed $\bm{b} \in \ell^p(\bbN)$ (known anisotropy) or $\bm{b} \in \ell^p_{\mathsf{M}}(\bbN)$ (unknown anisotropy). Here $\ell^p_{\mathsf{M}}(\bbN)$ is the \textit{monotone} $\ell^p$ space, which consists of sequences whose minimal monotone majorants are $\ell^p$-summable (see \S \ref{S:spaces}). We also show that these sampling-reconstruction procedures lead to algebraic convergence rates that are uniform in $\bm{b}$ of the form
\begin{equation*}
\cO(m^{1/2-1/t})
\end{equation*}
for any $p < t < 1$, with a constant depending on $t$. These upper bounds are close to the lower bounds for the $m$-widths given in our previous work \cite{adcock2024optimal} (see also \S \ref{S:anisotropy}). Therefore, we conclude that near-optimal approximation rates can be obtained from i.i.d.\ pointwise samples, and provide explicit methods for doing so.

Recently, various works have studied the approximation of $(\bm{b},\varepsilon)$-holomorphic functions from pointwise samples, with similar algebraic rates being established for polynomial-based methods or, recently, deep neural networks \cite{adcock2025near,adcock2024efficient,adcock2021deep,adcock2022sparse,herrman2022constructive,dung2022collocation,herrmann2024neural}. In particular, \cite{adcock2022sparse} describes approximation of such functions from pointwise samples in the scalar-valued case via Chebyshev and Legendre polynomials using least-squares and compressed sensing techniques. The work \cite{adcock2024efficient} extended this to the case of Hilbert-valued functions and showed that the resulting compressed sensing methods can be efficiently implemented in subexponential time. In \cite{adcock2021deep,adcock2025near}, by emulating polynomial-based techniques as deep neural network training procedures, it was shown that the same rates approximation rates can be achieved using deep learning. Moreover, \cite{adcock2025near} also extended these previous results from Hilbert- to Banach-valued functions.
The main distinguishing features in this work are:
\begin{enumerate}[(i)]
\item We consider (essentially) arbitrary Jacobi measures, while most previous works consider only specific measures (e.g., the uniform or Chebyshev/arcsine measure).

\item In doing so, we provide generalizations of various best polynomial approximation error decay rates for $(\bm{b},\varepsilon)$-holomorphic functions from tensor-product Chebyshev or Legendre polynomials considered in previous works to arbitrary tensor-product Jacobi polynomials.

\item We derive error bounds in probability and in expectation.

\item We establish \textit{uniform} upper bounds, i.e., bounds that hold for any function belonging to the given function class.

\item Since our reconstruction maps are based on least squares (known anisotropy) or compressed sensing (unknown anisotropy), we extend these procedures and their analysis from tensor-product Chebyshev or Legendre polynomials considered in previous works to arbitrary tensor-product Jacobi polynomials.

\item We combine our results with the lower bounds of \cite{adcock2024optimal} to assert  near-optimality of the corresponding sampling-recovery procedures.
\end{enumerate} 
More precisely, in the known anisotropy case, we combine three main arguments. First, standard least-squares analysis \cite{cohen2013stability}. Second, new bounds for certain weighted best approximation errors for Jacobi polynomial expansions. Third, a particular nested index set construction based on ideas of \cite{krieg2021function,krieg2021functionII} to achieve the desired rate.
In the unknown anisotropy case, we construct polynomial approximations via compressed sensing and weighted $\ell^1$-minimization. Our technical arguments are based on \cite{adcock2024efficient}. However, we modify these to derive stronger uniform error guarantees, and extend them from tensor-product Chebyshev and Legendre polynomials to arbitrary tensor-product Jacobi polynomials and i.i.d.\ sampling from the corresponding Jacobi measures. While Chebyshev and Legendre polynomials are, arguably, the most commonly used in practice, the generalization to Jacobi polynomials is not only mathematically interesting, it also has practical relevance to computational uncertainty quantification where the parameters follow beta distributions \cite{smith2013uncertainty}.

\subsection{Outline}
The remainder of this paper proceeds as follows. We commence in \S\ref{S:prelim} with various preliminaries. Next, in \S\ref{S:anisotropy} we recap the relevant prior results from \cite{adcock2024optimal} and in \S\ref{S:main_results} we present our main results. In \S\ref{S:Jacobi_pol} we provide several novel best approximation error decay rates for tensor-product Jacobi polynomials. Then in \S \ref{S:proofs_prob}--\ref{S:proofs_main} we prove our main results.  We also include three appendices, which provide additional technical results used in the paper and a table of notation.

\section{Preliminaries}\label{S:prelim}

This section introduces further notation, definitions and other preliminaries that are needed later.  

\subsection{Notation}
Let $\bbN$ and $\bbN_0$ be the sets of positive and nonnegative integers, respectively. We write $[N]= \lbrace 1,2,\ldots, N \rbrace$ for $N \in \bbN \cup \{ \infty \}$, with the convention that $[N]=\bbN$ if $N=\infty$. Let $\bbR^N$ be the vector space of real vectors of length $N \in \bbN$ and $\bbR^{\bbN}$ be the vector space of real sequences indexed over $\bbN$. In either space, we write $\bm{e}_j=(\delta_{j,k})_{k \in [N]}$ for the standard basis vectors, where $j \in \bbN$ or $j \in [N]$. 

For $N \in \bbN \cup \{ \infty \}$, we write $\bm{0} \in \bbN^N_0$ and $\bm{1} \in \bbN^{N}_{0}$ for the multi-indices consisting of all zeros and all  ones, respectively. For any multi-indices $\bm{\nu}=(\nu_i)_{i \in [N]}$ and $\bm{\mu}=(\mu_i)_{i \in [N]}$, the inequality $\bm{\nu} \leq \bm{\mu}$ (or $\bm{\nu} < \bm{\mu}$) is understood componentwise, i.e., $\bm{\nu} \leq \bm{\mu}$  if ${\nu_j} \leq {\mu_j}$ for all $j \in [N]$.   We also write 
\begin{equation*}
 \bnu^{\bm{\mu}} = \prod_{i \in [N]} \nu_i^{\mu_i}, \quad \text{ and } \quad  \bnu ! = \prod_{i \in [N]} \nu_i!,
 \end{equation*} 
with the convention that $0^0=1$.

For $1 \leq p \leq \infty$ $(0<p<1)$, we write $\nms{\cdot}_{p}$ for the usual vector $\ell^p$-norm (quasi-norm) on $\bbR^{\bbN}$,
and $\ell^p(\bbN)$ for the space of sequences for which $\|\bm{c}\|_{p}<\infty$ (see \S \ref{S:ellspaces} for a more general definition).
Given a set $S \subseteq [N]$ and $\bm{c} = (c_i)_{i \in [N]} \in \bbR^{N}$, we write $\bm{c}_S$ for the vector with $i$th entry equal to $c_i$ if $i \in S$ and zero otherwise. 
We also write  $S^c$ for the set complement $[N] \setminus S$ of $S$. 

Let $\bm{y} \in \cU$ denote the variable in $\cU = [-1,1]^{\bbN}$ and $\varrho$ be a probability measure on $\cU$. Let $(\cV,\ip{\cdot}{\cdot}_{\cV})$ be a Hilbert space with induced norm $\nms{\cdot}_{\cV}$. 
For $1 \leq p \leq \infty$, we define the weighted Lebesgue-Bochner space $L^p_{\varrho}(\cU;\cV)$ as the space consisting of (equivalence classes of) strongly $\varrho$-measurable functions $f: \cU \rightarrow \cV$ for which $\|f\|_{L^p_{\varrho}(\cU;\cV)}<\infty$, where
 \begin{equation}\label{def:normF}
\| f \|_{L^p_{\varrho}(\cU;\cV)} : = 
\begin{cases} 
\left( \int_{\cU} \nm{f( \bm{y} )}_{\cV}^p \D \varrho (\bm{y}) \right)^{1/p} & 1 \leq p < \infty ,
\\
\mathrm{ess} \sup_{\bm{y} \in \cU} \nm{f(\bm{y})}_{\cV}  & p = \infty.
\end{cases}
\end{equation}
For simplicity we write $L^{p}_{\varrho}(\cU)$ when $\cV=\bbR$. 

Finally, we write $C(\cdot)$ for a constant that only depends on its arguments, and which is allowed to vary from one line to the next. On occasion, we write $\widetilde{C}$ or $C_1$, $C_2$ and so forth to distinguish specific constants. We also write lower-case $c$ for a universal constant (i.e., a constant that is independent of any variables).

\subsection{The space of $(\pmb{b},\varepsilon)$-holomorphic functions} \label{S:spaces}

We consider the approximation of infinite-dimensional, Hilbert-valued functions of the form
\begin{equation}\label{f-def}
f : \cU  \rightarrow \cV,\ \bm{y} \mapsto f(\bm{y}).
\end{equation}
We assume such functions are holomorphic, in a sense we now make precise.
For a parameter  $\bm{\rho} = (\rho_j)_{j \in \bbN} \geq \bm{1}$ we define the Bernstein  polyellipse as
\begin{equation*}
\cE_{\bm{\rho}} = \cE_{\rho_1} \times \cE_{\rho_2} \times \cdots \subset \bbC^{\bbN},
\end{equation*}
where, for $\rho > 1$, $\cE_{\rho} = \{ \frac12 (z+z^{-1}) : z \in \bbC,\ 1 \leq | z | \leq \rho \} \subset \bbC$ is the classical Bernstein ellipse with foci $\pm 1$ and, by convention, $\cE_{\rho} = [-1,1]$ when $\rho = 1$.  
Now let $\bm{b} = (b_i)_{i\in\bbN} \in [0,\infty)^{\bbN} $ and $\varepsilon > 0$. We say that \ef{f-def} is \textit{$(\bm{b},\varepsilon)$-holomorphic} if it is holomorphic in the region
\begin{equation}
\label{def:b-eps-holo}
\cR(\bm{b},\varepsilon) = \bigcup \left\lbrace  \cE_{\bm{\rho}} : \bm{\rho} \in [1,\infty)^{\bbN},  \sum^{\infty}_{i=1} \left ( \frac{\rho_i + \rho^{-1}_i}{2} - 1 \right ) b_i \leq \varepsilon  \right\rbrace \subset \bbC^{\bbN}.
\end{equation}
As mentioned in \cite{adcock2024optimal}, by rescaling $\bm{b}$ we may choose  $\varepsilon=1$.
Thus, in this paper we consider $(\bm{b},1)$-holomorphic   functions, write $\cR(\bm{b})$ instead of $\cR(\bm{b},1)$ and define the set
\begin{equation}
\cH(\bm{b}) = \left \{ f : \cU \rightarrow \cV\text{ $(\bm{b},1)$-holomorphic} : \nm{f}_{L^{\infty}(\cR(\bm{b});\cV)} : = \esssup_{\bm{z} \in \cR(\bm{b})} \nm{f(\bm{z})}_{\cV}  \leq 1 \right \}.
\end{equation}
 See \cite{adcock2022sparse,chkifa2015breaking,schwab2019deep} for further discussion on 
$(\bm{b},\varepsilon)$-holomorphic functions.

We define the \textit{minimal monotone majorant} of a sequence $\bm{z}=(z_i)_{i \in \bbN} \in \ell^{\infty}(\bbN)$ as
\begin{equation}
\label{min-mon-maj}
\tilde{\bm{z}} = (\tilde{z}_i)_{i \in \bbN},\quad \text{where }
\tilde{z}_i = \sup_{j \geq i} | z_{j}|,\ \forall i \in \bbN.
\end{equation}
Given $0 < p < \infty$, we define the \textit{monotone $\ell^p$-space} $\ell^p_{\mathsf{M}}(\bbN)$ by
\begin{equation*}
\ell^p_{\mathsf{M}}(\bbN) = \{ \bm{z} \in \ell^{\infty}(\bbN) : \nm{\bm{z}}_{p,\mathsf{M}} : = \nm{\tilde{\bm{z}}}_{p} < \infty  \}.
\end{equation*}
For reasons we explain in more detail in the next section, we are particularly interested in the case where $\bm{b} \in \ell^p(\bbN)$ or $\bm{b} \in \ell^p_{\mathsf{M}}(\bbN)$ for some $0 < p < 1$. 
Accordingly, we also define the function classes
\begin{equation*}
\cH(p) =  \bigcup \left \{ \cH(\bm{b}) : \bm{b} \in \ell^p(\bbN),\ \bm{b} \in [0,\infty)^{\bbN},\   \nm{\bm{b}}_p \leq 1 \right \}
\end{equation*}
and
\begin{equation*}
\cH(p,{\mathsf{M}})  = \bigcup  \left \{ \cH(\bm{b}) : \bm{b} \in \ell^p_{\mathsf{M}}(\bbN),\ \bm{b} \in [0,\infty)^{\bbN},\  \nm{\bm{b}}_{p,\mathsf{M}} \leq 1 \right \}  .
\end{equation*}
Note that it is a short argument to show that $\cH(p,\mathsf{M})$ is, in fact, equivalent to the union over all $\bm{b} \in \ell^p(\bbN)$ that are monotonically decreasing.

\section{ $m$-widths of $(\text{\textit{\textbf{b}}},1)$-holomorphic functions: the results of \cite{adcock2024optimal}}\label{S:anisotropy}

Let $\cK$ be a set of (holomorphic) functions. To understand how well we can approximate functions in $\cK$ from finitely-many samples, we use the theory of $m$-widths (see, e.g., \cite{pinkus1968n-widths}). Recall that $\cV$  is a  Hilbert space, $\varrho$ is a tensor-product probability measure on $\cU$  and consider  the Lebesgue--Bochner space $\cX = L^2_{\varrho}(\cU ; \cV)$. Let $\cY = C(\cU;  \cV)$ be the normed vector space of continuous functions from $\cU$ to $\cV$ with respect to the uniform norm. Then, following \cite{adcock2024optimal}, we define the \textit{(adaptive) $m$-width} 
\begin{equation}
\label{Theta_m-def}
\Theta_m(\cK ; \cY,\cX) =  \inf \left \{ \sup_{f \in \cK} \nm{f - \cT ( \cL(f)) }_{\cX} : \cL : \cY \rightarrow \cV^m\text{ adaptive},\ \cT : \cV^m \rightarrow \cX \right \}.
\end{equation}
Here $\cL$ is an \textit{adaptive sampling operator} (note that this definition includes bounded linear operators) and $\cT : \cV^m \rightarrow \cX$ is an arbitrary (potentially nonlinear) reconstruction map.

Now consider  $(\bm{b},\varepsilon)$-holomorphic functions.  
In \cite{adcock2024optimal}, we made the distinction between the \textit{known anisotropy} (known $\bm{b}$) and \textit{unknown anisotropy} (unknown $\bm{b}$) cases. In the known anisotropy case, we now let $\cK = \cH(\bm{b})$ in \ef{Theta_m-def} for some fixed $\bm{b}$, so that the reconstruction map $\cT$ can depend on $\bm{b}$. Conversely, in the case of unknown anisotropy we let $\cK = \cH(p)$ or $\cK = \cH(p,\mathsf{M})$ in \ef{Theta_m-def}, in which case the reconstruction map is prohibited from depending on $\bm{b}$.

Keeping this in mind, for succinctness, when $\bm{b}$ is known, we write
\begin{equation}
\label{theta-m-b}
 \theta_m(\bm{b}) = \Theta_m(\cH(\bm{b}) ; \cY,\cX),
\end{equation}
and, given $0 < p < 1$, 
 \begin{equation}\label{eq:def_theta} 
\begin{split}
\overline{\theta_m}(p) &= \sup \left \{ \theta_m(\bm{b}) : \bm{b} \in \ell^p(\bbN),\ \bm{b} \in [0,\infty)^{\bbN},\  \nm{\bm{b}}_p \leq 1 \right \},
\\
\overline{\theta_m}(p,{\mathsf{M}}) &= \sup \left \{ \theta_m(\bm{b}) : \bm{b} \in \ell^p_{\mathsf{M}}(\bbN),\ \bm{b} \in [0,\infty)^{\bbN},\  \nm{\bm{b}}_{p,\mathsf{M}} \leq 1 \right \}.
\end{split}
\end{equation} 
In the unknown anisotropy case, we write
 \begin{equation}\label{theta-upsilon-unknown-aniso}
\theta_m(p) = \Theta_m(\cH(p) ; \cY,\cX), \quad 
\theta_m(p,{\mathsf{M}}) = \Theta_m(\cH(p,{\mathsf{M}}) ; \cY,\cX ).
 \end{equation}
 We now recap the main results of \cite{adcock2024optimal}. These results only considered the case where the probability measure $\varrho = \varrho_1 \times \varrho_1 \times \cdots$ is a tensor-product of a single probability measure on $[-1,1]$ (the existence of this measure is guaranteed by the Kolmogorov extension theorem \cite[Thm.\ 14.36]{klenke2013probability}).
However, they can be easily generalized to tensor-products of different one-dimensional measures, subject to the following assumption (see Remark \ref{rem:how-to-extend}).

\begin{assumption}
\label{measure-ass}
The probability measure $\varrho$ on $\cU$ has the form $\varrho = \varrho_1 \times \varrho_2 \times \cdots$, where each $\varrho_i$ is a probability measure on $[-1,1]$. Moreover, $\inf_{i} \sigma_i > 0$, where the $\sigma_i$ are given by
 \bes{
\sigma_i = \sqrt{\int^{1}_{-1} (y-\upsilon_i)^2 \D \varrho_i(y) },\quad \text{for } \upsilon_i = \int^{1}_{-1} y \D \varrho_i(y).
}
\end{assumption} 
 
With this in hand, the main results \cite[Thms.\ 4.4 \& 4.5]{adcock2024optimal} can be summarized as follows.
\begin{theorem}[Lower bounds]\label{t:partI}
Let $m\geq 1$, $\varrho$ satisfy Assumption \ref{measure-ass} and $0<p<1$. Then the following hold.
\begin{itemize}
\item  The $m$-widths  $\overline{\theta_m}(p)$, $\overline{\theta_m}(p,\mathsf{M})$ in \eqref{eq:def_theta} satisfy
\begin{equation}\label{known-lower}
\overline{\theta_m}(p) \geq \overline{\theta_m}(p,\mathsf{M}) \geq C(\varrho) \cdot 2^{-\frac1p} \cdot m^{\frac12-\frac1p}.
\end{equation}

\item  The $m$-width $\theta_m(p)$ in \eqref{theta-upsilon-unknown-aniso} satisfies
\begin{equation} 
\label{unknown-lower-1}
\theta_m(p) \geq C(\varrho) \cdot 2^{\frac12-\frac2p}.  
\end{equation}
\item The $m$-width $\theta_m(p,\mathsf{M})$ in \eqref{theta-upsilon-unknown-aniso} satisfies
\begin{equation}
\label{unknown-lower-2}
\theta_m(p,\mathsf{M}) \geq \overline{\theta_m}(p,\mathsf{M}) \geq C(\varrho) \cdot 2^{-\frac1p} \cdot m^{\frac12-\frac1p}.
\end{equation}
\end{itemize}
\end{theorem}

The bound \ef{known-lower} pertains to the known anisotropy case. It shows that no sampling-recovery pair can achieve a rate better than $m^{1/2-1/p}$ for $\bm{b} \in \ell^p(\bbN)$ or $\bm{b} \in \ell^p_{\mathsf{M}}(\bbN)$ when the reconstruction map is allowed to depend on $\bm{b}$. For the unknown anisotropy case, \ef{unknown-lower-1}  shows that it is impossible to approximate functions in $\cH(p)$ from finitely-many samples. This arises because functions in $\cH(p)$ can have important variables in infinitely-many different configurations \cite{adcock2024optimal}.
 The final result \ef{unknown-lower-2} asserts that when restricting to the class $\cH(p,\mathsf{M})$ one can achieve a rate of at best $m^{1/2-1/p}$.
 
We remark that \cite[Thms.~4.6 \& 4.7]{adcock2024optimal} also establish upper bounds for several of the $m$-widths considered above when $\varrho$ is the uniform measure. However, these bounds do not consider i.i.d. pointwise samples and are therefore not relevant to this work. In the next section, we present counterpart upper bounds based on such samples. In particular, we show that the rates $m^{1/2-1/p}$ can be `nearly' achieved from i.i.d.\ pointwise samples drawn from general tensor-product Jacobi measures, using suitable recovery maps based on the corresponding tensor Jacobi polynomials. 

\rem{
\label{rem:how-to-extend}
The extension of the results of \cite{adcock2024optimal} to the case where $\varrho$ satisfies Assumption \ref{measure-ass} involves a straightforward modification of several proofs. First, one defines the orthonormal (in $L^2_{\varrho}(\cU)$) functions $\psi_i(\bm{y}) = \frac{y_i - \upsilon_i}{\sigma_i}$ and then observes that \cite[Lem.\ 5.2]{adcock2024optimal} remains valid with $\sigma = \inf_i \sigma_i$ and $\tau = \sup_i |\upsilon_i|$. Observe that $|\tau| \leq 1$ since each $\varrho_i$ is a probability measure supported on $[-1,1]$. Lemmas 5.3 and 5.4 of \cite{adcock2024optimal} then extend without further modifications and with these values of $\sigma$ and $\tau$. Finally, the same applies to Theorems 4.4 and 4.5 of \cite{adcock2024optimal}, on which \cf{t:partI} is based.
}

\rem{
\label{rem:theta_bar_equate}
In general, $\overline{\theta_m}(p) \neq \overline{\theta_m}(p,\mathsf{M})$. However, $\overline{\theta_m}(p) = \overline{\theta_m}(p,\mathsf{M})$ whenever the measure $\varrho = \varrho_1 \times \varrho_1 \times \cdots $ is a tensor-product of a single measure $\varrho_1$ on $[-1,1]$. This follows from two observations. First, for such a measure, we have $\theta_m(\bm{b}) = \theta_{m}(\pi(\bm{b}))$ for any permutation $\pi$. Second, if $\bm{b} \in \ell^p(\bbN)$ with $\nm{\bm{b}}_p \leq 1$, then there is a permutation $\pi$ for which $\pi(\bm{b})$ is nonincreasing, and therefore $\pi(\bm{b}) \in \ell^p_{\mathsf{M}}(\bbN)$ with $\nm{\pi(\bm{b})}_{p,\mathsf{M}} = \nm{\bm{b}}_p$. See Appendix \ref{app:theta_bar_m} for the proof of this claim, and the exact relationship between $\theta_m(\bm{b})$ and $\theta_m(\pi(\bm{b}))$ in the general case.
}

\section{Main results}\label{S:main_results}
In this section, we present our main results. In order to do this, we first introduce (tensor) Jacobi measures and polynomials, and then a key technical assumption.

\subsection{Jacobi measures}

For $\alpha,\beta > - 1$  the \textit{Jacobi measure} on $[-1,1]$ is given by
\be{
\label{jacobi-meas}
\D \omega_{\alpha,\beta}(y) = (1-y)^{\alpha} (1+y)^{\beta} \D y,
}
{where $\omega_{\alpha,\beta}(y)$ is the Jacobi weight function} \cite[Chpt.\ 4]{szego1975orthogonal}.
We define the corresponding probability measure as
\be{
\label{1d-jacobi-meas}
\varrho_{\alpha,\beta} = \frac{\omega_{\alpha,\beta}}{\int^{1}_{-1} \D \omega_{\alpha,\beta}(y)}.
}
Now let $\bm{\alpha} = (\alpha_k)_{k \in \bbN}$ and $\bm{\beta} = (\beta_k)_{k \in \bbN}$ with $\bm{\alpha} > - \bm{1}$ and $\bm{\beta} > - \bm{1}$. Then we define the associated \textit{tensor-product Jacobi probability measure} $\varrho_{\bm{\alpha},\bm{\beta}}$ on $\cU = [-1,1]^{\bbN}$ as the infinite tensor product
\be{
\label{tensor-jacobi-measure}
\varrho = \varrho_{\bm{\alpha},\bm{\beta}} = \varrho_{\alpha_1,\beta_1} \times \varrho_{\alpha_2,\beta_2} \times \cdots.
}
Our main results do not allow for arbitrary sequences $\bm{\alpha}$ and $\bm{\beta}$. In particular, we make the following assumption, which asserts that the sequences are uniformly bounded above and below.
\begin{assumption}
\label{main-ass}
There is a $\tau > 0$ such that $\tau-1 \leq \alpha_j , \beta_j \leq 1/\tau$ for all $j \in \bbN$.
\end{assumption}

It is a short argument to show that Assumption \ref{main-ass} implies Assumption \ref{measure-ass} for the measure \ef{tensor-jacobi-measure}. This follows immediately, after noticing that $\sigma_i$ is a continuous function of $\alpha_i,\beta_i$ in this case. Therefore, the lower bounds of \cf{t:partI} apply in this setting.

\subsection{Upper bounds in expectation}

We now state our main results, starting with upper bounds in expectation. In these results, we consider reconstruction maps of the form
\begin{equation*}
\cT : (\cU \times \cV)^m \rightarrow L^2_{\varrho}(\cU ; \cV).
\end{equation*}
Note that we slightly change the definition of $\cT$ from that of \ef{Theta_m-def}, so that it now depends explicitly on the sample points. In particular, given sample points $\bm{y}_1,\ldots,\bm{y}_m \in \cU$, the approximation of a function $f$ is defined as $\cT ( \{ (\bm{y}_i , f(\bm{y}_i) ) \}^{m}_{i=1} )$.
For succinctness, we do not describe the specific recovery maps we consider in these theorems. As noted, both rely on Jacobi polynomials. For the known anisotropy case, we use least squares and for the unknown anisotropy case, we use compressed sensing. See \S \ref{S:setup} and \S \ref{S:reconstruction} for the respective constructions.

\begin{theorem}[Known anisotropy]
\label{t:main-res-known}
Let   $\bm{\alpha},\bm{\beta}$ satisfy Assumption \ref{main-ass} with constant $\tau > 0$ and $\varrho = \varrho_{\bm{\alpha},\bm{\beta}}$ be the tensor-product Jacobi probability measure as in \ef{tensor-jacobi-measure}. Let $m \geq 3$, $0 < p < 1$, $\bm{b} = (b_i)_{i \in \bbN} \in \ell^p(\bbN)$ with $\bm{b} \geq \bm{0}$ and $\pi : \bbN \rightarrow \bbN$ be a bijection that gives a nonincreasing rearrangement of $\bm{b}$, i.e., $b_{\pi(1)} \geq b_{\pi(2)} \geq \cdots$. Consider sample points $\bm{y}_1,\ldots,\bm{y}_{m} \sim_{\mathrm{i.i.d.}} \varrho$.
Then there is a reconstruction map $\cT : (\cU \times  \cV)^m  \rightarrow L^2_{\varrho}(\cU ; \cV)$ depending on $\bm{b}$, $p$, $\bm{\alpha}$ and $\bm{\beta}$ only such that
\begin{equation*}
\theta_m(\bm{b}) \leq  \bbE \left ( \sup_{f \in \cH(\bm{b})} \nm{f - \cT( \{ (\bm{y}_i , f(\bm{y}_i)) \}^{m}_{i=1} )}_{L^2_{\varrho}(\cU ; \cV)} \right ) \leq C(\bm{b}_{\pi},p,\tau) \cdot \left ( \frac{m}{\log(m) } \right )^{\frac12-\frac1p},
\end{equation*}
where $\bm{b}_{\pi} = (b_{\pi(i)})_{i \in \bbN}$
 and $\theta_{m}(\bm{b})$ is as in \ef{theta-m-b}. 
Moreover, for any $t \in (p,1)$,
\begin{equation*}
\overline{\theta_m}(p) \leq \sup_{\substack{\bm{b} \in \ell^p(\bbN), \bm{b} \geq \bm{0} \\ \nm{\bm{b}}_{p} \leq 1}} \bbE \left ( \sup_{f \in \cH(\bm{b})} \nm{f - \cT( \{ (\bm{y}_i , f(\bm{y}_i)) \}^{m}_{i=1} )}_{L^2_{\varrho}(\cU ; \cV)} \right )  \leq C(p,t,\tau) \cdot m^{\frac12-\frac1t},
\end{equation*}
where $\overline{\theta_m}(p)$ is as in \eqref{eq:def_theta}. 
\end{theorem}

Due to the considerations discussed in Remark \ref{rem:theta_bar_equate}, a crucial component of this theorem is that the constant in the first bound depends on the monotone sequence $\bm{b}_{\pi}$ and not $\bm{b}$. This is what allows us to derive the bound for $\overline{\theta_m}(p)$ in the second part of the theorem.

Theorem \ref{t:main-res-known} is suitable for the known anisotropy case, since the reconstruction map depends on $\bm{b}$. In our next result, we derive a reconstruction map that is independent of both $\bm{b}$ and $p$, which makes it suitable for the unknown anisotropy setting.

\begin{theorem}[Unknown anisotropy]
\label{t:main-res-unknown}
Let  $\bm{\alpha},\bm{\beta}$ satisfy Assumption \ref{main-ass} with constant $\tau > 0$ and $\varrho = \varrho_{\bm{\alpha},\bm{\beta}}$ be the tensor-product Jacobi probability measure as in \ef{tensor-jacobi-measure}.  Let $m \geq 3$ and consider sample points $\bm{y}_1,\ldots,\bm{y}_{m} \sim_{\mathrm{i.i.d.}} \varrho$. Then there is a reconstruction map $\cT : (\cU \times  \cV)^m  \rightarrow L^2_{\varrho}(\cU ; \cV)$ depending on $\bm{\alpha}$ and $\bm{\beta}$ only such that
\begin{equation*}
\theta_m(\bm{b}) \leq  \bbE \left ( \sup_{f \in \cH(\bm{b})} \nm{f - \cT( \{ (\bm{y}_i , f(\bm{y}_i)) \}^{m}_{i=1} )}_{L^2_{\varrho}(\cU ; \cV)} \right ) \leq C(\bm{b},p,\tau) \cdot \left ( \frac{m}{\log^{5}(m) } \right )^{\frac12-\frac1p},
\end{equation*}
for all $0 < p <1$ and $\bm{b} \in \ell^p_{\mathsf{M}}(\bbN)$ with $\bm{b} \geq \bm{0}$, where $\theta_{m}(\bm{b})$ is as in \ef{theta-m-b}.
Moreover, for any $0 < p < 1$ and $t \in (p,1)$,
\bes{
\theta_m(p,\mathsf{M}) \leq \bbE \left ( \sup_{f \in \cH(p,\mathsf{M})} \nm{f - \cT( \{ (\bm{y}_i , f(\bm{y}_i)) \}^{m}_{i=1} )}_{L^2_{\varrho}(\cU ; \cV)} \right )  \leq C(p,t,\tau) \cdot m^{\frac12-\frac1t},
}
where $\theta_m(p,\mathsf{M})$  is as in \eqref{theta-upsilon-unknown-aniso}.
\end{theorem}

Note that, in both Theorems \ref{t:main-res-known} and \ref{t:main-res-unknown}, the leftmost inequalities (namely, $\theta_m(\bm{b}) \leq \ldots$ and so forth) follow immediately from the definition of the relevant quantities, i.e., \ef{theta-m-b}--\ef{theta-upsilon-unknown-aniso}.

\subsection{Discussion}

We now discuss some important aspects of Theorems \ref{t:main-res-known} and \ref{t:main-res-unknown}.

Theorem~\ref{t:partI} provides lower bounds for the various $m$-widths, and establishes that the rate $m^{1/2-1/p}$ is the best possible. These lower bounds are valid for Banach-valued functions, whereas our main results hold only when $\cV$ is a Hilbert space. This gap was previously encountered in \cite{adcock2025near}. There, (nonuniform) upper bounds for Legendre and Chebyshev polynomial approximations were shown for Banach-valued functions. However, the algebraic rates of convergence shown therein are slower than those of the Hilbert-valued case. A similar extension of Theorems \ref{t:main-res-known} and \ref{t:main-res-unknown} to Banach spaces would also result in similar reductions of the rates. As of now, it is unknown how to attain upper bounds for the $m$-width for Banach-valued functions that match the lower bounds.

In the second parts of Theorems \ref{t:main-res-known} and \ref{t:main-res-unknown}, we trade the (poly)logarithmic factor from the first part for an (arbitrarily small) reduction in the algebraic rate in order to obtain an upper bound for the corresponding $m$-width. This is done because it is not possible, with our current proof techniques, to bound the constant $C(\bm{b}_{\pi},p,\tau)$ (respectively, $C(\bm{b},p,\tau)$) uniformly over sequences $\bm{b} \in \ell^p(\bbN)$ with $\nm{\bm{b}}_{p} \leq 1$ (respectively, $\bm{b} \in \ell^p_{\mathsf{M}}(\bbN)$ with $\nm{\bm{b}}_{p,\mathsf{M}} \leq 1$).
Instead, we use an idea from \cite{adcock2024optimal} which allows us to bound $\sup_{\nm{\bm{b}}_{p} \leq 1} C(\bm{b}_{\pi},t,\tau) < \infty$ for $t \in (p,1)$  and likewise for $C(\bm{b},t,\tau)$. See the proofs of Theorems \ref{t:main-res-known} and \ref{t:main-res-unknown} in \S\ref{S:proofs_prob} for further details.

An upper bound for $\overline{\theta_m}(p,\mathsf{M})$ was previously obtained in \cite[Theorem~4.6]{adcock2024optimal} in the case of the uniform measure. This result also gave a tighter bound for $\theta_m(\bm{b})$ without the $\log(m)$ term. However, this result does not consider i.i.d. pointwise samples, nor least-squares approximation.  
Due to the coupon collector's problem, we suspect that the rate $(m/\log(m))^{1/2-1/p}$ may be the best possible for this class of functions when using i.i.d. samples. See, e.g., \cite[Rmk.~3]{dolbeault2022optimal}  for a related discussion. 
On the other hand, it may be possible to remove this log factor by considering non-i.i.d.\ pointwise samples. A recent line of research uses frame subsampling techniques to facilitate stable and accurate least-squares approximation from a linear number of samples, improving on the log-linear number of samples needed in the i.i.d.\ case. See \cite{dolbeault2024randomized,bartel2023constructive,limonova2022sampling,krieg2024random,krieg2021function,nagel2021upper,pozharska2022note,temlyakov2021optimal} and references therein.

As noted previously, the unknown anisotropy case studied in Theorem \ref{t:main-res-unknown} requires a different approach to the least-squares approximation used in Theorem \ref{t:main-res-known}. Due to the intrinsic unknown nature of the entries of $\bm{b}$, we use techniques based on compressed sensing to establish theoretical guarantees on the recovery of certain polynomial coefficients in a large (but finite) index set. A key step involves using well-established results that assert the (weighted) \textit{Restricted Isometry Property (RIP)} for the relevant measurement matrix. The need to assert the RIP leads to a $\log^4(m)$ factor in the first bound in Theorem \ref{t:main-res-unknown}. Unfortunately, it is well-known that improving the log factors in RIP-type estimates is a challenging open problem. See \cite{brugiapaglia2021sparse} for further discussion.

\rem{
\label{rem:log-term}
The one additional power of $\log(m)$ in Theorem \ref{t:main-res-unknown}, which results in $\log^4(m) \cdot \log(m) = \log^5(m)$, stems from the desire to have a reconstruction map $\cT$ that is valid for all $0 < p < 1$.  In fact, the term
$\log^5(m)$ can be replaced by $\log^4(m) g(m)$, where $g$ is any increasing function $g : \bbN \rightarrow [1,\infty)$ with $g(m) \rightarrow +\infty$ as $m \rightarrow + \infty$ and which grows subalgebraically fast in $m$. We use $g(m) = \log(m)$ to simplify the presentation. In the general case, the constants $C$ also depend on $g$. See \S \ref{S:reconstruction} and \S \ref{S:proofs_main}.
}
While Theorem \ref{t:main-res-unknown} relies on compressed sensing, it is also possible to derive a weaker version of it by using the least-squares approximation from Theorem \ref{t:main-res-known}. Notice that any $\bm{b} \in \ell^p_{\mathsf{M}}(\bbN)$, $\bm{b} \geq \bm{0}$, $\nm{\bm{b}}_{p,\mathsf{M}} \leq 1$, satisfies $\bm{b} \leq \bm{b}_p$, where $\bm{b}_p = (i^{-1/p})_{i \in \bbN}$. It follows from \ef{def:b-eps-holo} that $\cR(\bm{b}_p,\varepsilon) \subseteq \cR(\bm{b},\varepsilon)$, which implies that $\cH(\bm{b}) \subseteq \cH(\bm{b}_p)$ and, in turn, that $\theta_m(\bm{b}) \leq \theta_m(\bm{b}_p)$. Now consider the reconstruction map $\cT = \cT_p$ from \cf{t:main-res-known} based on $\bm{b}_p$. Since $\bm{b}_p \in \ell^t(\bbN)$ for any $t > p$, the use of this map yields the bounds
\bes{
\theta_m(\bm{b}) \leq  \bbE \left ( \sup_{f \in \cH(\bm{b}_p)} \nm{f - \cT_p( \{ (\bm{y}_i , f(\bm{y}_i)) \}^{m}_{i=1} )}_{L^2_{\varrho}(\cU ; \cV)} \right ) \leq C(p,t,\tau) \cdot m^{\frac12-\frac1t},
}
and, since $\cT_p$ is independent of $\bm{b}$, 
\bes{
\theta_m(p,\mathsf{M}) \leq \bbE \left ( \sup_{f \in \cH(p,\mathsf{M})} \nm{f - \cT_p( \{ (\bm{y}_i , f(\bm{y}_i)) \}^{m}_{i=1} )}_{L^2_{\varrho}(\cU ; \cV)} \right )  \leq C(p,t,\tau) \cdot m^{\frac12-\frac1t},
}
for any $t \in (p,1)$. However, this approach has two disadvantages. First, the rate in the bound for $\theta_m(\bm{b})$ is strictly slower than the corresponding result in \cf{t:main-res-unknown}. Second, in both cases, the reconstruction map $\cT_p$ depends on $p$, whereas the map $\cT$ in \cf{t:main-res-unknown} is valid for all $0 < p < 1$.

\subsection{Upper bounds in probability}

We now state our two main `in probability' results.

\begin{theorem}[Known anisotropy]
\label{t:prob-bound-known}
Let   $\bm{\alpha},\bm{\beta}$ satisfy Assumption \ref{main-ass} with constant $\tau > 0$ and $\varrho = \varrho_{\bm{\alpha},\bm{\beta}}$ be the tensor-product Jacobi probability measure \ef{tensor-jacobi-measure}.
Let $0 < \epsilon < 1$, $m \geq 3$, $0 < p < 1$, $\bm{b} \in \ell^p(\bbN)$ with $\bm{b} \geq \bm{0}$ and $\pi : \bbN \rightarrow \bbN$ be a bijection that gives a nonincreasing rearrangement of $\bm{b}$, i.e., $b_{\pi(1)} \geq b_{\pi(2)} \geq \cdots$. Consider sample points $\bm{y}_1,\ldots,\bm{y}_{m} \sim_{\mathrm{i.i.d.}} \varrho$. Then there exists a reconstruction map $\cT : (\cU \times  \cV)^m  \rightarrow L^2_{\varrho}(\cU ; \cV)$   depending on $\bm{b}$, $\epsilon$, $\bm{\alpha}$ and $\bm{\beta}$ only such that
\bes{
\bbP \left ( \sup_{f \in \cH(\bm{b})} \nm{f - \cT( \{ (\bm{y}_i , f(\bm{y}_i) )\}^{m}_{i=1} )}_{L^2_{\varrho}(\cU ; \cV)} \leq C(\bm{b}_{\pi},p,{\tau}) \cdot \left ( \frac{m}{\log(m) + \log(\epsilon^{-1}) } \right )^{\frac12-\frac1p} \right ) \geq 1 - \epsilon,
}
where $\bm{b}_{\pi} = (b_{\pi(i)})_{i \in \bbN}$.
\end{theorem}

In this result, the reconstruction map depends on $\bm{b}$ and the probabilistic bound is for a fixed $\bm{b}$. Theorem \ref{t:prob-bound-known} is therefore suitable for the known anisotropy case. In the next result, we consider the unknown anisotropy case. There, the reconstruction map is independent of $\bm{b}$ and the probabilistic bound considers the intersection of events corresponding to all $\bm{b} \in \ell^p_{\mathsf{M}}(\bbN)$, $\bm{b} \geq \bm{0}$ and $0 < p < 1$.

\begin{theorem}[Unknown anisotropy]
\label{t:prob-bound-unknown}
Let   $\bm{\alpha},\bm{\beta}$ satisfy Assumption \ref{main-ass} with constant $\tau > 0$ and $\varrho = \varrho_{\bm{\alpha},\bm{\beta}}$ be the tensor-product Jacobi probability measure \ef{tensor-jacobi-measure}. 
Let $0 < \epsilon < 1$, $m \geq 3$ and consider sample points $\bm{y}_1,\ldots,\bm{y}_{m} \sim_{\mathrm{i.i.d.}} \varrho$.  Then there exists a reconstruction map $\cT : (\cU \times  \cV)^m  \rightarrow L^2_{\varrho}(\cU ; \cV)$ depending on $\epsilon$, $\bm{\alpha}$ and $\bm{\beta}$ only such that
\begin{align*}
\bbP \Bigg ( & \sup_{f \in \cH(\bm{b})} \nm{f - \cT( \{( \bm{y}_i , f(\bm{y}_i)) \}^{m}_{i=1} )}_{L^2_{\varrho}(\cU ; \cV)} \leq C(\bm{b},p,\tau) \cdot \left ( \frac{m}{\log^{5}(m) + \log(\epsilon^{-1}) } \right )^{\frac12-\frac1p}
\\
& \forall \bm{b} \in \ell^p_{\mathsf{M}}(\bbN),\ \bm{b} \geq \bm{0},\ 0 < p < 1 \Bigg ) \geq 1 - \epsilon.
\end{align*}
\end{theorem}

As with Theorem \ref{t:main-res-unknown}, the term $\log^{5}(m)$ could be replaced by $\log^4(m) g(m)$ for any function $g$ satisfying the conditions detailed in Remark \ref{rem:log-term}.

Theorems \ref{t:prob-bound-known} and \ref{t:prob-bound-unknown} are of independent interest, but they are also crucial in obtaining the `in expectation' bounds, Theorems \ref{t:main-res-known} and \ref{t:main-res-unknown}. Therefore, the remainder of the paper involves first establishing Theorems \ref{t:prob-bound-known} and \ref{t:prob-bound-unknown} and then using them to prove Theorems \ref{t:main-res-known} and \ref{t:main-res-unknown}.

\section{Jacobi polynomial expansions and best polynomial approximation}\label{S:Jacobi_pol}

The proofs of Theorems \ref{t:prob-bound-known} and \ref{t:prob-bound-unknown} involve various best approximation error bounds for holomorphic functions using Jacobi polynomials. Therefore, in this section we first introduce Jacobi polynomials and Jacobi polynomial expansions, then provide such bounds. Specifically, we consider the classical \textit{best $s$-term approximation} and the so-called \textit{weighted best $k$-term approximation} and \textit{best $s$-term approximation in anchored sets}.

\subsection{Jacobi polynomials}\label{sec:jacobi-polys}

For $\alpha,\beta > - 1$, the Jacobi polynomials are the univariate orthogonal polynomials with respect to the corresponding Jacobi measure \ef{jacobi-meas}. The classical polynomials are denoted by $P^{\alpha,\beta}_{\nu}$ for $\nu \in \bbN_0$ and have the normalization $P^{\alpha,\beta}_{\nu}(1) = {\nu + \alpha \choose \nu}$. Note that $\alpha = \beta = 0$ yields the Legendre polynomials, and $\alpha = \beta = -1/2$ yields the Chebyshev polynomials (of the first kind). See, e.g., \cite[Chpt.\ 4]{szego1975orthogonal} for an overview. In this work, we consider the corresponding orthonormal polynomials with respect to the probability measure \ef{1d-jacobi-meas}, which we denote by $\Psi^{\alpha,\beta}_{\nu}$, $\nu \in \bbN_0$. These polynomials are rescaled versions of the polynomials $P^{\alpha,\beta}_{\nu}$. See \ef{Jacobi-ON} for the precise scaling.

We construct the tensor-product Jacobi polynomials on $\cU$ via tensorization. For a multi-index $\bnu = (\nu_i)_{i \in \bbN} \in \bbN_0^{\bbN}$, we write $\supp(\bnu) = \{ i: \nu_i \neq 0\} \subseteq \bbN_0$ for its support and $\nm{\bm{\nu}}_0 = |\supp(\bnu)|$ for its number of nonzero entries. We denote the set of multi-indices with finitely-many nonzero entries as
\begin{equation}\label{def:cF}
\cF = \{ \bm{\nu} \in \bbN^{\bbN}_0 : \nm{\bm{\nu}}_0 < \infty \} \subset \bbN^{\bbN}_0.
\end{equation}
Now consider sequences $\bm{\alpha}, \bm{\beta} > - \bm{1}$ and a multi-index $\bm{\nu} \in \cF$. We construct the tensorized Jacobi polynomial corresponding to $\bnu$ as follows:
\begin{equation}\label{eq:basis}
\Psi_{\bm{\nu}}^{\bm{\alpha}, \bm{\beta}}(\bm{y}) =
\prod_{i \in \supp (\bm{\nu})}
\Psi_{\nu_i}^{\alpha_i, \beta_i} (y_i), \quad \bm{y} = (y_i)_{i \in \bbN}\in \cU.
\end{equation}
The set $\{  \Psi^{\bm{\alpha},\bm{\beta}}_{\bm{\nu}}\}_{\bm{\nu} \in \cF}$ forms an orthonormal basis of $L^2_{\varrho}(\cU)$, where $\varrho = \varrho_{\bm{\alpha},\bm{\beta}}$ is the measure \ef{tensor-jacobi-measure}.

\rem{
\label{rem:Jac-cond-equiv}
As we show in Proposition \ref{prop:equiv-ass}, our main assumption, Assumption \ref{main-ass}, is equivalent to the property
\begin{equation*}
\nm{\Psi^{\alpha_j,\beta_j}_{\nu_j}}_{L^{\infty}([-1,1])} \leq p(\nu_j),\quad \forall \nu_j \in \bbN_0, j \in \bbN,
\end{equation*}
for some polynomial $p$ that is independent of $j$ and depends on $\bm{\alpha}$ and $\bm{\beta}$ only. See \cite[Rmk.\ 3.5]{bachmayr2017sparse1} for a similar assumption in a related context. In other words, Assumption \ref{main-ass} ensures that the $L^{\infty}$-norms of the constituent Jacobi polynomials do not blow up too fast. This is a key ingredient in our proofs, which involve showing various weighted best approximation error bounds for Jacobi polynomial expansions of $(\bm{b},1)$-holomorphic functions (see Theorems \ref{t:weighted-lp-error}  and \ref{t:weighted-lp-error-monotone}), where the weights are based on these norms (see \ef{u-def}).
}

\subsection{Jacobi polynomial expansions}

Consider the tensor Jacobi polynomials $\{\Psi^{\bm{\alpha},\bm{\beta}}_{\bm{\nu}} \}_{\bm{\nu} \in \cF}$ defined in \S\ref{sec:jacobi-polys}. In the following, to simplify notation, we write $\Psi_{\bm{\nu}}$ instead of $\Psi^{\bm{\alpha},\bm{\beta}}_{\bm{\nu}} $ when there is no ambiguity in doing so. 
Let $f \in L^2_{\varrho}(\cU ; \cV)$. Then $f$ has the convergent (in $L^2_{\varrho}(\cU ; \cV)$) expansion 
\begin{equation}
\label{f-exp}
f = \sum_{\bm{\nu} \in \cF} c_{\bm{\nu}}  \Psi_{\bm{\nu}},
\end{equation}
 where
\begin{equation}
\label{f-coeff}
c_{\bm{\nu}} = c_{\bm{\nu}}^{\bm{\alpha},\bm{\beta}} = \int_{\cU} f(\bm{y})\Psi^{\bm{\alpha},\bm{\beta}}_{\bm{\nu}} (\bm{y}) \D \varrho_{\bm{\alpha},\bm{\beta}}(\bm{y}) \in \cV
\end{equation}
are the coefficients of $f$, which are elements of $\cV$. We write $\bm{c} = \bm{c}^{\bm{\alpha},\bm{\beta}} = (c^{\bm{\alpha},\bm{\beta}}_{\bm{\nu}})_{\bm{\nu} \in \cF}$
for the sequence of coefficients. Next, given an index set $S \subseteq  \cF$, we define the truncated expansion $f_S$ of $f$ as
 \begin{equation}\label{def:truncatedF}
 f_S = \sum_{\bm{\nu} \in S} c_{\bm{\nu}}  \Psi_{\bm{\nu}}.
 \end{equation}

\subsection{$\ell^p$ spaces, norms and best approximation of sequences}\label{S:ellspaces}
We now require some further notation and setup, which is based on \cite{adcock2022sparse}. Let $d \in \bbN \cup \{ \infty \}$ and $\Lambda \subseteq \bbN^{d}_{0}$ be a (multi-)index set, $\bm{c} = (c_{\bnu})_{\bnu \in \cF}$ be a vector or a sequence with $\cV$-valued entries and 
\begin{equation*}
\mathrm{supp}(\bm{c}) = \{ \bnu \in \Lambda : \|c_{\bnu}\|_{\cV} \neq 0 \}
\end{equation*}
be the support of $\bm{c}$. As before, given a subset $S \subseteq \Lambda$, we write $\bm{c}_{S}$ for the vector with $\bm{\nu}$th entry equal to $c_{\bnu}$ if $\bnu \in S$ and zero otherwise.
 
For $0 < p \leq \infty$, we write $\ell^p(\Lambda ; \cV)$ for the space of $\cV$-valued sequences $\bm{c} = (c_{\bm{\nu}})_{\bm{\nu} \in \Lambda}$ indexed over $\Lambda$ for which $\nm{\bm{c}}_{p;\cV} < \infty$, where
\eas{
\nm{\bm{c}}_{p;\cV} = \begin{cases} \left ( \sum_{\bm{\nu} \in \Lambda} \nm{c_{\bm{\nu}}}^p_{\cV} \right )^{\frac1p} & 0 < p < \infty ,
\\
\sup_{\bm{\nu} \in \Lambda} \nm{c_{\bm{\nu}}}_{\cV} & p = \infty. 
\end{cases}
}
When $\cV = \bbR$, we just write $\ell^p(\Lambda)$ and $\nms{\cdot}_p$. Note that $\nms{\cdot}_{p;\cV}$ defines a norm when $1 \leq p \leq \infty$ and quasi-norm when $0<p<1$. Now let $\bm{c} \in \ell^p(\Lambda ; \cV)$ and $s \in \bbN$ with $1 \leq s \leq | \Lambda |$. We define the $\ell^p$-norm \textit{best $s$-term approximation error} of $\bm{c}$  as
\begin{equation}\label{sigma-s-def}
\sigma_s(\bm{c})_{p;\cV} = \inf \left \{ \nm{\bm{c} - \bm{z}}_{p;\cV} : \bm{z} \in \ell^p(\Lambda ; \cV),\ | \mathrm{supp}(\bm{z}) | \leq s \right \}.
\end{equation}
Next, let $\bm{w} = (w_{\bm{\nu}})_{\bm{\nu} {\in \Lambda}} > \bm{0}$ be a vector of positive weights. For $1 \leq p \leq 2$, we define the weighted $\ell^p_{\bm{w}}({\Lambda}; \cV)$ {space} as the space of $\cV$-valued sequences $\bm{c} = (c_{\bm{\nu}})_{\bm{\nu} \in {\Lambda}}$ for which $\nm{\bm{c}}_{p,\bm{w};\cV} < \infty$, where 
\bes{
\nm{\bm{c}}_{p,\bm{w};\cV} =  \left ( \sum_{\bm{\nu} \in {\Lambda}} w^{2-p}_{\bm{\nu}} \nm{c_{\bm{\nu}}}^p_{\cV} \right )^{\frac1p} ,\quad 1 \leq p \leq 2 .
}
Notice that $\nms{\cdot}_{p,\bm{w} ; \cV}$ coincides with the unweighted norm $\nms{\cdot}_{p;\cV}$ for any $\bm{w}$ when $p=2$, or for any $p$ when $\bm{w}=\bm{1}$.  Next, we define the weighted cardinality of an index set  $S \subseteq {\Lambda}$  as $| S |_{\bm{w}} = \sum_{\bm{\nu} \in S} w^2_{\bm{\nu}}$ and, for $0 < k \leq |\Lambda|_{\bm{w}}$, we define the  $\ell^p_{\bm{w}}$-norm \textit{weighted best $(k,\bm{w})$-term
 approximation error} of a $\cV$-valued vector or sequence $\bm{c} \in \ell^p_{\bm{w}}(\Lambda ; \cV)$ as
\begin{equation}
\label{weighted-k-w-term}
\sigma_{{k}}(\bm{c})_{p,\bm{w};\cV} = \inf \left \{ \nm{\bm{c} - \bm{z}}_{p,\bm{w};\cV} : \bm{z} \in \ell^p_{\bm{w}}(\Lambda ; \cV),\ | \mathrm{supp}(\bm{z}) |_{\bm{w}} \leq {k} \right \}.
\end{equation}
We also need several concepts involving anchored sets. A set $\Lambda \subseteq \bbN^d_0$, where $d \in \bbN \cup \{ \infty \}$, of (multi-)indices is \textit{lower} if, whenever $\bm{\nu} \in \Lambda$  and $\bm{\mu} \leq \bm{\nu} $, it also holds that $\bm{\mu} \in \Lambda$. Moreover, $\Lambda$ is \textit{anchored} if it is lower and if, whenever $\bm{e}_j \in \Lambda$ for some $j \in [d]$, it also holds that $\{\bm{e}_1,\bm{e}_2, \ldots, \bm{e}_{j} \}\subseteq \Lambda$.
A scalar-valued sequence $\bm{d} = (d_{\bm{\nu}})_{\bm{\nu} \in \Lambda}$ is \textit{monotonically nonincreasing} if $d_{\bm{\nu}} \geq d_{\bm{\mu}}$ whenever 
$\bm{\nu} \leq \bm{\mu}$. It is \textit{anchored} if it is monotonically nonincreasing and $d_{\bm{e}_j} \leq d_{\bm{e}_i}$ whenever $i,j \in [d]$ with $i \leq j$. Now let $\bm{c} \in \ell^{\infty}({\Lambda};\cV)$. An \textit{anchored majorant} of $\bm{c}$ is any scalar-valued sequence $\bm{d} = (d_{\bm{\nu}})_{\bm{\nu} \in \Lambda}$ that is anchored and satisfies $d_{\bm{\nu}} \geq \nm{c_{\bm{\nu}}}_{\cV}$, $\forall \bnu \in \Lambda$. The \textit{minimal anchored majorant} of $\bm{c}$ is the (unique) scalar-valued sequence $\tilde{\bm{c}}$ that is an anchored majorant and for which $\tilde{\bm{c}} \leq \bm{d}$ for any other anchored majorant $\bm{d}$. It has the explicit expression
\begin{equation}\label{def:min_anch}
\tilde{c}_{\bm{\nu}} = 
\begin{cases}
\sup \{ \|c_{\bm{\mu}}\|_{\cV}: \bm{\mu} \geq \bm{\nu} \} & \text{ if }  \bm{\nu} \neq \bm{e}_j \text{ for any } j \in[d],  \\ 
\sup \{ \|c_{\bm{\mu}}\|_{\cV}: \bm{\mu} \geq \bm{e}_i \text{ for some } i \geq j \} & \text{ if } \bm{\nu} =\bm{e}_j \text{ for some } j \in [d].
\end{cases}
\end{equation}
Given $0 < p \leq \infty$, we define the \textit{anchored $\ell^p$ space} $\ell^p_{\mathsf{A}}(\Lambda;\cV)$ as the space of sequences $\bm{c} \in \ell^{\infty}(\Lambda;\cV)$ for which $\tilde{\bm{c}} \in \ell^p(\Lambda)$, and define the (quasi-)norm $\|\bm{c}\|_{p,\mathsf{A};\cV} = \|\tilde{\bm{c}}\|_{p}$.
Finally, for $1 \leq s \leq |\Lambda|$ we also define the $\ell^p$-norm \textit{best $s$-term approximation error in anchored sets} by
\begin{equation}\label{def:best_anch}
\sigma_{s,\mathsf{A}}(\bm{c})_{p;\cV} = \inf \left \{ \nm{\bm{c} - \bm{z}}_{p;\cV} : \bm{z} \in \ell^p(\Lambda ; \cV),\ | \mathrm{supp}(\bm{z}) | \leq {s}, \, \mathrm{supp}(\bm{z}) \text{ anchored}\right \}.
\end{equation}

\subsection{Best approximation of polynomial coefficients}

We now present two key results on the best approximation of sequences of polynomial coefficients for $(\bm{b},1)$-holomorphic functions. In these results, we consider the tensor Jacobi polynomial basis \ef{eq:basis} and the so-called \textit{intrinsic} weights
\begin{equation}
\label{u-def}
\bm{u} = (u_{\bm{\nu}})_{\bm{\nu} \in \cF},\quad \text{where }u_{\bm{\nu}} = \nm{\Psi_{\bm{\nu}}}_{L^{\infty}(\cU)}
\end{equation}
and $\Psi_{\bm{\nu}} = \Psi^{\bm{\alpha},\bm{\beta}}_{\bm{\nu}}$ is as in \eqref{eq:basis} for all $\bnu \in \cF$. Notice that $\bm{u} \geq \bm{1}$, since the $\Psi_{\bm{\nu}}$ are orthonormal functions and $\varrho$ is a probability measure. 

In our first result, we give bounds for the weighted best $(k,\bm{u})$-term approximation error \ef{weighted-k-w-term}.

\begin{theorem}
\label{t:weighted-lp-error}
Let  $\bm{\alpha},\bm{\beta}$ satisfy Assumption \ref{main-ass} with constant $\tau > 0$ and $\varrho = \varrho_{\bm{\alpha},\bm{\beta}}$ be the tensor-product Jacobi probability measure \ef{tensor-jacobi-measure}. Let $0 < p < 1$, $p < q \leq 2$,  $\bm{b} \in \ell^p(\bbN)$ with $\bm{b} \geq \bm{0}$, $k > 0$,  and consider the weights $\bm{u}$ as in \ef{u-def}. Then the following hold.  
\begin{itemize}
\item[(a)] There exists a set $S \subset \cF$ with $|S|_{\bm{u}} \leq k$ depending on $k$, $\bm{b}$, $\bm{\alpha}$ and $\bm{\beta}$ only such that 
\begin{equation}\label{eq:thm41}
\sigma_k(\bm{c})_{q,\bm{u} ; \cV} \leq \nm{\bm{c} - \bm{c}_S}_{q,\bm{u} ; \cV} \leq C(\bm{b},p,\tau) \cdot k^{\frac1q-\frac1p}, \quad \forall f \in \cH(\bm{b}),
\end{equation}
  where $\bm{c} = (c_{\bm{\nu}})_{\bm{\nu} \in \cF}$ are the coefficients \ef{f-coeff} of $f$ with respect to the tensor Jacobi polynomials \ef{eq:basis}.
Moreover, the sets can be chosen to be nested. If $k' \geq k$ then there are sets $S \subseteq S' \subset \cF$ with $|S|_{\bm{u}} \leq k$ and $|S'|_{\bm{u}} \leq k'$ such that \ef{eq:thm41} holds for $S$ with $k$ and $S'$ with $k'$.
  \item[(b)] For any $r \in (p,1)$ the constant in \ef{eq:thm41} satisfies
\begin{equation*}
  \sup_{\nm{\bm{b}}_{p,\mathsf{M}} \leq 1}  C(\bm{b},r,\tau)  \leq \widetilde{C}(p,r,\tau).
  \end{equation*}
\end{itemize}
 \end{theorem}
In the second result, we provide similar bounds for the best $s$-term approximation error in anchored sets \ef{def:best_anch}. Moreover, we show that the upper bound also holds in the stronger weighted $\ell^q_{\bm{u}}(\cF ; \cV)$-norm -- a property that will be needed later in the proofs of the main results.

\begin{theorem}
\label{t:weighted-lp-error-monotone}
Let  $\bm{\alpha},\bm{\beta}$ satisfy Assumption \ref{main-ass} with constant $\tau > 0$ and $\varrho = \varrho_{\bm{\alpha},\bm{\beta}}$ be the tensor-product Jacobi probability measure \ef{tensor-jacobi-measure}. Let $0 < p < 1$, $\bm{b} \in \ell^p_{\mathsf{M}}(\bbN)$ with $\bm{b} \geq \bm{0}$, $s \in \bbN$, $p \leq q \leq 2$ and consider the weights $\bm{u}$ as in \ef{u-def}. Then the following hold. 
\begin{itemize}
\item[(a)] There exists an anchored set $S \subset \cF$ with $|S| \leq s$ depending on $s$, $\bm{b}$, $\bm{\alpha}$, $\bm{\beta}$ and $q$ only such that
\begin{equation}\label{eq:glencoe}
\sigma_{s,\mathsf{A}}(\bm{c})_{q;\cV}  \leq \nm{\bm{c} - \bm{c}_S}_{q,\bm{u};\cV} \leq C(\bm{b},p,{q},\tau) \cdot s^{\frac1q-\frac1p},  \quad \forall f \in \cH(\bm{b}),
\end{equation}
where $\bm{c} = (c_{\bm{\nu}})_{\bm{\nu} \in \cF}$ are the coefficients \ef{f-coeff} of $f$ with respect to the tensor Jacobi polynomials \ef{eq:basis}.
Moreover, the sets can be chosen to be nested. If $s' \in \bbN$, $s' \geq s$, then there are sets $S \subseteq S' \subset \cF$ with $|S| \leq s$ and $|S'| \leq s'$ such that \ef{eq:glencoe} holds for $S$ with $s$ and $S'$ with $s'$.
\item[(b)] For any $r \in (p,1)$ the constant in \ef{eq:glencoe} satisfies
\begin{equation}
  \sup_{\nm{\bm{b}}_{p,\mathsf{M}} \leq 1}  C(\bm{b},r,{q},\tau) \leq \widetilde{C}(p,r,q,\tau).
  \end{equation}
\end{itemize}
\end{theorem}

Note that the first inequality in both \ef{eq:thm41} and \ef{eq:glencoe} follows immediately from the definition of the best approximation error term and, in the case of \ef{eq:glencoe}, the fact that the weights $\bm{u} \geq \bm{1}$.

These two theorems are extensions of well-known results in the literature. Early results of this type were developed in  \cite{chkifa2015breaking,cohen2015approximation,cohen2010convergence} for Taylor or Legendre polynomial coefficients of holomorphic functions arising as solution maps of certain parametric PDEs and considering, primarily, the best $s$-term approximation \ef{sigma-s-def}, the best $s$-term approximation in anchored sets \ef{def:best_anch} or a variant of the latter based on lower sets only.
A more direct proof for Chebyshev or Legendre polynomial coefficients of $(\bm{b},\varepsilon)$-holomorphic functions (as opposed to parametric PDE solutions) was presented in \cite[Thms.\ 3.28, 3.33]{adcock2022sparse}, based on the arguments in  \cite{cohen2015approximation}. Analysis of the best $(k,\bm{u})$-term approximation error was first considered in \cite[Lem.\ 7.23]{adcock2022sparse} and later in \cite[Thm.\ A.3]{adcock2024efficient}, again for Chebyshev or Legendre polynomial coefficients only.

Parts (a) of Theorems \ref{t:weighted-lp-error} and \ref{t:weighted-lp-error-monotone} extend these previous works by allowing for arbitrary Jacobi polynomial coefficients. We also make explicit the fact that the sets can be chosen in a nested fashion, as this will be needed in later proofs. The uniform bounds for the constants given in parts (b) of Theorems \ref{t:weighted-lp-error} and \ref{t:weighted-lp-error-monotone} are novel, and were not considered in these previous works. They generalize results first shown in \cite{adcock2024optimal}, which considered Legendre polynomials only.

\subsection{Bounds for Jacobi polynomial coefficients of holomorphic functions and an abstract summability lemma}

The remainder of this section is devoted to the proofs of Theorems \ref{t:weighted-lp-error} and \ref{t:weighted-lp-error-monotone}. They involve two main ingredients. First, a bound for the Jacobi polynomial coefficients of holomorphic functions, and second, an abstract summability lemma for sequences satisfying such a bound. Using these results, to prove Theorem \ref{t:weighted-lp-error} we construct a sequence $\bar{\bm{d}}$ based on the coefficients $\bm{c}$ and weights $\bm{u}$, and show that it is $\ell^p$-summable. For Theorem \ref{t:weighted-lp-error-monotone} we construct a similar sequence $\bar{\bm{d}}$ and show that it is $\ell^p_{\mathsf{A}}$-summable. Finally, the proofs of Theorems \ref{t:weighted-lp-error} and \ref{t:weighted-lp-error-monotone} follow by considering appropriate versions of Stechkin's inequality.

Consider $\bnu \in \cF$, where $\cF$ is as in \eqref{def:cF}. Then there exists $N \in \bbN$  such that $\nu_k=0$ for all $k >N$. Thus, the following lemma is a straightforward extension of \cite[Cor.\ B.2.7]{zech2018sparse}. For this reason, we omit its proof.
\begin{lemma}
\label{l:jacobi-coeff-bd}
Let $\bm{\alpha},\bm{\beta} > - \bm{1}$ and suppose that $f : \cU \rightarrow \cV$ is holomorphic in $\cE_{\bm{\rho}}$ for some $\bm{\rho} \geq \bm{1}$. Then its coefficients \eqref{f-coeff} with respect to the tensor Jacobi polynomials \ef{eq:basis} satisfy
\begin{equation}
\label{jacobi-coeff-bd-rho}
\nm{c_{\bm{\nu}}}_{\cV} \leq \nm{f}_{L^{\infty}(\cE_{\bm{\rho}} ; \cV)} \prod_{k \in I(\bm{\nu},\bm{\rho})} \frac{\rho_k^{-\nu_k+1}}{(\rho_k-1)^2} (\nu_k+1),
\end{equation}
where $I(\bm{\nu},\bm{\rho}) = \mathrm{supp}(\bm{\nu}) \cap \{ k : \rho_k > 1 \}$.
\end{lemma}

We now present an abstract summability lemma that asserts $\ell^p$- and $\ell^p_{\mathsf{A}}$-summability for sequences satisfying a bound of the form \ef{jacobi-coeff-bd-rho}.

\lem{
\label{l:abstract-summability}
Let $0 < p < 1$, $\varepsilon > 0$, $\bm{b} = (b_j)_{j \in \bbN} \in \ell^p(\bbN)$, $\bm{b} \geq \bm{0}$ and $\xi : (1,\infty) \rightarrow [0,\infty)$ be continuous and bounded as $t \rightarrow \infty$.
Let $c,\gamma > 0$ be constants and consider a scalar-valued sequence $\bm{d} = (d_{\bm{\nu}})_{\bm{\nu} \in \cF}$ such that $|d_{\bm{0}} | \leq c$ and, for every $\bm{\nu} \in \cF \backslash \{ \bm{0} \}$,
\be{
\label{d-bound-summability}
| d_{\bm{\nu}} | \leq \prod_{k \in \mathrm{supp}(\bnu)} \xi(\rho_k) \rho^{-\nu_k}_k  (c \nu_k+1)^{\gamma},
}
for all $\bm{\rho} = (\rho_j)_{j \in \bbN} \geq \bm{1}$ satisfying $\{ k : \rho_k > 1 \} \supseteq \mathrm{supp}(\bnu)$ and
\be{
\label{rho-b-cond}
\sum^{\infty}_{j=1} \left ( \frac{\rho_j + \rho^{-1}_j}{2} - 1 \right ) b_j \leq \varepsilon.
}
Then the following hold.
\begin{itemize}
\item[a)] The sequence $\bm{d} \in \ell^p(\cF)$ and there exists a constant $C(\bm{b},p,\varepsilon,c,\gamma,\xi)$ such that
\begin{equation}\label{eq:bounddp}
\nm{\bm{d}}_{p} \leq C(\bm{b},p,\varepsilon,c,\gamma,\xi).
\end{equation}
Moreover, for any $r \in (p,1)$ there is a constant $\widetilde{C}(p,r,\varepsilon,c,\gamma,\xi)$ such that
\begin{equation}\label{eq:boundC}
\sup_{\|\bm{b}\|_{p,\mathsf{M}} \leq 1}  C(\bm{b},r,\varepsilon,c,\gamma,\xi) \leq \widetilde{C}(p,r,\varepsilon,c,\gamma,\xi).
\end{equation}
\item[b)] If $\bm{b} \in \ell^p_{\mathsf{M}}(\bbN)$ then $\bm{d} \in \ell^p_{\mathsf{A}}(\cF)$ and   there exists a constant $C(\bm{b},p,\varepsilon,c,\gamma,\xi)$ such that
\begin{equation}\label{eq:bounddpA}
\nm{\bm{d}}_{p,\mathsf{A}} \leq C(\bm{b},p,\varepsilon,c,\gamma,\xi).
\end{equation}
Moreover, for any $r \in (p,1)$ there is a constant $\widetilde{C}(p,r,\varepsilon,c,\gamma,\xi)$ such that
\begin{equation}\label{eq:boundC_A}
\sup_{\|\bm{b}\|_{p,\mathsf{M}} \leq 1}  C(\bm{b},r,\varepsilon,c,\gamma,\xi) \leq \widetilde{C}(p,r,\varepsilon,c,\gamma,\xi).
\end{equation}
\end{itemize}
}

This lemma extends various results in the literature \cite{adcock2022sparse,chkifa2015breaking,cohen2015approximation,cohen2010convergence} in several ways. First, the coefficients $d_{\bm{\nu}}$ do not need to be the (norms of the) orthogonal polynomial coefficients of a $(\bm{b},\varepsilon)$-holomorphic function. They just need to satisfy \ef{d-bound-summability}. In particular, combining this result with Lemma \ref{l:jacobi-coeff-bd} leads to summability results for arbitrary Jacobi polynomial coefficients. Second, this lemma establishes the uniform bounds \ef{eq:boundC} and \ef{eq:boundC_A}, which, as noted, were not considered in past works, with the exception of \cite{adcock2024optimal}, which considered Legendre polynomial coefficients only.

\begin{proof}
We first prove the summability bounds \ef{eq:bounddp} and \ef{eq:bounddpA}. To do so, we adapt the arguments of \cite[Thm.\ 3.28]{adcock2022sparse} and \cite[Thm.\ 3.33]{adcock2022sparse}, respectively, which, as noted, are based on \cite{cohen2015approximation}. Having done this, we show \ef{eq:boundC} and \ef{eq:boundC_A} by extending the arguments of  \cite[Lems.\ A.2 \& A.5]{adcock2024optimal}.

\pbk
\textit{Proof of \ef{eq:bounddp}.}
Write $\bbN = E \cup F$, where $E = [d]$ and $F = \bbN \backslash [d]$ and $d \in \bbN$ is to be chosen later. Define $\kappa = 1 + {\varepsilon}/({2 \nm{\bm{b}}_1})$ and, for $\bm{\nu} \in \cF$, the sequence $\tilde{\bm{\rho}} = \tilde{\bm{\rho}}(\bm{b},\varepsilon,\bm{\nu}) \geq \bm{1}$ by
\begin{equation*}
\frac{\tilde{\rho}_j + \tilde{\rho}^{-1}_j}{2} = \begin{cases} \kappa & j \in \supp(\bm{\nu}_E), \\ \kappa + \frac{\varepsilon \nu_j}{2 b_j \nm{\bm{\nu}_F}_1} & j \in \supp(\bm{\nu}_F), \\
1 & j \notin \supp(\bm{\nu}). \end{cases}
\end{equation*}
Then it is readily checked that $\bm{\rho} = \tilde{\bm{\rho}}$ satisfies \ef{rho-b-cond} and $\{ k : \rho_k > 1 \} =\mathrm{supp}(\bnu)$. Moreover, using the fact that the solution $x = a + \sqrt{a^2-1}$ to the equation $(x+1/x)/2 = a$, where $a \geq 1$, satisfies $x \geq a$, we see that $\tilde{\rho_j} \geq \kappa > 1$, $\forall j \in \mathrm{supp}(\bm{\nu})$.
Since the function $\xi$ is continuous on $(1,\infty)$ and bounded as $t \rightarrow \infty$, there is a constant $C(\kappa,\xi)$ such that
\be{
\label{C-c-tilde-const}
\xi(t) \leq C(\kappa,\xi),\quad \forall t \geq \kappa.
}
Therefore
\bes{
\xi(\tilde{\rho}_j) \leq C(\kappa,\xi) \leq C(\kappa,\xi) \nu_j + 1,\quad \forall j \in \mathrm{supp}(\bm{\nu}).
}
This and \ef{d-bound-summability} now give
\bes{
| d_{\bm{\nu}} | \leq  \prod_{j \in \supp(\bm{\nu})} (C(\kappa,\xi) \nu_j + 1) \tilde{\rho}^{-\nu_j}_j (c \nu_j+1)^{\gamma}, 
\quad {\forall \bnu \in \cF}
}
and therefore
\be{
\label{d-custom-bd}
| d_{\bm{\nu}} | \leq  B(\bm{\nu} ) : = \prod_{j \in \supp(\bm{\nu})} \tilde{\rho}^{-\nu_j}_j (\tilde{c} \nu_j+1)^{\tilde{\gamma}}, 
\quad {\forall \bnu \in \cF,}
}
where $\tilde{c} = \tilde{c}(\kappa,\xi,c) = \max \{ C(\kappa,\xi) , c \}$ and $\tilde{\gamma} = \gamma + 1$.
We now split $B(\bm{\nu}) = B_{E}(\bm{\nu}) B_{F}(\bm{\nu})$, where
\be{
\label{B_S_def}
B_{S}(\bm{\nu}) = \prod_{j \in \supp(\bm{\nu_S})} \tilde{\rho}_{j}^{-\nu_j} (\tilde{c} \nu_j+1)^{\tilde{\gamma}}\quad \text{for }S = E,F.
}
Let $\cF_S = \{ \bm{\nu} \in \cF : \supp(\bm{\nu}) \subseteq S \}$ for $S = E,F$. Since
\be{
\label{d-bound-sigma}
\nm{\bm{d}}^p_p \leq \sum_{\bm{\nu} \in \cF} B(\bm{\nu})^p = \sum_{\bm{\nu} \in \cF_{E}} B_E(\bm{\nu})^p \cdot \sum_{\bm{\nu} \in \cF_{F}} B_F(\bm{\nu})^p = : \Sigma_{E} \cdot \Sigma_F,
}
to prove the result it suffices to show that $\Sigma_E , \Sigma_F < \infty$. Consider $\Sigma_E$. 
By construction, $\tilde{\rho}_j = \tilde{\kappa}$ for $j \in \supp(\bm{\nu}_E)$, where $\tilde{\kappa} = \kappa + \sqrt{\kappa^2-1} > \kappa > 1$.
Therefore
\be{
\label{Sigma-E-bd}
\Sigma_E  = \sum_{\bm{\nu} \in \bbN^d_0} \prod_{j \in \supp(\bm{\nu}) } \left ( \tilde{\kappa}^{-\nu_j} (\tilde{c} \nu_j + 1)^{\tilde{\gamma}}  \right )^p
 = \left ( \sum^{\infty}_{\nu = 0} \left ( \tilde{\kappa}^{-\nu} (\tilde{c} \nu + 1)^{\tilde{\gamma}} \right )^p \right )^d
 < \infty.
}
Now consider $\Sigma_F$. Using \ef{B_S_def} and the definition of $\tilde{\rho}_j$, we have
\be{
\label{B_F_bound}
B_F(\bm{\nu})  \leq \prod_{j \in \supp(\bm{\nu}_F)}  \left ( \frac{2 b_j \nm{\bm{\nu}_F}_1 }{\varepsilon \nu_j} \right )^{\nu_j} (\tilde c \nu_j + 1)^{\tilde \gamma},
}
and therefore, using the estimate $n! \leq n^n \leq n! \E^n$, $\forall n \in \bbN_0$,
\eas{
B_F(\bm{\nu})  \leq   \frac{\nm{\bm{\nu}_F}^{\nm{\bm{\nu}_F}_1}_1}{\bm{\nu}^{\bm{\nu}_F}_{F}} \left ( \frac{2 \bm{b}_F}{\varepsilon} \right )^{\bm{\nu}_F} \prod_{j \in \supp(\bm{\nu}_F)} (\tilde c \nu_j + 1)^{\tilde \gamma} \leq \frac{\nm{\bm{\nu}_F}_1 !}{\bm{\nu}_F !} \left ( \frac{2 \E \bm{b}_F}{\varepsilon} \right )^{\bm{\nu}_F} \prod_{j \in \supp(\bm{\nu}_F)} (\tilde c \nu_j + 1)^{\tilde \gamma}.
}
Now observe that $(\tilde{c} x + 1)^{\tilde\gamma} \leq ((\tilde{c}+1) x)^{\tilde\gamma} \leq  ((\tilde{c}+1) x)^{\tilde\gamma} + 1$ for any $x \geq 1$.
Therefore, we obtain
\be{
\label{Sigma-F-bd}
\Sigma_F \leq \nm{\bm{g}(\bm{b},\varepsilon)}^p_p.
}
Here $\bm{g}(\bm{b},\varepsilon) = ( g({\bm{b}},\varepsilon)_{\bnu})_{\bnu \in \cF}$ is the sequence with $\bm{\nu}$th term
\be{
\label{g-b-eps}
 g({\bm{b}},\varepsilon)_{\bnu} =
  \dfrac{\|\bnu\|_1!}{\bnu!} \bm{h}({\bm{b}},\varepsilon)^{\bnu} 
 \prod_{j \in \bbN}  \left( (\tilde{c} +1)^{\tilde{\gamma}} \nu_j^{\tilde \gamma}+1\right),\quad \text{where}\quad 
 \bm{h} ({\bm{b}},\varepsilon)_j = \dfrac{2\E {{b}}_{j+d}}{\varepsilon},\ \forall j \in \bbN.
}
To complete the proof, we use a so-called \textit{abstract summability condition} (see \cite[Lem.\ 3.29]{adcock2022sparse} or \cite[Lem.\ 3.21]{cohen2015approximation}). This implies that $g(\bm{b},\varepsilon) \in \ell^p(\cF)$, and therefore $\Sigma_F < \infty$, if and only if $\bm{h}(\bm{b},\varepsilon) \in \ell^p(\bbN)$ with $\nm{\bm{h}(\bm{b},\varepsilon)}_{1} < 1$. Since $\bm{b} \in \ell^p(\bbN)$ by assumption, we immediately have that $\bm{h}(\bm{b},\varepsilon) \in \ell^p(\bbN)$. To obtain $\nm{\bm{h}(\bm{b},\varepsilon)}_{1} < 1$ we now choose $d \in \bbN$ sufficiently large so that
\bes{
 \sum_{j > d} b_{j} < \frac{\varepsilon}{2 \E}.
}
This completes the proof of \ef{eq:bounddp}.

\pbk
\textit{Proof of \ef{eq:bounddpA}.} We write $\bbN = E \cup F$, where $E = [d]$ and $F = \bbN \backslash [d]$ once more and $d \in \bbN$ is now large enough so that
\begin{equation}\label{d-first-cond-anchored}
\sum_{j>d} {b}_j \leq \dfrac{\varepsilon}{4\beta},
\end{equation}
for some $\beta$ that will be chosen later. Given $\bnu \in \cF$, we now define $\tilde{\bm{\rho}} = \tilde{\bm{\rho}}(\bm{b} , \beta , \varepsilon, \bm{\nu}) \geq \bm{1}$ by
\begin{equation}\label{def:trho}
\frac{\tilde{\rho}_j + \tilde{\rho}^{-1}_j}{2} = \begin{cases} \kappa & j \in \supp(\bm{\nu}_E), \\
 \kappa +\beta+ \frac{\varepsilon \nu_j}{4 b_j \nm{\bm{\nu}_F}_1} & j \in \supp(\bm{\nu}_F), \\
1 & j \notin \supp(\bm{\nu}), \end{cases}
\end{equation}
where $\kappa=1+\varepsilon / (4 \|\bm{b}\|_1)$. Then $\{ k : \tilde \rho_k > 1 \} \supseteq \mathrm{supp}(\bm{\nu})$ and it is a short argument to show that $\bm{\rho} = \tilde{\bm{\rho}}$ also satisfies \ef{rho-b-cond}. Hence \ef{d-custom-bd} also holds for this choice of $\tilde{\bm{\rho}}$. The remainder of the proof involves constructing an upper bound $\widetilde{B}(\bm{\nu}) \geq B(\bm{\nu})$ for which $(\widetilde{B}(\bm{\nu}))_{\bm{\nu} \in \cF}$ is monotonically nonincreasing, anchored and $\ell^p$-summable. Due to the definition of the $\ell^p_{\mathsf{A}}$-norm (which is defined as the $\ell^p$-norm of the minimal anchored majorant of $\bm{d}$) and the bound \ef{d-custom-bd}, the existence of such a sequence immediately implies \ef{eq:bounddpA}.

We first construct $\widetilde{B}(\bm{\nu})$. We split $B(\bm{\nu}) = B_{E}(\bm{\nu}) B_{F}(\bm{\nu})$ once more, where $B_S(\bm{\nu})$ is as in \ef{B_S_def} with $\tilde{\bm{\rho}}$ as defined by \ef{def:trho}. Now consider $B_E(\bm{\nu})$. Note that there is a constant $D =D(\tilde{c},\tilde \gamma,\tkappa) \geq 1$ such  that
\begin{equation}\label{D-const-def}
 	(\tilde{c} n+1)^{\tilde\gamma}\leq D \left( \dfrac{1+\tkappa}{2}\right)^n,\ \forall n \in \bbN_0.
\end{equation} 	
 Using this and the fact that $\tilde{\rho}_j = \tilde{\kappa} : = \kappa + \sqrt{\kappa^2-1} > 1$ for $j \in E$, we deduce that
\bes{
B_E(\bnu) = \prod_{j \in \supp(\bm{\nu}_E)} \tilde{\rho}^{-\nu_j}_j (\tilde{c} \nu_j+1)^{\tilde\gamma}
\leq
D^d \prod_{j \in \supp(\bm{\nu}_E)} \eta^{\nu_j}  =: \widetilde{B}_E(\bnu),
}
where $\eta = \frac{1+\tkappa}{2\tkappa}<1$. Now consider $B_F(\bm{\nu})$. Since $\tilde{\rho}_j > \beta + \varepsilon \nu_j / (4 b_j \nm{\bm{\nu}_F}_1)$ for all $j \in F$, \ef{B_S_def} implies that
\eas{
B_F(\bnu)
& \leq
\prod_{j \in \supp(\bm{\nu}_F)} (\tilde{c} \nu_j+1)^{\tilde\gamma} \left(\beta+ \frac{\varepsilon \nu_j}{4b_j\|\bnu_F\|_1}\right)^{-\nu_j}
\\
 & \leq \prod_{j \in \supp(\bm{\nu}_F)} (\tilde{c} \nu_j+1)^{\tilde\gamma}\left(\beta+ \frac{\varepsilon \nu_j}{4\tilde{b}_j\|\bnu_F\|_1}\right)^{-\nu_j} =: \widetilde{B}_F(\bnu).
}
Here $\tilde{\bm{b}}=(\tilde{b}_j)_{j \in \bbN} \in \ell^p(\bbN)$ is the minimal monotone majorant of $\bm{b}$, as defined in \S\ref{S:spaces}. Having specified both terms $\widetilde{B}_E(\bm{\nu})$ and $\widetilde{B}_F(\bm{\nu})$, we now formulate the upper bound $\widetilde{B}(\bm{\nu}) = \widetilde{B}_E(\bm{\nu})\widetilde{B}_F(\bm{\nu})$.

We next show that $\widetilde{B}(\bm{\nu})$ is monotonically nonincreasing. Following Step 3 of the proof of \cite[Thm.\ 3.33]{adcock2022sparse}, we do this by showing that this holds for both $\widetilde{B}_E(\bm{\nu})$ and $\widetilde{B}_F(\bm{\nu})$.
First, let $\bm{\mu} \leq \bnu$ and note that {$ \supp(\bm{\bm{\mu}}_E) \subseteq  \supp(\bm{\nu}_E)$}. Then, since $\eta = {(1+\tkappa)}/{(2\tkappa)} <1$, we readily see that $\widetilde{B}_E(\bm{\nu}) \leq \widetilde{B}_E(\bm{\mu})$. Hence $\widetilde{B}_E(\bm{\nu})$ is monotonically nonincreasing. To establish the same for $\widetilde{B}_F(\bnu)$ we show that,  for every $\bnu \in \cF$ and every $i>d$ the inequality $\widetilde{B}_F(\bnu+\bm{e}_i) \leq \widetilde{B}_F(\bnu)$ holds. First, assume that $\nu_i \neq 0$. Using the definition of $\widetilde{B}_F(\bnu)$, we see that
\eas{
\dfrac{\widetilde{B}_F(\bnu+\bm{e}_i)}{\widetilde{B}_F(\bnu)} &= \frac{(\tilde{c} (\nu_i+1) +1)^{\tilde \gamma}}{(\tilde{c}\nu_i+1)^{\tilde \gamma}} \frac{\left (\beta + \frac{\varepsilon (\nu_i+1)}{4 \tilde{b}_i (\nm{\bm{\nu}_F}_1+1)} \right )^{-(\nu_i+1)} }{\left ( \beta + \frac{\varepsilon \nu_i}{4 \tilde{b}_i \nm{\bm{\nu}_F}_1} \right )^{-\nu_i} } \prod_{j \in \mathrm{supp}(\bm{\nu}_F) \backslash \{i\} } \left (  \frac{\beta + \frac{\varepsilon \nu_j}{4 \tilde{b}_j (\nm{\bm{\nu}_F}_1+1)}}{ \beta + \frac{\varepsilon \nu_j}{4 \tilde{b}_j \nm{\bm{\nu}_F}_1} }  \right )^{-\nu_j} .
}
Consider the three terms on the right hand side separately. For the first, we observe that
\begin{equation*}
\dfrac{(\tilde{c} (\nu_i+1) +1)^{\tilde\gamma}}{(\tilde{c}\nu_i+1)^{\tilde\gamma}}= \left ( \dfrac{(\nu_i+1)(\tilde{c}+ \frac{1}{(\nu_i+1)})}{\nu_i(\tilde{c}+\frac{1}{\nu_i})} \right )^{\tilde \gamma} \leq 2^{\tilde\gamma}.
\end{equation*}
For the second, we use the inequality $\frac{a+c}{b+c} \geq \frac{a}{b}$ for all $b \geq a > 0$ and $c \geq 0$. Since $i > d$, and therefore $\nu_i \leq \nm{\bm{\nu}_F}_1$, an application of this inequality with $c = 1$ gives
\bes{
\frac{\left (\beta + \frac{\varepsilon (\nu_i+1)}{4 \tilde{b}_i (\nm{\bm{\nu}_F}_1+1)} \right )^{-(\nu_i+1)} }{\left ( \beta + \frac{\varepsilon \nu_i}{4 \tilde{b}_i \nm{\bm{\nu}_F}_1} \right )^{-\nu_i} } \leq \left (\beta + \frac{\varepsilon (\nu_i+1)}{4 \tilde{b}_i (\nm{\bm{\nu}_F}_1+1)} \right )^{-1}  \frac{\left (\beta + \frac{\varepsilon \nu_i}{4 \tilde{b}_i \nm{\bm{\nu}_F}_1} \right )^{-\nu_i} }{\left ( \beta + \frac{\varepsilon \nu_i}{4 \tilde{b}_i \nm{\bm{\nu}_F}_1} \right )^{-\nu_i} } \leq \beta^{-1}.
}
For the third term, we use the same inequality again, except with $c = \beta$. This gives
\bes{
\prod_{j \in \mathrm{supp}(\bm{\nu}_F) \backslash \{i\} } \left (  \frac{\beta + \frac{\varepsilon \nu_j}{4 \tilde{b}_j (\nm{\bm{\nu}_F}_1+1)}}{ \beta + \frac{\varepsilon \nu_j}{4 \tilde{b}_j \nm{\bm{\nu}_F}_1} }  \right )^{-\nu_j} \leq \prod_{j \in \mathrm{supp}(\bm{\nu}_F) \backslash \{i\} } \left( \dfrac{\|\bnu_F\|_1}{\|\bnu_F\|_1+1} \right)^{-\nu_j} \leq \left( \dfrac{\|\bnu_F\|_1}{\|\bnu_F\|_1+1} \right)^{-\|\bnu_F\|_1} \leq \E.
}
Here, in the last step, we used the elementary inequality $(1+1/n)^n \leq \E$, $\forall n \in \bbN$. Combining this with the other two bounds, we deduce that
\bes{
\dfrac{\widetilde{B}_F(\bnu+\bm{e}_i)}{\widetilde{B}_F(\bnu)} \leq 
 \dfrac{2^{\tilde\gamma}\E}{\beta}.
}
Now suppose that $\nu_i=0$. Then we get
\bes{
\dfrac{\widetilde{B}_F(\bnu+\bm{e}_i)}{\widetilde{B}_F(\bnu)} \leq (\tilde{c}+1)^{\tilde\gamma} \left( \beta+ \dfrac{\varepsilon}{4\tilde{b}_i(\|\bnu_F\|_1+1)} \right)^{-1} \left( \dfrac{\|\bnu_F\|_1}{\|\bnu_F\|_1+1} \right)^{-\|\bnu_F\|_1} \leq \dfrac{(\tilde{c}+1)^{\tilde{\gamma}}\E}{\beta}.
}
In order to ensure that $\widetilde{B}_F(\bnu+\bm{e}_i) / \widetilde{B}_F(\bnu)  \leq 1$, in both cases, it is therefore sufficient to choose 
\begin{equation}\label{eq:choice_beta}
\beta \geq (\max \lbrace \tilde{c},1 \rbrace+1)^{\tilde{\gamma}}\E.
\end{equation}
This shows that $\widetilde{B}_F(\bnu)$ is monotonically nonincreasing, and therefore so is $\widetilde{B}(\bm{\nu})$. 

We now show that this sequence  is anchored, i.e., $\widetilde{B}(\bm{e}_j) \leq \widetilde{B}(\bm{e}_i)$ for all $i , j \in \bbN$ with $i \leq j$. We proceed as \cite[Thm.\ 3.33, Step 6]{adcock2022sparse}.  Observe that
\begin{equation*}
\widetilde{B}(\bm{e}_j) = \begin{cases} 
D^d\eta & j \in E, \\
(\tilde{c}+1)^{\tilde\gamma} \left( \beta+ \dfrac{\varepsilon}{4\tilde{b}_j} \right)^{-1} & j \in F.
\end{cases}
\end{equation*}
Let $i , j \in \bbN$ with $i \leq j$. We consider the following three cases separately: $i,j \in E$, $i,j \in F$ and $i \in E$, $j \in F$. For the first case, we have $\widetilde{B}(\bm{e}_j) = \widetilde{B}(\bm{e}_i) =D^d\eta $, as required. Now, consider the second case.  Since $\tilde{\bm{b}}$  is monotonically nonincreasing, we have $\tilde{b}_j \leq \tilde{b}_i$ and therefore $\widetilde{B}(\bm{e}_j)  \leq \widetilde{B}(\bm{e}_i)$, as required. Finally, consider the third case. The condition $\widetilde{B}(\bm{e}_j)  \leq \widetilde{B}(\bm{e}_i)$ for $i \in E$, $j \in F$ is equivalent to
\bes{
 (\tilde{c}+1)^{\tilde{\gamma}} \left ( \beta + \frac{\varepsilon}{4 \tilde{b}_j} \right )^{-1} \leq D^d \eta\quad \Longleftrightarrow \quad \beta + \frac{\varepsilon}{4 \tilde{b}_j}  \geq \frac{(\tilde{c}+1)^{\tilde\gamma}}{ D^d \eta} .
}
Since $D \geq 1$ by definition, this is guaranteed by choosing $\beta \geq (\tilde{c}+1)^{\tilde{\gamma}} /  \eta$. However, $1/\eta = 2 \tilde{\kappa} / (1+\tilde{\kappa}) \leq 2$. So this is implied by \ef{eq:choice_beta}.
With any such $\beta$, we conclude that $\widetilde{B}(\bnu)$ is both monotonically decreasing and anchored.

We are now, finally, ready to establish \ef{eq:bounddpA}. We first write
\be{
\label{d-bound-sigma-A}
\nm{\bm{d}}^p_{p,\mathsf{A}} \leq \sum_{\bm{\nu} \in \cF} \widetilde{B}(\bm{\nu})^p = \sum_{\bm{\nu} \in \cF_{E}} \widetilde{B}_E(\bm{\nu})^p \cdot \sum_{\bm{\nu} \in \cF_{F}} \widetilde{B}_F(\bm{\nu})^p = : \widetilde{\Sigma}_{E} \cdot \widetilde{\Sigma}_F.
}
Using the definition of $\widetilde{B}_{E}(\bnu)$, we see that
\be{
\label{tilde-Sigma-E-bd}
\widetilde{\Sigma}_E = D^{dp} \left ( \sum^{\infty}_{n=0} \eta^{p n} \right )^d < \infty,
}
since $\eta = \frac{1+\tilde\kappa}{2\tilde\kappa} < 1$ by definition. For $\widetilde{\Sigma}_F$, we first notice that $\widetilde{B}_F(\bnu)$ also satisfies \ef{B_F_bound}, except with $2 b_j$ replaced by $4 \tilde{b}_j$. Hence, we may argue exactly as in the corresponding part of the proof of \ef{eq:bounddp} to deduce that
\be{
\label{tilde-Sigma-F-bd}
\widetilde{\Sigma}_F \leq \nm{\tilde{\bm{g}}(\bm{b},\varepsilon)}^p_p,
}
where $\tilde{\bm{g}}(\bm{b},\varepsilon) = ( \tilde{g}({\bm{b}},\varepsilon)_{\bnu})_{\bnu \in \cF}$ is the sequence with $\bm{\nu}$th term
\begin{equation}\label{tilde-g-b-eps}
\tilde{g}({\bm{b}},\varepsilon)_{\bnu} = \dfrac{\|\bnu\|_1!}{\bnu!} \tilde{\bm{h}}({\bm{b}},\varepsilon)^{\bnu} 
 \prod_{j \in \bbN}  \left((\tilde{c}+1)^{\tilde\gamma} \nu_j^{\tilde{\gamma}}+1\right),\quad \text{for}\quad 
\tilde{\bm{h}} ({\bm{b}},\varepsilon)_j = \dfrac{4\E {\tilde{b}}_{j+d}}{\varepsilon},\ \forall j \in \bbN.
 \end{equation}
We now show that $\Sigma_F < \infty$ exactly as in the proof of \ef{eq:bounddp}, by choosing $d \in \bbN$ sufficiently large so that
\bes{
\sum_{j > d} \tilde{b}_j < \frac{\varepsilon}{4 \E}
}
holds, in addition to \ef{d-first-cond-anchored}. This completes the proof of \ef{eq:bounddpA}.

\pbk
\textit{Proof of \ef{eq:boundC} and \ef{eq:boundC_A}.} Let $C_1(\bm{b},r,\varepsilon,c,\gamma,\xi)$ and $C_2(\bm{b},r,\varepsilon,c,\gamma,\xi)$ be any constants such that \ef{eq:bounddp} and \ef{eq:bounddpA} hold, respectively, with $p$ replaced by $r$. Suppose that $\bm{b} \in \ell^r_{\mathsf{M}}(\bbN)$. Then part (b) implies that $\bm{d} \in \ell^r_{\mathsf{A}}(\cF)$ and, therefore,
since the $\ell^r_{\mathsf{A}}$-norm dominates the $\ell^r$-norm, we have
\bes{
\nm{\bm{d}}_r \leq \nm{\bm{d}}_{r,\mathsf{A}} \leq C_2(\bm{b},r,\varepsilon,c,\gamma,\xi).
}
Hence we may assume without loss of generality that $C_1(\bm{b},r,\varepsilon,c,\gamma,\xi) \leq C_2(\bm{b},r,\varepsilon,c,\gamma,\xi)$ for all $\bm{b} \in \ell^r_{\mathsf{M}}(\bbN)$. Using this and the fact that any $\bm{b} \in \ell^p_{\mathsf{M}}(\bbN)$ belongs to $\ell^r_{\mathsf{M}}(\bbN)$ for $r > p$, we deduce that 
\bes{
\sup_{\nm{\bm{b}}_{p,\mathsf{M}} \leq 1} C_1(\bm{b},r,\varepsilon,c,\gamma,\xi) \leq \sup_{\nm{\bm{b}}_{p,\mathsf{M}} \leq 1} C_2(\bm{b},r,\varepsilon,c,\gamma,\xi).
}
This implies that \ef{eq:boundC} is a consequence of \ef{eq:boundC_A}, and therefore it remains to show \ef{eq:boundC_A}.

Let $\bm{b} \in \ell^p_{\mathsf{M}}(\bbN)$ with $\|\bm{b}\|_{p,\mathsf{M}}\leq 1$ and $r \in (p,1)$. Due to \ef{d-bound-sigma-A}, \ef{tilde-Sigma-E-bd} and \ef{tilde-Sigma-F-bd}, we see that the constant $C_2(\bm{b},r,\varepsilon,c,\gamma,\xi)$ can be chosen as
\begin{equation}\label{eq:constant_c}
C_2(\bm{b},r,\varepsilon,c,\gamma,\xi) = D^d \left ( \sum^{\infty}_{n=0} \eta^{r n} \right )^{d/r}  \| \tilde{\bm{g}} (\bm{b},\varepsilon)  \|_r,
\end{equation}
where $\eta = \frac{1+\tilde{\kappa}}{2 \tilde{\kappa}} \geq 1$, $D = D(\tilde{c},\tilde{\gamma},\tilde{\kappa})$ is any constant such that \ef{D-const-def} holds, $\tilde{\bm{g}}(\bm{b},\varepsilon)$ is as in \ef{tilde-g-b-eps} and $d \in \bbN$ is sufficiently large so that \ef{d-first-cond-anchored} holds for $\beta$ satisfying \ef{eq:choice_beta}.

Recall that $\tilde{c} = \tilde{c}(\kappa,\xi,c) = \max \{ C(\kappa,\xi) , c \}$, where $C(\kappa,\xi)$ is any constant such that \ef{C-c-tilde-const} holds.
Notice that $\tilde{\kappa} > \kappa = 1+\varepsilon/(4 \nm{\bm{b}}_1) \geq 1 + \varepsilon/4$ since $\nm{\bm{b}}_1 \leq \nm{\bm{b}}_p \leq \nm{\bm{b}}_{p,\mathsf{M}} \leq 1$. Therefore, we have $\tilde{c} \leq \bar{c} = \bar{c}(\varepsilon,\xi,c)$. Recall also that $\tilde\gamma = \gamma+1$ depends on $\gamma$ only. Therefore, we have $D \leq \widetilde{D}$ for some $\widetilde{D} = \widetilde{D}(\varepsilon,c,\gamma,\xi)$. We now also choose $\beta = \tilde{\beta}(\varepsilon,c,\gamma,\xi) = ( \max \{ \bar{c} , 1 \} + 1)^{\tilde{\gamma}} \E $ so that \ef{eq:choice_beta} holds. Finally, observe that
\bes{
\eta = \frac{1+\tilde{\kappa}}{2 \tilde{\kappa}} \leq \frac{1+\frac{1}{1+\varepsilon/4}}{2} = \tilde{\eta} = \tilde{\eta}(\varepsilon) < 1.
}
Substituting this and the other bounds into \ef{eq:constant_c}, we deduce that 
\bes{
C_2(\bm{b},r,\varepsilon,c,\gamma,\xi) \leq (C_3(r,\varepsilon,c,\gamma,\xi))^d \cdot \| \tilde{\bm{g}} (\bm{b},\varepsilon)  \|_r,
}
for any $d$ such that \ef{d-first-cond-anchored} holds with $\beta = \tilde{\beta}$ as defined above.

We now consider the term $\| \tilde{\bm{g}} (\bm{b},\varepsilon)  \|_r$. Recall that $\tilde{\bm{g}} (\bm{b},\varepsilon)$ is defined in \ef{tilde-g-b-eps}, where $\tilde{\bm{b}}$ is the monotone majorant of $\bm{b}$. Since $\bm{b} \in \ell^p_{\mathsf{M}}(\bbN)$ with $\nm{\bm{b}}_{p,\mathsf{M}} \leq 1$ we have $\tilde{b}_j \leq j^{-1/p}$. Therefore, we may bound $\tilde{\bm{g}}(\bm{b},\varepsilon) \leq \bar{\bm{g}}(p,\varepsilon)$, where
\bes{
\bar{\bm{g}}(p,\varepsilon)_{\bm{\nu}} = \dfrac{\|\bnu\|_1!}{\bnu!} \bar{\bm{h}}(p,\varepsilon)^{\bnu} 
 \prod_{j \in \bbN}  \left((\bar{c}+1)^{\tilde\gamma} \nu_j^{\tilde{\gamma}}+1\right),\quad \text{for}\quad 
\bar{\bm{h}} (p,\varepsilon)_j = \dfrac{4\E}{\varepsilon (j+d)^{1/p}},\ \forall j \in \bbN.
}
This implies that 
\be{
\label{eq:constant_c_continued}
C_2(\bm{b},r,\varepsilon,c,\gamma,\xi) \leq (C_3(r,\varepsilon,c,\gamma,\xi))^d \cdot \| \bar{\bm{g}} (p,\varepsilon)  \|_r,
}
for any $d$ such that \ef{d-first-cond-anchored} holds with $\beta = \tilde{\beta}$ as above. We now show that $\bar{\bm{g}}(p,\varepsilon) \in \ell^r(\cF)$, exactly as in the earlier parts of the proof. First, notice that $\bar{\bm{h}}(p,\varepsilon) \in \ell^r(\bbN)$ since $r > p$. Also, $\nm{\bar{\bm{h}}(p,\varepsilon) }_1 < 1$, provided $d$ is sufficiently large so that
\bes{
\sum^{\infty}_{j=1} \frac{1}{(j+d)^{1/p}} < \frac{\varepsilon}{4 \E}.
}
Now pick $d = d(p,\varepsilon,c,\gamma,\xi)$ as the minimal $d \in \bbN$ such that this and \ef{d-first-cond-anchored} hold, the latter with $\beta = \tilde{\beta}$. Then the abstract summability condition (see \cite[Lem.\ 3.29]{adcock2022sparse} or \cite[Lem.\ 3.21]{cohen2015approximation}), implies that $\bar{\bm{g}}(p,\varepsilon) \in \ell^r(\cF)$. Since this sequence depends on $\bar{c} = \bar{c}(\varepsilon,\xi,c)$, $\tilde{\gamma} = \gamma+1$, $p$, $\varepsilon$ and $d = d(p,\varepsilon,c,\gamma,\xi)$ only, we deduce that $\nm{\bar{\bm{g}}(\bm{b},\varepsilon)}_r \leq C_4(p,r,\varepsilon,c,\gamma,\xi)$. Hence, \ef{eq:constant_c_continued} implies that 
\bes{
C_2(\bm{b},r,\varepsilon,c,\gamma,\xi) \leq (C_3(r,\varepsilon,c,\gamma,\xi))^{d(p,\varepsilon,c,\gamma,\xi)} \cdot C_4(p,r,\varepsilon,c,\gamma,\xi) = \widetilde{C}(p,r,\varepsilon,c,\gamma,\xi).
}
Since $\bm{b} \in \ell^p_{\mathsf{M}}(\bbN)$, $\nm{\bm{b}}_{p,\mathsf{M}} \leq 1$, was arbitrary, we deduce that \ef{eq:boundC_A} holds, and therefore \ef{eq:boundC} as well. This completes the proof.
\end{proof}

\subsection{Proofs of Theorems \ref{t:weighted-lp-error} and \ref{t:weighted-lp-error-monotone}}

We now prove Theorems \ref{t:weighted-lp-error} and \ref{t:weighted-lp-error-monotone}. As noted, this involves constructing certain surrogate sequences, then asserting their summability, and finally, applying appropriate variants of Stechkin's inequality.

\begin{proof}
[Proof of \cf{t:weighted-lp-error}]
We first prove part (a). By \cf{l:jacobi-coeff-bd}, the coefficients of any $f \in \cH(\bm{b})$ satisfy
\begin{equation}
\label{tilde-c-def}
\nm{c_{\bm{\nu}}}_{\cV} \leq \inf \left \{ \prod_{k \in I(\bm{\nu},\bm{\rho})} \frac{\rho^{-\nu_k+1}_{{k}}}{(\rho_k-1)^2} (\nu_k+1) : \bm{\rho} = (\rho_k)^{\infty}_{k=1} \geq \bm{1}\text{ satisfies \ef{rho-b-cond} with $\varepsilon = 1$} \right \} = : d_{\bm{\nu}}.
\end{equation}
We now show that $\bm{d} = (d_{\bm{\nu}})_{\bm{\nu} \in \cF}$ is $\ell^p_{\bm{u}}$-summable, {where $\bm{u}=(u_{\bnu})_{\bnu \in \cF} \geq 1$ are the intrinsic weights defined in \eqref{u-def}}. Observe that this is equivalent to showing that the sequence $\bar{\bm{d}} = (\bar{d}_{\bm{\nu}})_{\bm{\nu} \in \cF}$ is $\ell^p$-summable, where
\begin{equation*}
\bar{d}_{\bm{\nu}} = d_{\bm{\nu}} u^{2/p-1}_{\bm{\nu}}, \quad {\forall \bnu \in \cF}.
\end{equation*}
Now  Assumption \ref{main-ass} and \cf{l:jacobi-uniform-norm} imply that
\bes{
u_{\bm{\nu}} = \prod_{j \in \mathrm{supp}(\bm{\nu})} \nm{\Psi^{\alpha_j,\beta_j}_{\nu_j}}_{L^{\infty}([-1,1])} \leq \prod_{j \in \mathrm{supp}(\bnu)} (1 + c \nu_j )^{\gamma},
}
where $c = c(\tau)$ and $\gamma = \gamma(\tau)$. Let $\bm{\nu} \in \cF$, $\bm{\rho} \geq \bm{1}$ with $\{ k : \rho_k > 1 \} \subseteq \mathrm{supp}(\bnu)$ be such that \ef{rho-b-cond} holds with $\varepsilon =1$. Then we have
\begin{equation*}
\bar{d}_{\bm{\nu}} \leq \prod_{k \in \mathrm{supp}(\bnu)} \frac{\rho^{-\nu_k+1}_{{k}}}{(\rho_k-1)^2} (\overline{c} \nu_k+1)^{\overline{\gamma}},
\end{equation*}
where $\bar{c} = \bar{c}(p,\tau)$ and $\bar{\gamma} = \bar{\gamma}(p,\tau)$.
We now apply \cf{l:abstract-summability}, part (a) with $\varepsilon = 1$ and the function $\xi(t) = t / (t-1)^2$. This gives that $\bar{\bm{d}} \in \ell^p(\cF)$ and therefore
\begin{equation*}
\nm{{\bm{d}}}_{p,\bm{u}} = \nm{\bar{\bm{d}}}_{p} \leq C(\bm{b},p,\overline{c},\overline{\gamma}),
\end{equation*}
where $C(\bm{b},p,\overline{c},\overline{\gamma})$ is the constant from \ef{eq:bounddp} with $c$ and $\gamma$ replaced by $\bar{c}$ and $\bar{\gamma}$, respectively, and the above choices of $\varepsilon$ and $\xi$.
Having shown weighted summability, we are now ready to prove \ef{eq:thm41} and the assertion below it that the set $S$ can be chosen in a nested manner. We do this using the weighted Stechkin's inequality (see \cite[Lem. 3.12]{adcock2022sparse}) and, for the latter, by inspecting its proof. The proof of this inequality constructs the set $S$ as follows. Let $\pi : \cF \rightarrow \bbN$ be a bijection that gives a nonincreasing rearrangement of the sequence $(d_{\bm{\nu}} / u_{\bm{\nu}} )_{\bm{\nu} \in \cF}$ in absolute value, i.e., $|d_{\pi(1)}| / u_{\pi(1)} \geq |d_{\pi(2)}| / u_{\pi(2)} \geq \cdots $. Let $k > 0$ be as in the theorem statement, $M \in \bbN \cup \{ \infty \}$ be the maximal integer such that $\sum^{M}_{j=1} u^2_{\pi(j)} \leq k$ and define $S = \{ \pi(1),\ldots,\pi(M) \}$. The remainder of the proof of \cite[Lem. 3.12]{adcock2022sparse} goes on to show that this choice of $S$ yields the bound
\bes{
 \sigma_k( {\bm{d}})_{q,\bm{u}} \leq \nm{\bm{d} - \bm{d}_S}_{q,\bm{u}} \leq \nm{{\bm{d}}}_{p,\bm{u}} k^{\frac1q-\frac1p}.
}
Here $\sigma_k( {\bm{d}})_{q,\bm{u}}$ denotes the weighted best $(k,\bm{u})$-term approximation error in the case $\cV = \bbR$. See \ef{weighted-k-w-term}.
Observe that the set $S$ depends on $k$, the intrinsic weights $\bm{u}$ and the sequence $\bm{d}$ only. The former depend on $\bm{\alpha}$ and $\bm{\beta}$, while the latter depends on $\bm{b}$. Hence $S = S(k,\bm{b},\bm{\alpha},\bm{\beta})$, as required.
Moreover, it is clear from the construction that $S$ can be chosen so that the nestedness property holds. Using this, we deduce that
\bes{
\sigma_k(\bm{c})_{q,\bm{u};\cV} \leq \sigma_k( {\bm{d}})_{q,\bm{u}} \leq C(\bm{b},p,\overline{c},\overline{\gamma}) \cdot k^{\frac1q-\frac1p},
}
for the coefficients $\bm{c}$ of any $f \in \cH(\bm{b})$. Since $\bar{c} = \bar{c}(p,\tau)$ and $\bar{\gamma} = \bar{\gamma}(p,\tau)$, this completes the proof of (a).
Part (b) follows immediately from \ef{eq:boundC} applied to the constant
$C(\bm{b},r,\overline{c},\overline{\gamma})$. That is, for any $r \in (p,1)$, we have
\begin{equation*}
\sup_{\|\bm{b}\|_{p,\mathsf{M}} \leq 1}  C(\bm{b},r,\overline{c},\overline{\gamma}) \leq \widetilde{C}(p,r,\bar{c},\bar{\gamma})
\end{equation*}
for some constant $\widetilde{C}(p,r,\bar{c},\bar{\gamma})$. We now recall once more that $\bar{c} = \bar{c}(p,\tau)$ and $\bar{\gamma} = \bar{\gamma}(p,\tau)$.
\end{proof}

\begin{proof}
[Proof of \cf{t:weighted-lp-error-monotone}]
We once more bound the coefficients $\bm{c}$ of any $f \in \cH(\bm{b})$ using the sequence $\bm{d}$ defined in \ef{tilde-c-def}. Now observe that, for any $S \subseteq \cF$,
\begin{equation*}
\nm{\bm{c} - \bm{c}_S}_{q,\bm{u} ; \cV} = \left ( \sum_{\bm{\nu \notin S}} u^{2-q}_{\bm{\nu}} \nm{c_{\bm{\nu}}}^q_{\cV} \right )^{\frac{1}{q}} \leq \left ( \sum_{\bm{\nu \notin S}} u^{2-q}_{\bm{\nu}} d_{\bm{\nu}}^q \right )^{\frac{1}{q}} =  \nm{\bar{\bm{d}} - \bar{\bm{d}}_S}_{q},
\end{equation*}
where $\bar{\bm{d}}$ is the sequence with terms $\bar{d}_{\bm{\nu}} = d_{\bm{\nu}} u^{2/q-1}_{\bm{\nu}}$. Now, as in the previous proof, we have
\begin{equation*}
\bar{d}_{\bm{\nu}} \leq \prod_{k \in \supp(\bm{\nu})} \frac{\rho^{-\nu_k+1}_k}{(\rho_k-1)^2} (\overline{c} \nu_k+1)^{\overline{\gamma}},
\end{equation*}
for all $\bm{\rho} \geq \bm{1}$ satisfying \ef{rho-b-cond} and with $\{ k : \rho_k > 1 \} \supseteq \mathrm{supp}(\bnu)$, where $\bar{c} = \bar{c}(q,\tau)$ and $\bar{\gamma} = \bar{\gamma}(q,\tau)$.
Using this and \cf{l:abstract-summability}, part (b) with $\xi(t) = t / (t-1)^2$ and $\varepsilon = 1$, we deduce that $\bar{\bm{d}} \in \ell^p_{\mathsf{A}}(\cF)$ with
\begin{equation*}
\nm{\bar{\bm{d}}}_{p,\mathsf{A}} \leq C(\bm{b},p,\overline{c},\overline{\gamma}),
\end{equation*}
where $C(\bm{b},p,\overline{c},\overline{\gamma})$ is the constant from \ef{eq:bounddpA} with $c$ and $\gamma$ replaced by $\bar{c}$ and $\bar{\gamma}$, respectively, and the above choices of $\varepsilon$ and $\xi$.

To complete the proof of part (a), we use Stechkin's inequality in anchored sets \cite[Lem.~3.32]{adcock2022sparse} and its proof. Let $\tilde{\bm{d}}$ be the minimal anchored majorant \ef{def:min_anch} of $\bar{\bm{d}}$, so that $\nm{\tilde{\bm{d}}}_{p} = \nm{\bar{\bm{d}}}_{p,\mathsf{A}} \leq C(\bm{b},p,\overline{c},\overline{\gamma}).$ Given $s \in \bbN$, let $S$ be an index set associated with $s$ largest entries of $\tilde{\bm{d}}$. Since $\tilde{\bm{d}}$ is anchored, $S$ can be chosen to be an anchored set. Then, since $\tilde{\bm{d}}$ is a majorant of $\bar{\bm{d}}$, we have
\bes{
\nm{\bar{\bm{d}} - \bar{\bm{d}}_S}^q_{q} = \sum_{\bm{\nu} \notin S} \nm{\bar{d}_{\bm{\nu}}}^q_{\cV} \leq \sum_{\bm{\nu} \notin S} (\tilde{d}_{\nu})^q = \sigma_s(\tilde{\bm{d}})^q_q \leq \nm{\tilde{\bm{d}}}^q_p \cdot s^{1-q/p} \leq C(\bm{b},p,\bar{c},\bar{\gamma})^q s^{1-q/p},
}
where in the penultimate inequality we applied the standard version of Stechkin's inequality (see, e.g., \cite[Lem.\ 3.5]{adcock2022sparse}). Using this and the fact that $\bm{u} \geq \bm{1}$, we deduce that
\begin{equation*}
\sigma_{s,\mathsf{A}}(\bm{c})_{q;\cV}  \leq
\nm{\bm{c} - \bm{c}_S}_{q;\cV} \leq
\nm{\bm{c} - \bm{c}_S}_{q,\bm{u};\cV} \leq \nm{\bar{\bm{d}} - \bar{\bm{d}}_S}_{q} \leq C(\bm{b},p,\overline{c},\overline{\gamma}) \cdot s^{\frac1q-\frac1p}. 
\end{equation*}
The bound \ef{eq:glencoe} now follows, since $\bar{c} = \bar{c}(q,\tau)$ and $\bar{\gamma} = \bar{\gamma}(q,\tau)$. Observe also that the set $S$ is anchored by construction and depends on $s$ and $\bar{\bm{d}}$. Since $\bar{\bm{d}}$ depends on $\bm{b}$, $\bm{\alpha}$, $\bm{\beta}$ and $q$ only, we deduce that $S$ has the desired form. Finally, the nestedness property follows immediately, due to the construction of $S$ as an anchored set containing $s$ largest entries of $\tilde{\bm{d}}$. This completes the proof of part (a).

As in the previous proof, part (b) now follows immediately for any $r \in (p,1)$ by applying \ef{eq:boundC_A} to the constant $C(\bm{b},r,\bar{c},\bar{\gamma})$ and recalling that $\bar{c} = \bar{c}(q,\tau)$ and $\bar{\gamma} = \bar{\gamma}(q,\tau)$ once more.
\end{proof}

\section{Proofs of Theorems \ref{t:prob-bound-known} and \ref{t:prob-bound-unknown}} \label{S:proofs_prob}

We now prove Theorems \ref{t:prob-bound-known} and \ref{t:prob-bound-unknown}.

\subsection{Reconstruction map: known anisotropy}\label{S:setup}
We commence with Theorem \ref{t:prob-bound-known}. Our first task is to define the reconstruction map $\cT$. Let $\cU=[-1,1]^{\bbN}$, $\cV$ be a Hilbert space,  
$\bm{\alpha},\bm{\beta}$ satisfy Assumption \ref{main-ass} and $\varrho = \varrho_{\bm{\alpha},\bm{\beta}}$ be as in \eqref{tensor-jacobi-measure}. Let $m \geq 3$, {$0<p<1$}, $0 < \epsilon < 1$ and $\bm{b} = (b_i)_{i \in \bbN} \in \ell^p(\bbN)$ with $\bm{b} \geq \bm{0}$ be given, and define
\begin{equation}\label{s-def-LS}
k = \frac{m}{c \log(m/\epsilon)} ,
\end{equation}
where $c \geq 1$ is a constant that will be chosen in the proof.

Now let $\pi : \bbN \rightarrow \bbN$ be the bijection that gives a nonincreasing rearrangement of $\bm{b}$, i.e., $b_{\pi(1)} \geq b_{\pi(2)} \geq \cdots$, and write $\bm{b}_{\pi} = (b_{\pi(i)})_{i \in \bbN}$. Notice that $\bm{b}_{\pi} \in \ell^p_{\mathsf{M}}(\bbN)$ since it is monotonically nonincreasing and $\nm{\bm{b}_{\pi}}_{p,\mathsf{M}} = \nm{\bm{b}}_p$. In what follows, we will also abuse notation and write $\pi : \bbR^{\bbN} \rightarrow \bbR^{\bbN}$ for the action of this permutation on a sequence, i.e., $\pi(\bm{c}) = (c_{\pi(i)})_{i \in \bbN}$ for $\bm{c} = (c_i)_{i \in \bbN} \in \bbR^{\bbN}$.

Let $\bm{\alpha}_{\pi} = \pi(\bm{\alpha})$ and $\bm{\beta}_{\pi} = \pi(\bm{\beta})$ and consider the tensor-product Jacobi polynomials \ef{eq:basis} corresponding to $\bm{\alpha}_{\pi}$ and $\bm{\beta}_{\pi}$ and let the weights $\bm{u}_{\pi}$ be as in \ef{u-def}, also for $\bm{\alpha}_{\pi}$ and $\bm{\beta}_{\pi}$. \cf{t:weighted-lp-error}, part (a) with $q=2$ now implies the existence of an index set $S_{\pi} = S_{\pi}(k,\bm{b}_{\pi},\bm{\alpha}_{\pi},\bm{\beta}_{\pi})$ with $|S_{\pi}| \leq |S_{\pi}|_{\bm{u}_{\pi}} \leq k$  such that 
\be{
\label{g-Spi-bound}
\nm{g - g_{S_{\pi}}}_{L^2_{\varrho_{\pi}(\cU ; \cV)}} \leq C(\bm{b}_{\pi},p,\tau) \cdot k^{\frac12-\frac1p},\quad \forall g \in \cH(\bm{b}_{\pi}),
}
where $g_{S_{\pi}}$ is as in \ef{def:truncatedF} with the Jacobi polynomials corresponding to $\bm{\alpha}_{\pi}$ and $\bm{\beta}_{\pi}$ and $\varrho_{\pi} = \varrho_{\bm{\alpha}_{\pi},\bm{\beta}_{\pi}}$ is as in \ef{tensor-jacobi-measure}. Here, we also used Parseval's identity and the fact that $\nms{\cdot}_{2,\bm{u} ; \cV} = \nms{\cdot}_{2;\cV}$. Following this, we now define the set $S= S(k,\bm{b},\bm{\alpha},\bm{\beta})$ by 
\bes{
S = \pi(S_{\pi}) = \{ \pi(\bm{\nu}) : \bm{\nu} \in S_{\pi} \}.
}
Note that the reason for introducing the permutation $\pi$ is to obtain a bound \ef{g-Spi-bound} depending on the monotonic sequence $\bm{b}_{\pi}$. If we worked directly with $\bm{\alpha}$ and $\bm{\beta}$, it would only be possible to derive a bound for $\overline{\theta_{m}}(p,\mathsf{M})$ in Theorem \ref{t:main-res-known}, as opposed to one for $\overline{\theta_m}(p)$.

Having defined $S$, we next construct the reconstruction map.
Analogously to \cite[Eqn.\ (5.17)]{adcock2025near}, given any sample points $\{\bm{y}_i \}^{m}_{i=1} \subset \cU$ we define the matrix $\bm{A}$   by 
\begin{equation}\label{eq:def_A}
\bm{A} = \left(  \dfrac{\Psi^{\bm{\alpha},\bm{\beta}}_{\bnu}(\bm{y_i})}{\sqrt{m}}\right)_{i \in [m], \, \bnu \in S} \in \bbR^{m \times |S|}
\end{equation}
and consider the map $\cT : (\cU  \times \cV)^m \rightarrow L^2_{\varrho}(\cU ; \cV)$
defined by 
\begin{equation}\label{eq:min_LS}
\cT ( \{ (\bm{y}_i , f_i) \}^{m}_{i=1}  ) = \sum_{\bm{\nu} \in S} \hat{c}_{\bm{\nu}} \Psi^{\bm{\alpha},\bm{\beta}}_{\bm{\nu}},\qquad \text{where }\hat{\bm{c}}_{S} = (\hat{c}_{\bm{\nu}})_{\bm{\nu} \in S} \in \argmin{\bm{z} \in \cV^{{|S|}}}  \nm{\bm{A} \bm{z} - m^{-1/2} \bm{f}}_{2;\cV},
\end{equation}
for any $\bm{f} = (f_i)^{m}_{i=1} \in \cV^m$. Note that  $\cT$ can be made well defined by picking the unique minimizer with smallest $\ell^2({S} ; \cV)$-norm. 
Observe that $\cT$ depends on $\bm{b}$, $\bm{\alpha}$, $\bm{\beta}$ and $\epsilon$ only.

\subsection{Proof of Theorem \ref{t:prob-bound-known} }

The map $\cT$ is similar to that used in
\cite[Thm.\ 4.4]{adcock2025near}. However, there are some key differences. First, we now consider arbitrary Jacobi polynomials, which may also have different parameters $\alpha_i$, $\beta_i$ in each variable. The latter necessitates the use of the permutation $\pi$ in the construction of $\cT$ and a more involved argument in the proof below. Conversely, \cite[Thm.\ 4.4]{adcock2025near} only considers tensor-product Chebyshev or Legendre polynomials. Second, in Theorem \ref{t:prob-bound-known} we establish a stronger uniform guarantee that holds simultaneously for all functions $f \in \cH(\bm{b})$, whereas \cite[Thm.\ 4.4]{adcock2025near} is nonuniform and therefore weaker. Its proof consequently requires a different approach to that of \cite[Thm.\ 4.4]{adcock2025near}. This is based on ideas found in \cite{krieg2021function,krieg2021functionII}, and involves expanding the error in a sequence of nested index sets. See Remark \ref{rem:uniform-vs-nonuniform-known} for some further discussion on this point.

\begin{proof}[Proof of Theorem \ref{t:prob-bound-known}]
Consider $\cT$ as defined above and sample points $\bm{y}_1,\ldots,\bm{y}_m \sim_{\mathrm{i.i.d.}} \varrho$. We divide the proof into several steps.

\pbk
\textit{Step 1: Construction of nested index sets.}
We start with the construction of a specific nested collection of index sets.  
Let $k_0 =k$, $S_0 = S$ be as defined above and set $k_j=2^j k$ for $j= 1, 2,\ldots $. We once more apply \cf{t:weighted-lp-error} with parameters $\bm{\alpha}_{\pi}$ and $\bm{\beta}_{\pi}$ and weights $\bm{u}_{\pi}$. This implies that there is a sequence of nested index sets 
\bes{
S_{0,\pi} \subseteq S_{1,\pi} \subseteq S_{2,\pi} \subseteq \cdots 
}
satisfying $|S_{j,\pi}|_{\bm{u}_{\pi}} \leq k_j$ and
\be{
\label{Sjpi-errors}
\nm{g - g_{S_{j,\pi}}}_{L^2_{\varrho_{\pi}}(\cU ; \cV)} \leq C(\bm{b}_{\pi},p,\tau)  \cdot (k_j)^{1/2-1/p},\quad \forall g \in \cH(\bm{b}_{\pi})
}
and all $j = 1, 2, \ldots $, where $g_{S_{j,\pi}}$ is as in \ef{def:truncatedF} with $f = g$ and the Jacobi polynomials corresponding to $\bm{\alpha}_{\pi}$ and $\bm{\beta}_{\pi}$, and $\varrho_{\pi} = \varrho_{\bm{\alpha}_{\pi},\bm{\beta}_{\pi}}$ is as in \ef{tensor-jacobi-measure}. Given this, we define the new nested index sets
\bes{
S_0 \subseteq S_1 \subseteq S_{2} \subseteq \cdots 
}
by $S_j = \pi(S_{j,\pi})$, $j = 1,2,\ldots$. Let $\bm{u}$ be as in \ef{u-def} for parameters $\bm{\alpha}$ and $\bm{\beta}$. Observe that
\be{
\label{tensor-jacobi-permute}
\Psi^{\bm{\alpha},\bm{\beta}}_{\pi(\bm{\nu})}(\bm{y}) = \Psi^{\bm{\alpha}_{\pi} , \bm{\beta}_{\pi} }_{\bm{\nu}}(\pi(\bm{y})),\quad \forall \bm{\nu} \in \cF,\ \bm{y} \in \cU.
}
Therefore
\bes{
|S_j|_{\bm{u}} = \sum_{\bm{\nu} \in S_j} u^2_{\bm{\nu}} = \sum_{\bm{\nu} \in S_{j,\pi}} u^2_{\pi(\bm{\nu})} = \sum_{\bm{\nu} \in S_{j,\pi}} \nm{\Psi^{\bm{\alpha},\bm{\beta}}_{\pi(\bm{\nu})}}^2_{L^{\infty}(\cU)} = \sum_{\bm{\nu} \in S_{j,\pi}}\nm{\Psi^{\bm{\alpha}_{\pi},\bm{\beta}_{\pi}}_{\bm{\nu}}}^2_{L^{\infty}(\cU)} = |S_{j,\pi} |^2_{\bm{u}_{\pi}}.
}
We deduce that $|S_j|_{\bm{u}} \leq k_j$ for all $j = 0,1,2,\ldots$.

\pbk
\textit{Step 2: Estimating the truncation errors.}
Now let $f \in \cH(\bm{b})$ be arbitrary. Consider the region $\cR(\bm{b},\varepsilon)$ defined by \ef{def:b-eps-holo}. Since $\pi(\cE_{\bm{\rho}}) = \cE_{\pi(\bm{\rho})}$ for any Bernstein polyellipse and $\sum^{\infty}_{i=1} ((\rho_{\pi(i)} + \rho^{-1}_{\pi(i)})/2-1) b_{\pi(i)} = \sum^{\infty}_{i=1} ((\rho_i + \rho^{-1}_i)/2 - 1) b_i$, we have that $\pi(\cR(\bm{b},\varepsilon)) = \cR(\bm{b}_{\pi},\varepsilon)$. Therefore
\be{
\label{Hb-permute}
f \in \cH(\bm{b}) \quad \Longleftrightarrow \quad f \circ \pi \in \cH(\bm{b}_{\pi}). 
}
Write $g = f \circ \pi \in \cH(\bm{b}_{\pi})$.
We now relate the truncated Jacobi polynomial expansions of $f$ and $g$ corresponding to the  parameters $\bm{\alpha},\bm{\beta}$ and $\bm{\alpha}_{\pi},\bm{\beta}_{\pi}$, respectively. Observe that
\bes{
f_{S_j} = \sum_{\bm{\nu} \in S_j} c^{\bm{\alpha},\bm{\beta}}_{\bm{\nu}} \Psi^{\bm{\alpha},\bm{\beta}}_{\bm{\nu}} = \sum_{\bm{\nu} \in S_{j,\pi}} c^{\bm{\alpha},\bm{\beta}}_{\pi(\bm{\nu})} \Psi^{\bm{\alpha},\bm{\beta}}_{\pi(\bm{\nu})}, 
}
where the coefficients $c^{\bm{\alpha},\bm{\beta}}_{\bm{\nu}}$ are as in \ef{f-coeff}. This and \ef{tensor-jacobi-permute} give that
\bes{
c^{\bm{\alpha},\bm{\beta}}_{\pi(\bm{\nu})} = \int_{\cU} f(\bm{y}) \Psi^{\bm{\alpha},\bm{\beta}}_{\pi(\bm{\nu})}(\bm{y}) \D \varrho_{\bm{\alpha},\bm{\beta}}(\bm{y}) = \int_{\cU}f(\pi(\bm{y})) \Psi^{\bm{\alpha}_{\pi},\bm{\beta}_{\pi}}_{\bm{\nu}}(\bm{y}) \D \varrho_{\bm{\alpha}_{\pi} , \bm{\alpha}_{\pi}}(\bm{y}) = d^{\bm{\alpha}_{\pi},\bm{\beta}_{\pi}}_{\bm{\nu}},
}
where $d^{\bm{\alpha}_{\pi},\bm{\beta}_{\pi}}_{\bm{\nu}}$ are the coefficients of $g$ with respect to the Jacobi polynomials with indices $\bm{\alpha}_{\pi}$, $\bm{\beta}_{\pi}$. This gives
\bes{
f_{S_j}(\bm{y}) = \sum_{\bm{\nu} \in S_{j,\pi}} d^{\bm{\alpha}_{\pi},\bm{\beta}_{\pi}}_{\bm{\nu}} \Psi^{\bm{\alpha}_{\pi},\bm{\beta}_{\pi}}_{\bm{\nu}}(\pi(\bm{y})) = g_{S_j,\pi}(\pi(\bm{y})).
}
Hence, by \ef{Sjpi-errors}, we get
\be{
\label{Sj-errors}
\nm{f - f_{S_j}}_{L^2_{\varrho}(\cU ; \cV)} = \nm{g - g_{S_j,\pi}}_{L^2_{\varrho_{\pi}}(\cU ; \cV)} \leq C(\bm{b}_{\pi},p,\tau) \cdot (k_j)^{1/2-1/p}.
}

\pbk
\textit{Step 3: Error bound for the reconstruction map.} Consider the reconstruction  $\hat{f} = \cT (  \{ (\bm{y}_i , f(\bm{y}_i)) \}^{m}_{i=1} )$ of $f$. Note that, upon squaring the objective function, \eqref{eq:min_LS} is equivalent to an algebraic least-squares problem.  Hence, basic analysis of least-squares problems (see, e.g., the proof of Theorem 5.3 in \cite{adcock2022sparse}) implies that
\be{
\label{LS-error-main}
\nm{f - \hat{f}}_{L^2_{\varrho}(\cU ; \cV)} \leq \nm{f - f_S}_{L^2_{\varrho}(\cU ; \cV)} + {\tau}^{-1} \nm{f - f_S}_{\mathsf{disc}},
}
where ${\tau} = \sigma_{\min}(\bm{A})$ is the minimum singular value of $\bm{A}$, for $\bm{A}$  as in  \eqref{eq:def_A},   and $\nms{\cdot}_{\mathsf{disc}}$ is the discrete seminorm given by $\|g\|_{\mathsf{disc}}^2 = m^{-1} \sum_{i=1}^m \|g(\bm{y}_i)\|_{\cV}^2$ for any $g$ that is finite at the sample points $\bm{y}_i$.
We now expand the term $\nm{f - f_S}_{\mathsf{disc}}$ using the previously defined sets, along with the triangle inequality to get
\begin{equation}
\label{error-split-1-LS}
\nm{f - \hat{f}}_{L^2_{\varrho}(\cU ; \cV)} \leq \nm{f - f_S}_{L^2_{\varrho}(\cU ; \cV)} + {\tau}^{-1} \sum^{\infty}_{j=1} \nm{f_{S_j \backslash S_{j-1}}}_{\mathsf{disc}} ,
\end{equation}
where, for $j = 1,2,\ldots$, $f_{S_j \backslash S_{j-1}}$ is the truncated expansion of $f$ corresponding to the index set $S_j \backslash S_{j-1}$ and with respect to Jacobi polynomials with indices $\bm{\alpha},\bm{\beta}$.
Following \cite{krieg2021functionII}, we next define the matrices
\begin{equation*}
\bm{A}^{(j)} = \left ( \frac{1}{\sqrt{m}} \Psi_{\bm{\nu}}(\bm{y}_i) \right )_{i \in [m],\ \bm{\nu} \in S_j \backslash S_{j-1}}, \quad j=1,2,\ldots . 
\end{equation*}
Hence
\begin{equation*}
\nm{f_{S_j \backslash S_{j-1}}}_{\mathsf{disc}} = \nm{\bm{A}^{(j)} \bm{c}_{S_j \backslash S_{j-1}}}_{2;\cV} \leq {\gamma}_j \nm{\bm{c}_{S_j \backslash S_{j-1}}}_{2;\cV} ,\quad j=1,2,\ldots , 
\end{equation*}
where $\gamma_j=\sigma_{\max}(\bm{A}^{(j)})$ is the maximum singular value of $\bm{A}^{(j)}$ for $j=1,2\ldots$ and $\bm{c} = \bm{c}^{\bm{\alpha},\bm{\beta}}$ are the coefficients of $f$ as in \ef{f-coeff}. By Parseval's identity and the construction of the sets, we get
\begin{equation*}
\nm{f_{S_j \backslash S_{j-1}}}_{\mathsf{disc}} \leq {\gamma}_j \nm{f_{S_j \backslash S_{j-1}}}_{L^2_{\varrho}(\cU ; \cV)} \leq {\gamma}_j  \nm{f - f_{S_{j-1}}}_{L^2_{\varrho}(\cU ; \cV)}, \quad \quad j=1,2,\ldots ,
\end{equation*}
and therefore
\begin{equation*}
\nm{f - \hat{f}}_{L^2_{\varrho}(\cU ; \cV)} \leq \nm{f - f_S}_{L^2_{\varrho}(\cU ; \cV)} + {\tau}^{-1} \sum^{\infty}_{j=1} {\gamma}_j \nm{f - f_{S_{j-1}}}_{L^2_{\varrho}(\cU ; \cV)}.
\end{equation*}
Combining this with \ef{Sj-errors} and using the fact that $k_j = 2^j k$, we deduce that
\begin{equation}
\label{error-split-2-LS}
\nm{f - \hat{f}}_{L^2_{\varrho}(\cU ; \cV)} \leq C(\bm{b}_{\pi},p,\tau)  \cdot k^{\frac12-\frac1p} \cdot \left ( 1 + {\tau}^{-1} \sum^{\infty}_{j=1} {\gamma}_j (2^j)^{\frac12-\frac1p} \right ).
\end{equation}

\pbk
\textit{Step 4: Estimating the constants ${\tau}$ and ${\gamma}_j$, $j = 1,2,\ldots$.} The next step is to estimate these constants. We do this via the matrix Chernoff bound \cite[Thm.~1.1]{tropp2012user-friendly}.

Note that $\bm{A}$, as defined in \eqref{eq:def_A}, is a random matrix with $\bbE(\bm{A}^* \bm{A}) = \bm{I}$. We may write $\bm{A}^* \bm{A}$ as a sum of i.i.d.\ self-adjoint, nonnegative definite random matrices
\begin{equation*}
\bm{A}^*\bm{A} = \sum^{m}_{i=1} \bm{X}_i,\qquad \text{where }\bm{X}_i = \left ( \frac{1}{m} {\overline{\Psi_{\bm{\nu}}(\bm{y}_i)}} \Psi_{\bm{\nu}'}(\bm{y}_i) \right )_{\bm{\nu} , \bm{\nu}' \in S }, \quad {i=1,\ldots m}.
\end{equation*}
Observe that, for $\bm{z} = (z_{\bm{\nu}})_{\bm{\nu} \in S} \in \bbC^s$, the Cauchy--Schwarz inequality and definition \ef{u-def} of the weights $\bm{u}$ gives
\begin{equation*}
\bm{z}^* \bm{X}_i \bm{z} = \frac1m \left | \sum_{\bm{\nu} \in S} z_{\bm{\nu}} \Psi_{\bm{\nu}}(\bm{y}_i) \right |^2 \leq \frac{|S|_{\bm{u}}}{m} \nm{\bm{z}}^2_{2}
\end{equation*}
almost surely.
Therefore, $\lambda_{\max}(\bm{X}_i) \leq |S|_{\bm{u}} /m \leq k / m$ almost surely. An application of the matrix Chernoff bound now gives that
\begin{equation}
\label{alpha-event-prob}
\bbP ( {\tau} \leq 1/2 ) \leq k \exp \left ( - 0.4 m / k \right ).
\end{equation}
Next, we consider the matrices $\bm{A}^{(j)}$, $j = 1,2,\ldots$. By the same arguments, except with the bound $k/m$ replaced by $k_j / m$, we deduce that
\begin{equation*}
\bbP({\gamma}_j \geq \sqrt{1+\delta} ) \leq k_j \exp \left ( - \frac{(1+\delta)\log(1+\delta) - \delta}{k_j / m} \right ), \quad j=1,2,\ldots,
\end{equation*}
for all $\delta \geq 0$.
Observe that $(1+\delta)\log(1+\delta) - \delta \geq \delta \log(\delta) / 2$ for all $\delta \geq 1$. Therefore
\begin{equation}
\label{beta-event-prob}
\bbP({\gamma}_j \geq \sqrt{1+\delta} ) \leq k_j \exp \left ( - \frac{\delta \log(\delta) m}{2 k_j} \right ),\quad \forall \delta \geq 1, \ j = 1,2,\ldots.
\end{equation}

\pbk
\textit{Step 5: Parameter choices and final arguments.} We now define
\begin{equation*}
\delta_j = 2^j,\quad \epsilon_j  = \epsilon / 2^{j+1},\quad j = 1,2,\ldots,
\end{equation*}
and consider the event
\begin{equation*}
E = E_0 \cap E_1 \cap E_2 \cap \cdots,
\end{equation*}
where $E_0 = \{ \tau > 1/2 \}$ and $E_j = \{ \gamma_j <  \sqrt{1+\delta_j} \}$ for $j = 1,2,\ldots$. 
We first show that $\bbP(E) \geq 1-\epsilon$. Using \ef{alpha-event-prob} we see that, choosing $c\geq 5/2$ in \ef{s-def-LS},
\begin{equation*}
\bbP(E^c_0) \leq k \exp \left ( - 0.4 c \log(m/\epsilon) \right ) \leq k \epsilon / m \leq \epsilon / 2.
\end{equation*}
In the last inequality we used the fact that $m \geq 3$, so that $\log(m/\epsilon) \geq 1$ and therefore $k \leq m/2$. Next, using \ef{beta-event-prob} with $\delta = \delta_j$ and recalling that $k_j = 2^j k$, we see that
\begin{equation*}
\bbP(E^c_j)  \leq 2^j k \exp \left ( - \frac{j 2^{j} m \log(2)}{2^{j+1} k} \right ) \leq 2^j m (m/\epsilon)^{-c j\log(\sqrt{2})} \leq \epsilon / 2^{j+1}, \quad j=1,2,\ldots
\end{equation*}
for $c \geq 12$.
Here we also used the facts that $k \leq m$ and $m \geq 3$. Combining this with the bound for $\bbP(E^c_0)$ and the union bound, we deduce that
\begin{equation*}
\bbP(E^c) \leq \bbP(E^c_0) + \bbP(E^c_1) + \bbP(E^c_2) + \cdots  \leq \epsilon,
\end{equation*}
as required. Applying \ef{error-split-2-LS}, we deduce that
\begin{equation*}
\nm{f - \hat{f}}_{L^2_{\varrho}(\cU ; \cV)} \leq C(\bm{b}_{\pi},p,\tau)  \cdot k^{\frac12-\frac1p}  \cdot  \left ( 1 + 2 \sum^{\infty}_{j=1} \sqrt{1+2^j} (2^j)^{\frac12-\frac1p}  \right ) ,
\end{equation*}
with probability at least $1-\epsilon$. Since the sum converges (recall that $p < 1$) and the event $E$ is independent of $f \in \cH(\bm{b})$, it follows that
\begin{equation*}
\nm{f - \hat{f}}_{L^2_{\varrho}(\cU ; \cV)} \leq C(\bm{b}_{\pi},p,\tau) \cdot k^{\frac12-\frac1p} ,\quad \forall f \in \cH(\bm{b}),
\end{equation*}
with probability at least $1-\epsilon$. This completes the proof.
\end{proof}

\rem{
\label{rem:darned-constants-known}
Note that we have in fact shown that the constant $C(\bm{b}_{\pi},p,\tau)$ in Theorem \ref{t:prob-bound-known} can be taken as $C_0(p) \cdot C_1(\bm{b}_{\pi},p,\tau)$, where $C_0(p) > 0$ depends on $p$ only and $C_1(\bm{b}_{\pi},p,\tau)$ is the constant from \cf{t:weighted-lp-error}, part (a) with $q = 2$ and applied with parameters $\bm{\alpha}_{\pi}$, $\bm{\beta}_{\pi}$ and measure $\varrho_{\bm{\alpha}_{\pi},\bm{\beta}_{\pi}}$. We will use this observation later in the second part of the proof of Theorem \ref{t:main-res-known}.
}

\rem{
\label{rem:uniform-vs-nonuniform-known}
The proofs of \cite[Thm.\ 4.4]{adcock2025near} and Theorem \ref{t:prob-bound-known} diverge in how they bound the discrete error in \ef{LS-error-main}. The former uses the fact that $\nm{f - f_S}^2_{\mathsf{disc}}$ is a sum of independent random variables with mean $\nm{f -  f_{S}}^2_{L^{2}_{\varrho}(\cU ; \cV)}$ to provide a nonuniform bound via Bernstein's inequality. Conversely, the proof of Theorem \ref{t:prob-bound-known} uses the splitting \ef{error-split-1-LS} and repeated applications of the matrix Chernoff inequality to derive a stronger uniform bound.
}

\subsection{Reconstruction map: unknown anisotropy}\label{S:reconstruction}
In this section we consider the case where $\bm{b}$ is unknown. Here we need a different approach to the proof of Theorem~\ref{t:prob-bound-known}, where the construction of $\cT$ relied heavily on a fixed $\bm{b}$. Instead, we now build the reconstruction map $\cT$  based on the solution of a compressed sensing problem. Since it presents few additional difficulties, we now describe the reconstruction map that achieves the factor $\log^4(m) g(m)$ described in Remark \ref{rem:log-term}, where $g$ is any function $g : \bbN \rightarrow [1,\infty)$ with $g(m) \rightarrow +\infty$ subalgebraically fast as $m \rightarrow + \infty$. Theorems \ref{t:main-res-unknown} and \ref{t:prob-bound-unknown} then follow as special cases with $g(m) = \log(m)$.

Let $\cU=[-1,1]^{\bbN}$ and $\cV$ be a Hilbert space, $\bm{\alpha},\bm{\beta}$ satisfy Assumption \ref{main-ass}  and $\varrho = \varrho_{\bm{\alpha},\bm{\beta}}$ be the probability measure in \eqref{tensor-jacobi-measure}. 
We follow the setup of \cite{adcock2025near}, except that we consider tensor Jacobi polynomials instead of just Chebyshev or Legendre polynomials. Let
\be{
\label{k-def}
k = \frac{m}{c L},\quad \text{where }L = L(m,\epsilon) = \log^{4}(m) g(m) + \log(\epsilon^{-1})
}
and $c \geq 1$ is a universal constant that will be chosen later in the proof. Given $k$, let
\be{
\label{lambda-def}
n = \left \lceil k^{\sqrt{g(m)} }\right \rceil, \quad  \lambda = \frac{3}{7 \sqrt{k}}
}
and define the finite set
\begin{equation}
\label{HC_index_set_inf}
\Lambda = \Lambda^{\mathsf{HCI}}_{n} = \left \{ \bm{\nu} = (\nu_k)^{\infty}_{k=1} \in \cF : \prod^n_{k=1} (\nu_k + 1) \leq n\text{ and $\nu_k = 0$ for $k > n$} \right \} \subset \cF.
\end{equation}
A key property of this set is that it contains every anchored set (see  \S\ref{S:ellspaces}) of size at most $n$ (see, e.g., \cite[Prop.~2.18]{adcock2022sparse}). Write $N = |\Lambda|$, let
$\{ \bm{y}_i \}^{m}_{i=1} \subset \cU$ be any sample points and consider the matrix
\begin{equation*}
\bm{A} = \left(  \dfrac{\Psi^{\bm{\alpha},\bm{\beta}}_{\bnu}(\bm{y}_i)}{\sqrt{m}}\right)_{i \in [m], \bnu \in \Lambda} \in \bbR^{m \times N}.
\end{equation*}
Let $\bm{u}$ be the weights given {by \eqref{u-def}} for $\Psi_{\bm{\nu}} = \Psi^{\bm{\alpha},\bm{\beta}}_{\bnu}$. We now define the reconstruction map $\cT : (\cU \times \cV)^m \rightarrow L^2_{\varrho}(\cU ; \cV)$ as
\begin{equation}\label{eq:def_Tunkn}
\begin{split}
&\cT ( \{ (\bm{y}_i,f_i) \}^{m}_{i=1}  )  = \sum_{\bm{\nu} \in \Lambda} \hat{c}_{\bm{\nu}} \Psi_{\bm{\nu}},
\\
& \text{where }\hat{\bm{c}}_{\Lambda} = (\hat{c}_{\bm{\nu}})_{\bm{\nu} \in \Lambda} \in \argmin{\bm{z} \in \cV^N} ~ \Big\{\lambda \nm{\bm{z}}_{1,\bm{u} ; \cV} + \nm{\bm{A} \bm{z} - m^{-1/2} \bm{f}}_{2;\cV}\Big\}
 \end{split}
\end{equation}
and $\bm{f} = (f_i)^{m}_{i=1} \in \cV^m$.
Note that this is not well defined at present, since the optimization problem may have multiple minimizers. We can make it well defined by picking the unique minimizer with smallest $\ell^2(\Lambda ; \cV)$-norm.

\subsection{Proof of Theorem \ref{t:prob-bound-unknown} }
The proof involves modifying the proof of \cite[Thm.\ 4.2]{adcock2025near}. A significant difference between these results is that
the latter provides a nonuniform guarantee that holds for a fixed function $f$. Moreover, \cite[Thm.\ 4.2]{adcock2025near} pertains to Chebyshev and Legendre polynomials only. In contrast, we now establish a uniform error bound with high probability that holds for arbitrary Jacobi polynomials. See Remark \ref{rem:uniform-vs-nonuniform} for some further discussion.

As we see next, in the proof of Theorem \ref{t:prob-bound-unknown} we need to bound the distance $\nm{\bm{x} - \bm{z}}_{2;\cV}$ between two certain vectors $\bm{x},\bm{z} \in \cV^N$. Such a bound can be obtained from the well known \textit{weighted robust Null Space Property (rNSP)} (see \cite[Defn. 6.22]{adcock2022sparse} or \cite[Sec.~4.1]{rauhut2016interpolation}). A  bounded linear operator $\bm{A}:\cV^N \rightarrow \cV^m$ has the weighted rNSP over $\cV$ of order $(k,\bm{u})$ with constants $0\leq \rho <1$ and $\gamma\geq 0$ if  
\begin{equation*}
\nm{\bm{x}_{S}}_{2;\cV} \leq \frac{\rho \nm{\bm{x}_{S^c}}_{1,\bm{u} ; \cV}}{\sqrt{k}} + \gamma \nm{\bm{A} \bm{x}}_{2;\cV},\quad \forall \bm{x} \in \cV^N,
\end{equation*}
for any $S \subseteq [N]$ with $|S|_{\bm{u}} \leq k$. 

\begin{proof}[Proof of Theorem \ref{t:prob-bound-unknown} ]
We divided the proof into three steps.

\pbk
\textit{Step 1: Establishing the rNSP.} Observe that the matrix $\bm{A}$ in \ef{eq:def_A} extends to a bounded linear operator from $\cV^N$ to $\cV^m$ in the obvious way. We first show that $\bm{A}$ has the weighted rNSP of order $(k,\bm{u})$ with $\rho = 1/2$ and $\gamma = 3/2$, where $k >0$ is as in \ef{k-def}.  By \cite[Lem.\ 6.3]{adcock2025near}, it suffices to show that $\bm{A} \in \bbR^{m \times N}$ has the weighted RIP over $\bbR$ of order $(2k,\bm{u})$ with constant $\delta = 1/(4 \sqrt{2}+1)$. To show this, we apply \cite[Lem.\ 6.4]{adcock2025near}. Notice that this lemma is stated for Chebyshev and Legendre polynomials only, but it in fact applies to any orthonormal system of functions. This lemma asserts that the desired weighted RIP holds with probability at least $1-\epsilon$, provided
\be{
\label{m-cond-wRIP}
m \geq c_0 \cdot \delta^{-2} \cdot k \cdot \left ( \log^2(k/\delta) \cdot \log^2(\E n) + \log(2/\epsilon) \right ),
}
for some universal constant $c_0$, where $n$ is as in \ef{lambda-def}. Recall from \ef{k-def} that $k \leq m$ since $m \geq 3$ (which implies that $L \geq 1$) and $c \geq 1$ by assumption. This implies that $n \leq 2 m^{\sqrt{g(m)}}$. Hence, with $\delta = 1/(4 \sqrt{2}+1)$, the right-hand side of the above expression satisfies
\bes{
c_0 \cdot \delta^{-2} \cdot k \cdot \left ( \log^2(k/\delta) \cdot \log^2(\E n) + \log(2/\epsilon) \right ) \lesssim k \cdot \left (\log^{4}(m) g(m) + \log(1/\epsilon) \right ) = k \cdot L.
}
Therefore, \ef{m-cond-wRIP} holds for a sufficiently-large choice of the universal constant $c$ in \ef{k-def}.

\pbk
\textit{Step 2: Error bound.} Let $\bm{b} \in \ell^p_{\mathsf{M}}(\bbN)$ for some $0 < p <1$, $f \in \cH(\bm{b})$ with expansion \ef{f-exp} and coefficients $\bm{c} = \bm{c}^{\bm{\alpha},\bm{\beta}}$, and consider its reconstruction $\hat{f} = \cT(\{ (\bm{y}_i  ,  f(\bm{y}_i) ) \}^{m}_{i=1} )$. Then Parseval's identity implies that
\begin{equation*}
\nm{f - \hat{f}}_{L^2_{\varrho}(\cU ; \cV)} \leq \nm{f - f_{\Lambda}}_{L^2_{\varrho}(\cU ; \cV)} + \nm{\bm{c}_{\Lambda} - \hat{\bm{c}}_{\Lambda}}_{2;\cV},
\end{equation*}
where $f_{\Lambda}$ is as in \ef{def:truncatedF}, $\bm{c}_{\Lambda} = (c_{\bm{\nu}})_{\bm{\nu} \in \Lambda} \in \cV^N$ and $\hat{\bm{c}}_{\Lambda}$ is as in \ef{eq:def_Tunkn}. Observe that
\begin{equation}
\label{A-c-etilde}
\bm{A} \bm{c}_{\Lambda} + \widetilde{\bm{e}} = \frac{1}{\sqrt{m}} (f(\bm{y}_i))^{m}_{i=1},\quad \text{where }
\widetilde{\bm{e}} = \dfrac{1}{\sqrt{m}} (f(\bm{y}_i)-f_{\Lambda}(\bm{y}_i))_{i=1}^m.
\end{equation}
Recall that, with probability at least $1-\epsilon$, $\bm{A}$ has the weighted rNSP of order $(k,\bm{u})$ with $\rho = 1/2$ and $\gamma = 3/2$. We now apply \cite[Lem.\ 7.4]{adcock2024efficient} with $\bm{w} = \bm{u}$, the given values of $\rho$ and $\gamma$, $\bm{x} = \bm{c}_{\Lambda}$, $\widetilde{\bm{x}} = \hat{\bm{c}}_{\Lambda}$, $\bm{b} = m^{-1/2} (f(\bm{y}_i))^{m}_{i=1}$ and $\bm{e} = -\widetilde{\bm{e}}$. Since $\hat{\bm{c}}_{\Lambda}$ is a minimizer of \ef{eq:def_Tunkn} and $\lambda$ is defined in \ef{lambda-def} so that
\bes{
\lambda = \frac{3}{7 \sqrt{k}} = \frac{(1+\rho)^2}{(3+\rho) \gamma \sqrt{k}}
}
this lemma gives
\bes{
\nm{\bm{c}_{\Lambda}-\hat{\bm{c}}_{\Lambda}}_{2;\cV} \lesssim \frac{\sigma_k(\bm{c}_{\Lambda})_{1,\bm{u};\cV}}{\sqrt{k}} +\nm{\widetilde{\bm{e}}}_{2;\cV}.
}
Recall \ef{A-c-etilde} and notice that
\begin{equation}
\label{widetilde-e-bd}
\nm{\widetilde{\bm{e}}}_{2;\cV} \leq \nm{f - f_{\Lambda}}_{L^{\infty}(\cU ; \cV)} \leq \nm{\bm{c}-\bm{c}_{\Lambda}}_{1,\bm{u} ; \cV}.
\end{equation}
Here we used the definition \ef{u-def} of the weights $\bm{u}$. Using this and the previous inequalities, we deduce that
\be{
\label{error-bound-1}
\nm{f - \hat{f}}_{L^2_{\varrho}(\cU ; \cV)} \lesssim \frac{\sigma_k(\bm{c}_{\Lambda})_{1,\bm{u} ; \cV}}{\sqrt{k}}   +\nm{\bm{c}-\bm{c}_{\Lambda}}_{1,\bm{u} ; \cV}.
}
Since $\bm{b} \in \ell^p_{\mathsf{M}}(\bbN)$, part (a) of Theorem \ref{t:weighted-lp-error} with $q = 1$ implies that
\begin{equation*}
{\sigma_k(\bm{c}_{\Lambda})_{1,\bm{u} ; \cV} \leq \sigma_k(\bm{c})_{1,\bm{u} ; \cV}  \leq
C_1(\bm{b},p,\tau) \cdot k^{1-\frac1p},}
\end{equation*}
where  $C_1(\bm{b},p,\tau)$ is the corresponding constant. 
Also, part (a) of \cf{t:weighted-lp-error-monotone} with $q=1$ and the fact that $\Lambda$ contains all anchored sets of size at most $n$ (see, e.g., \cite[Prop.\ 2.18]{adcock2022sparse}) imply that there exists an anchored set $S \subset \Lambda$ with $|S| \leq n$ such that
\begin{equation*}
\nm{\bm{c}-\bm{c}_{\Lambda}}_{1,\bm{u} ; \cV} \leq {
\nm{\bm{c}-\bm{c}_{S}}_{1,\bm{u} ; \cV} }\leq C_{2}(\bm{b},p,\tau) \cdot n^{1-\frac1p},
\end{equation*}
where  $C_{2}(\bm{b},p,\tau)$ is the corresponding constant.
This gives
\begin{equation*}
\nm{f - \hat{f}}_{L^2_{\varrho}(\cU ; \cV)} \lesssim C_1(\bm{b},p,\tau) \cdot k^{\frac12-\frac1p} + C_{2}(\bm{b},p,\tau) \cdot n^{1-\frac1p},
\end{equation*}
with probability at least $1-\epsilon$, and up to a possible change in the constants $C_1$ and $C_2$ by constants independent of $\bm{b}$.

\pbk
\textit{Step 3: Final arguments.} We now divide into two cases, $k < 1$ and $k \geq 1$. Suppose first that $k < 1$. Then this and the fact that $n \geq 1$ imply that
\bes{
\nm{f - \hat{f}}_{L^2_{\varrho}(\cU ; \cV)} \leq C(\bm{b},p,\tau) (m/L)^{1/2-1/p},
}
where $C(\bm{b},p,\tau) = C_1(\bm{b},p,\tau) + C_2(\bm{b},p,\tau)$, with $C_i(\bm{b},p,\tau)$, $i = 1,2$, as above. Now suppose that $k \geq 1$.
Recall the definition of $n$ in \ef{lambda-def}. Given $0 < p <1$, there exists an $m_0(p,g)$ such that $\sqrt{g(m)} \geq \frac{1/p-1/2}{1/p-1}$ for all $m \geq m_0(p,g)$. Therefore
\bes{
n \geq k^{\frac{1/p-1/2}{1/p-1}},\quad \forall m \geq m_0(p,g).
}
This, the definition of $k$ and the facts that $n \geq 1$ and $k \geq 1$ give that
\bes{
\nm{f - \hat{f}}_{L^2_{\varrho}(\cU ; \cV)} \leq C(\bm{b},p,\tau) \cdot \begin{cases} \left ( \frac{m}{L} \right )^{\frac12-\frac1p} & m \geq m_0(p,g) \\ 2 & 3 \leq m < m_0(p,g) \end{cases} \leq C(\bm{b},p,g,\tau) \cdot \left ( \frac{m}{L} \right )^{\frac12-\frac1p}.
}
Combining the two cases $k < 1$ and $k \geq 1$ we deduce that, for any $m \geq 3$, 
\bes{
\nm{f - \hat{f}}_{L^2_{\varrho}(\cU ; \cV)} \leq C(\bm{b},p,g,\tau) \cdot \left ( \frac{m}{L} \right )^{\frac12-\frac1p}.
}
Note that the probabilistic event which leads to this error bound (i.e., the event that the rNSP holds) is independent of $f$, $\bm{b}$ and $p$. Therefore, this bound holds simultaneously for all $f \in \cH(\bm{b})$, $\bm{b} \in \ell^p_{\mathsf{M}}(\bbN)$,  $\bm{b} \geq \bm{0}$ and $0 < p < 1$, with probability at least $1-\epsilon$.
This completes the proof.
\end{proof}

\begin{remark}
\label{rem:darned-constants}
Much as in Remark \ref{rem:darned-constants-known}, we have in fact shown that the constant $C(\bm{b},p,g,\tau)$ in Theorem \ref{t:prob-bound-unknown} has the form
\bes{
C(\bm{b},p,g,\tau) = C_0(p,g) \cdot \left ( C_1(\bm{b},p,\tau) + C_2(\bm{b},p,\tau) \right ),
}
where $C_0(p,g)$ depends on $p$ and $g$ only,
$C_1(\bm{b},p,\tau)$ is the constant of Theorem \ref{t:weighted-lp-error}, part (a) with $q = 1$ and $C_2(\bm{b},p,\tau)$ is the constant of Theorem \ref{t:weighted-lp-error-monotone}, part (a) with $q = 1$. We will use this in the second part of the proof of Theorem \ref{t:main-res-unknown} below. 
\end{remark}

\rem{
\label{rem:uniform-vs-nonuniform}
In order to obtain a uniform guarantee, in the above proof we bound the term $\nm{\widetilde{\bm{e}}}_{2;\cV}$ in \ef{widetilde-e-bd} using the crude, but deterministic bound $\nm{\widetilde{\bm{e}}}_{2;\cV} \leq \nm{f - f_{\Lambda}}_{L^{\infty}(\cU ; \cV)} $. In contrast, the proof of \cite[Thm.\ 4.2]{adcock2025near} uses the fact that $\nm{\widetilde{\bm{e}}}^2_{2;\cV}$ is, much as in Remark \ref{rem:uniform-vs-nonuniform-known}, a sum of independent random variables with mean $\nm{f -  f_{\Lambda}}^2_{L^{2}_{\varrho}(\cU ; \cV)}$ in order to provide a sharper probabilistic bound using Bernstein's inequality. The advantage of this approach is that it allows for the smaller choice $n \approx m / (c_0 L)$ in Step 2. Conversely, in the above proof, to compensate for the cruder bound \ef{widetilde-e-bd} we need to choose $n$ larger, as in \ef{lambda-def}. This has a practical downside, since the computational cost of implementing the reconstruction map $\cT$ scales linearly in $N = |\Lambda| = |\Lambda^{\mathsf{HCI}}_n|$ (recall that the optimization problem in \ef{eq:def_Tunkn} is over $\cV^N$). It is currently unclear whether uniform bounds with the same rate of convergence as in Theorem \ref{t:prob-bound-unknown} could be obtained with a smaller choice of $n$.
}

\section{Proofs of Theorems \ref{t:main-res-known} and \ref{t:main-res-unknown}}\label{S:proofs_main}

\begin{proof}
[Proof of Theorem \ref{t:main-res-known}]
Let
\be{
\label{eps-def-known}
\epsilon = m^{\frac12-\frac1p} < 1,
}
and consider the reconstruction map asserted by Theorem \ref{t:prob-bound-known} with this value of $\epsilon$, which we denote by $\cT_0$. Define the contraction $\cC : L^2_{\varrho}(\cU ; \cV) \rightarrow L^2_{\varrho}(\cU ; \cV)$ by
\begin{equation*}
\cC(g) = \min \left \{ 1 , \frac{1}{\nm{g}_{L^2_{\varrho}(\cU ; \cV)}} \right \} g,
\end{equation*}
and consider the reconstruction map $\cT : (\cU   \times \cV)^m \rightarrow L^2_{\varrho}(\cU ; \cV)$ defined by $\cT = \cC \circ \cT_0$. Then 
\begin{equation*}
\theta_m(\bm{b}) \leq  \bbE \left ( \sup_{f \in \cH(\bm{b})} \nm{f - \cT( \{ (\bm{y}_i , f(\bm{y}_i) ) \}^{m}_{i=1} )}_{L^2_{\varrho}(\cU ; \cV)} \right ).
\end{equation*}
Now, let $C(\bm{b}_{\pi},p,{\tau})$ be the constant of \cf{t:prob-bound-known} and $E$ be the event
\begin{equation*}
E = \left \{ \sup_{f \in \cH(\bm{b})} \nm{f - \cT_0( \{ (\bm{y}_i , f(\bm{y}_i) ) \}^{m}_{i=1} )}_{L^2_{\varrho}(\cU ; \cV)} \leq  C(\bm{b}_{\pi},p,\tau)  \cdot \left ( \frac{m}{\log(m) } \right )^{\frac12-\frac1p} \right \}.
\end{equation*}
Hence, Theorem \ref{t:prob-bound-known} gives that $\bbP(E^c) \leq \epsilon$, where $\epsilon$ is as in \ef{eps-def-known}. Now suppose that $E$ occurs and recall that any $f \in \cH(\bm{b})$ satisfies
\begin{equation}\label{eq:bound_f}
\nm{f}_{L^2_{\varrho}(\cU ; \cV)} \leq \nm{f}_{L^{\infty}(\cR(\bm{b}) ; \cV)} \leq 1.
\end{equation}
Therefore, since $f = \cC(f)$ and $\cC$ is a contraction,
\begin{align*}
\nm{f - \cT( \{ (\bm{y}_i , f(\bm{y}_i)) \}^{m}_{i=1} ) }_{L^2_{\varrho}(\cU ; \cV)} & = \nm{\cC(f) - \cC \circ \cT_0( \{ (\bm{y}_i , f(\bm{y}_i)) \}^{m}_{i=1} ) }_{L^2_{\varrho}(\cU ; \cV)}
\\
& \leq \nm{f - \cT_0( \{ (\bm{y}_i , f(\bm{y}_i)) \}^{m}_{i=1} ) }_{L^2_{\varrho}(\cU ; \cV)}
\\
& \leq C(\bm{b}_{\pi},p,{\tau})  \left ( \frac{m}{\log(m)} \right )^{\frac12-\frac1p}.
\end{align*}
Conversely, if $E$ does not occur, then, due to the fact that $\nm{\cC(g)}_{L^2_{\varrho}(\cU ; \cV)} \leq 1$, we always have
\bes{
\nm{f - \cT( \{ (\bm{y}_i , f(\bm{y}_i) ) \}^{m}_{i=1} ) }_{L^2_{\varrho}(\cU ; \cV)} \leq 2.
}
We now use the law of total expectation, to obtain
\begin{align*}
\bbE  \left ( \sup_{f \in \cH(\bm{b})} \nm{f - \cT( \{ (\bm{y}_i , f(\bm{y}_i)) \}^{m}_{i=1} )}_{L^2_{\varrho}(\cU ; \cV)} \right ) 
=   & ~  \bbE \left ( \sup_{f \in \cH(\bm{b})} \nm{f - \cT( \{ (\bm{y}_i , f(\bm{y}_i) )\}^{m}_{i=1} )}_{L^2_{\varrho}(\cU ; \cV)} | E \right ) \bbP(E)
\\
&	 + \bbE \left ( \sup_{f \in \cH(\bm{b})} \nm{f - \cT( \{ (\bm{y}_i , f(\bm{y}_i) ) \}^{m}_{i=1} )}_{L^2_{\varrho}(\cU ; \cV)} | E^c \right ) \bbP(E^c)
\\
 \leq & ~  C(\bm{b}_{\pi},p,{\tau}) \left ( \frac{m}{\log(m)} \right )^{\frac12-\frac1p} + 2 \epsilon. 
\end{align*}
The first result now follows from the definition \ef{eps-def-known} of $\epsilon$. 

We now prove the second part of the theorem. Let $\bm{b} \in \ell^p(\bbN)$, $\bm{b} \geq \bm{0}$ with $\nm{\bm{b}}_{p} \leq 1$ and $t \in (p,1)$. Let $C(\bm{b}_{\pi},t,\tau)$ be the constant of \cf{t:prob-bound-known} with $p$ replaced by $t$ (since $\bm{b} \in \ell^t(\bbN)$ as $t > p$), and recall from Remark \ref{rem:darned-constants-known} that the constant $C(\bm{b}_{\pi},t,\tau) = C_0(t) \cdot C_1(\bm{b}_{\pi},t,\tau)$, where $C_0(t)$ depends on $t$ only and $C_1(\bm{b}_{\pi},t,\tau)$ is the constant from    \cf{t:weighted-lp-error}, part (a) with $q=2$, $p$ replaced by $t$ and with parameters $\bm{\alpha}_{\pi}$, $\bm{\beta}_{\pi}$ and measure $\varrho_{\bm{\alpha}_{\pi},\bm{\beta}_{\pi}}$. Applying \cf{t:weighted-lp-error}, part (b) we deduce that
\be{
\label{eq:c_tilde2}
C(\bm{b}_{\pi},t,\tau) \leq \widetilde{C}(p,t,\tau).
}
Now consider the event
\bes{
E = \left \{ \sup_{f \in \cH(\bm{b})} \nm{f - \cT_0( \{ (\bm{y}_i , f(\bm{y}_i) ) \}^{m}_{i=1} )}_{L^2_{\varrho}(\cU ; \cV)} \leq  \widetilde{C}(p,t,\tau)  \cdot \left ( \frac{m}{\log(m) } \right )^{\frac12-\frac1t} \right \}.
}
Then, Theorem \ref{t:prob-bound-known} and \ef{eq:c_tilde2} imply that $\bbP(E^c) \leq \epsilon$, where $\epsilon$ is as in \ef{eps-def-known}. We now argue in the same way as before to deduce that
\bes{          
\bbE \left ( \sup_{f \in \cH(\bm{b})} \nm{f - \cT( \{ (\bm{y}_i , f(\bm{y}_i)) \}^{m}_{i=1} )}_{L^2_{\varrho}(\cU ; \cV)} \right ) 
 \leq  \widetilde{C}(p,t,\tau)  \cdot \left ( \frac{m}{\log(m)} \right )^{\frac12-\frac1t}
}
and therefore
\begin{equation*}
  \overline{\theta_m}(p )  \leq \sup_{\substack{\bm{b} \in \ell^p(\bbN), \bm{b} \geq \bm{0} \\ \nm{\bm{b}}_{p} \leq 1 }} \bbE \left ( \sup_{f \in \cH(\bm{b})} \nm{f - \cT( \{ (\bm{y}_i , f(\bm{y}_i)) \}^{m}_{i=1} )}_{L^2_{\varrho}(\cU ; \cV)} \right )  \leq \widetilde{C}(p,t,\tau) \cdot \left( \frac{m}{\log(m)}\right)^{\frac12-\frac1t}.
  \end{equation*}  
To remove the factor $\log(m)$ from the bound we note that $t \in (p,1)$ is arbitrary. Hence, up to a possible change in the constant $\widetilde{C}(p,t,\tau)$, the result follows immediately.
\end{proof}

\begin{proof}[Proof of Theorem \ref{t:main-res-unknown}]
In order to prove the first part of Theorem \ref{t:main-res-unknown}, let
\be{
\label{eps-def-unknown}
\epsilon = m^{-\frac1p} < 1,
}
and consider the reconstruction map $\cT$ asserted by \cf{t:prob-bound-unknown} with this value of $\epsilon$. This theorem implies that $\bbP(E^c) \leq \epsilon$, where
\begin{equation*}
E = \bigcap_{\substack{\bm{b} \in \ell^p_{\mathsf{M}}(\bbN),\bm{b} \geq 0 \\ 0 < p < 1}} \left \{ \sup_{f \in \cH(\bm{b})} \nm{f - \cT(\{ (\bm{y}_i , f(\bm{y}_i) )\}^{m}_{i=1}) }_{L^2_{\varrho}(\cU ; \cV)} \leq C(\bm{b},p,g,\tau)  \left ( \frac{m}{\log^{4}(m) g(m)} \right )^{1/2-1/p} \right \}.
\end{equation*}
Consider arbitrary $0 < p <1$, $\bm{b} \in \ell^p_{\mathsf{M}}(\bbN)$, $\bm{b} \geq \bm{0}$, and $f \in \cH(\bm{b})$.
If the event $E$ occurs, then we have 
\begin{equation*}
 \nm{f - \cT(\{ (\bm{y}_i , f(\bm{y}_i) )\}^{m}_{i=1}) }_{L^2_{\varrho}(\cU ; \cV)} \leq C(\bm{b},p,g,\tau)  \left ( \frac{m}{\log^{4}(m) g(m)} \right )^{1/2-1/p}.
\end{equation*}
Conversely, suppose that the event $E$ does not occur. Then
\begin{equation}
\nm{\cT(\{ (\bm{y}_i , f(\bm{y}_i)) \}^{m}_{i=1}) }_{L^2_{\varrho}(\cU ; \cV)} = \nm{\hat{\bm{c}}_{\Lambda}}_{2;\cV} \leq \nm{\hat{\bm{c}}_{\Lambda}}_{1,\bm{u} ; \cV}.
\end{equation}
Writing $\bm{f} = (f(\bm{y}_i))^{m}_{i=1}$ and using the fact that  $\hat{\bm{c}}_{\Lambda}$ is a minimizer of \eqref{eq:def_Tunkn}, we have  
\begin{equation*}
\lambda \nm{\hat{\bm{c}}_{\Lambda}}_{1,\bm{u} ; \cV} \leq \lambda \nm{\bm{0}}_{1,\bm{u} ; \cV} + m^{-1/2} \nm{\bm{f}}_{2;\cV} \leq \nm{f}_{L^{\infty}(\cU ; \cV)}\leq 1.
\end{equation*}
Using \ef{k-def} and \ef{lambda-def}, we deduce that
\be{
\label{CS-error-bad-event}
\nm{f - \cT(\{ (\bm{y}_i , f(\bm{y}_i) )\}^{m}_{i=1}) }_{L^2_{\varrho}(\cU ; \cV)}  \leq \nm{f}_{L^2_{\varrho}(\cU ; \cV)} + {1}/{\lambda}  \lesssim \sqrt{m}.
}
Having bounded the reconstruction error when $E$ does or does not occur, we now argue as in the proof of \cf{t:main-res-known} by the law of total expectation to obtain the final result.

The proof of the second part of Theorem \ref{t:main-res-unknown} uses similar arguments to those in \cite[Theorem~4.7]{adcock2024optimal}. Fix $0 < p < 1$. Then
\cf{t:prob-bound-unknown} implies that $\bbP(\widetilde{E}^c) \leq \epsilon$, where
\begin{equation*}
\widetilde{E} = \bigcap_{\substack{\bm{b} \in \ell^t_{\mathsf{M}}(\bbN),\bm{b} \geq 0 \\ t \in (p,1)}} \left \{ \sup_{f \in \cH(\bm{b})} \nm{f - \cT(\{ (\bm{y}_i , f(\bm{y}_i) )\}^{m}_{i=1}) }_{L^2_{\varrho}(\cU ; \cV)} \leq C(\bm{b},t,g,\tau)  \left ( \frac{m}{\log^{4}(m) g(m) } \right )^{1/2-1/t} \right \}.
\end{equation*}
From Remark~\ref{rem:darned-constants}, we note that $C(\bm{b},t,g,\tau)= C_0(t,g) \cdot (C_1(\bm{b},t,\tau)+C_{2}(\bm{b},t,\tau))$, where $C_1(\bm{b},t,\tau)$ is the constant of Theorem \ref{t:weighted-lp-error}, part (a) with $q=1$ and $C_2(\bm{b},t,\tau)$ is the constant of Theorem \ref{t:weighted-lp-error-monotone}, part (a) with $q=1$.
 Now for any $t \in (p,1)$ applying Theorem \ref{t:weighted-lp-error}, part (b) and Theorem \ref{t:weighted-lp-error-monotone}, part (b) with $r = t$, we get that
\begin{equation} \label{eq:bound_eqs}
\sup_{\|\bm{b}\|_{p,\mathsf{M}}\leq 1}  C(\bm{b},t,g,\tau) \leq  C_0(t,g) \cdot \left(
\sup_{\|\bm{b}\|_{p,\mathsf{M}}\leq 1}  C_1(\bm{b},t,\tau) +
\sup_{\|\bm{b}\|_{p,\mathsf{M}}\leq 1}  C_{2}(\bm{b},t,\tau)  \right) \leq 
\widetilde{C}(p,t,g,\tau).
\end{equation}
Now recall that
\begin{equation*}
\cH(p,\mathsf{M}) = \bigcup \lbrace \cH(\bm{b}): \bm{b} \in \ell^p_{\mathsf{M}} (\bbN), \bm{b} \geq 0, \|\bm{b}\|_{p,\mathsf{M}} \leq 1 \rbrace.
\end{equation*}
Since $t \in (p,1)$, any $\bm{b} \in  \ell^p_{\mathsf{M}}(\bbN)$ satisfies $\bm{b} \in  \ell^t_{\mathsf{M}}(\bbN)$,  and therefore we have 
\begin{align*}
\widetilde{E} 
& \subseteq \bigcap_{\substack{\bm{b} \in \ell^p_{\mathsf{M}}(\bbN),\bm{b} \geq 0 \\ \nm{\bm{b}}_{p,\mathsf{M}} \leq 1}} \left \{ \sup_{f \in \cH(\bm{b})} \nm{f - \cT(\{ (\bm{y}_i , f(\bm{y}_i) )\}^{m}_{i=1}) }_{L^2_{\varrho}(\cU ; \cV)} \leq C(\bm{b},t,g,\tau)  \left ( \frac{m}{\log^{4}(m) g(m) } \right )^{1/2-1/t} \right \} \\
 &\subseteq \bigcap_{\substack{\bm{b} \in \ell^p_{\mathsf{M}}(\bbN),\bm{b} \geq 0 \\ \nm{\bm{b}}_{p,\mathsf{M}} \leq 1}} \left \{ \sup_{f \in \cH(\bm{b})} \nm{f - \cT(\{ (\bm{y}_i , f(\bm{y}_i)) \}^{m}_{i=1}) }_{L^2_{\varrho}(\cU ; \cV)} \leq \widetilde{C}(p,t,g,\tau)  \left ( \frac{m}{\log^{4}(m) g(m) } \right )^{1/2-1/t} \right \} 
 \\
& = \left \{ \sup_{f \in \cH(p,\mathsf{M})} \nm{f - \cT(\{ (\bm{y}_i , f(\bm{y}_i)) \}^{m}_{i=1}) }_{L^2_{\varrho}(\cU ; \cV)} \leq \widetilde{C}(p,t,g,\tau)  \left ( \frac{m}{\log^{4}(m) g(m)} \right )^{1/2-1/t} \right \}  
 =: E.
\end{align*}
Hence $\bbP(E) \geq \bbP(\widetilde{E}) \geq 1-\epsilon$. We now argue as above. We first bound the error using \ef{CS-error-bad-event} when the event $E$ does not occur. We then use the law of total expectation to deduce the bound
\bes{
\theta(p,\mathsf{M}) \leq \bbE \left ( \sup_{f \in \cH(p,\mathsf{M})} \nm{f - \cT(\{ (\bm{y}_i , f(\bm{y}_i)) \}^{m}_{i=1}) }_{L^2_{\varrho}(\cU ; \cV)} \right ) \leq \widetilde{C}(p,t,g,\tau)  \left ( \frac{m}{\log^4(m) g(m)} \right )^{1/2-1/t}.
}
To remove the $\log^4(m) g(m)$ factor from the bound we note that $t\in (p,1)$ is arbitrary and $g$ grows subalgebraically fast in $m$. Then, up to a possible change in the constant $\widetilde{C}(p,t,g,\tau)$, the result follows immediately.
\end{proof}

\section*{Acknowledgments}
BA acknowledges the support of the Natural Sciences and Engineering Research Council of Canada of Canada (NSERC) through grant RGPIN-2021-611675. They would like to thank David Krieg for useful discussions, and the two anonymous referees, whose feedback greatly improved the paper.

\appendix

\section{Supporting results on Jacobi polynomials}\label{app:poly-ests}

In this appendix, we present several supporting results on Jacobi polynomials that are used in the main paper.
As in \S \ref{sec:jacobi-polys}, let $\alpha,\beta > - 1$ and $P^{\alpha,\beta}_{\nu}(y)$ denote the standard Jacobi polynomial of degree $\nu \in \bbN_0$. Note that $P^{\alpha,\beta}_{0} \equiv 1$. These polynomials are orthogonal with respect to the Jacobi measure \ef{jacobi-meas}, and  satisfy the orthogonality relation
\bes{
\int^{1}_{-1} P^{\alpha,\beta}_{\nu}(y) P^{\alpha,\beta}_{\mu}(y) \D \omega_{\alpha,\beta}(y) = \delta_{\nu \mu} {h^{\alpha,\beta}_{\nu}}, 
}
where
\begin{equation}\label{eq:h_jacobi}
h^{\alpha,\beta}_{\nu} = \left(\frac{2^{\alpha+\beta+1}}{2 \nu + \alpha + \beta + 1} \right) \frac{\Gamma(\nu+\alpha+1) \Gamma(\nu + \beta + 1)}{\Gamma(\nu+\alpha+\beta+1) \Gamma(\nu+1)},\quad \nu \in \bbN_0.
\end{equation}
See \cite[Chpt.\ 4]{szego1975orthogonal} (also \cite[Sec.~2]{guo2009generalized}, \cite[Table 18.3.1]{oliver2010nist}).
Therefore, the orthonormal Jacobi polynomials with respect to the probability measure \ef{1d-jacobi-meas} can be expressed as
\begin{equation}
\label{Jacobi-ON}
\Psi^{\alpha,\beta}_{\nu} = ( h^{\alpha,\beta}_0 / h^{\alpha,\beta}_{\nu} )^{1/2} P^{\alpha,\beta}_{\nu}, \quad  \forall \nu \in \bbN_0.
\end{equation}
We now show a key technical lemma (Lemma \ref{l:jacobi-uniform-norm}) on the uniform growth of the Jacobi polynomials. This lemma was used in the proofs of Theorems \ref{t:weighted-lp-error}--\ref{t:weighted-lp-error-monotone}. Its proof relies on the following result.
\lem{
Let $\alpha,\beta > -1$, $\varrho = \varrho_{\alpha,\beta}$ be as in \ef{1d-jacobi-meas} and $P$ be any polynomial of degree $\nu \in \bbN$. Then
\bes{
\nm{P}_{L^{\infty}([-1,1])} \leq \sqrt{ \frac{9 \cdot 2^{\alpha+\beta+1} \Gamma(\alpha+1) \Gamma(\beta+1)}{\Gamma(\alpha+\beta+2)}} (\sqrt{3} \nu)^{1+|\alpha|+|\beta|} \nm{P}_{L^2_{\varrho}([-1,1])}.
}
}
\prf{
By the Markov brother's inequality, $\nm{P'}_{L^{\infty}([-1,1])} \leq \nu^2 \nm{P}_{L^{\infty}([-1,1])}$. Hence there is an interval $I$ of length $1/(3 \nu^2)$ with $I \subseteq [-1+1/(3 \nu^2) , 1 - 1/(3 \nu^2) ]$ for which $|P(y)| \geq \nm{P}_{L^{\infty}([-1,1])} / 3$, $\forall y \in I$. Using this, the fact that $\varrho = \varrho_{\alpha,\beta} = \omega_{\alpha,\beta} / h^{\alpha,\beta}_0$, where $h^{\alpha,\beta}_0$ is as in \ef{eq:h_jacobi}, and \ef{jacobi-meas}, we get
\eas{
\frac{\nm{P}^2_{L^{\infty}([-1,1])}}{\nm{P}^2_{L^2_{\varrho}([-1,1])}} \leq \frac{9}{\int_{I} \D \varrho(y)} \leq \frac{9 \cdot (3 \nu^2) \cdot h^{\alpha,\beta}_0}{\min_{|y| \leq 1-1/(3 \nu^2)} \{ (1-y)^{\alpha}(1+y)^{\beta} \} }.
}
Now observe that $\min_{|y| \leq 1-1/(3 \nu^2)} (1-y)^{\alpha} = (3 \nu^2)^{-\alpha}$ if $\alpha \geq 0$ and $\min_{|y| \leq 1-1/(3 \nu^2)} (1-y)^{\alpha} \geq 2^{\alpha}$ if $-1 < \alpha < 0$. Since $\nu \geq 1$, we deduce that $\min_{|y| \leq 1-1/(3 \nu^2)} (1-y)^{\alpha} \geq (3 \nu^2)^{-|\alpha|}$. Using this and \ef{eq:h_jacobi}, we obtain
\bes{
\frac{\nm{P}^2_{L^{\infty}([-1,1])}}{\nm{P}^2_{L^2_{\varrho}([-1,1])}} \leq 9 \cdot (3 \nu^2)^{1+|\alpha|+|\beta|} \frac{2^{\alpha+\beta+1} \Gamma(\alpha+1) \Gamma(\beta+1)}{\Gamma(\alpha+\beta+2)},
}
as required.
}

\lem{
\label{l:jacobi-uniform-norm}
Suppose that $\tau - 1 \leq \alpha,\beta \leq 1/\tau$ for some $\tau > 0$ and consider the orthonormal Jacobi polynomials $\{\Psi^{\alpha,\beta}_{\nu} \}_{\nu \in \bbN_0}$ with respect to \ef{1d-jacobi-meas}. Then 
\bes{
\nm{\Psi^{\alpha,\beta}_{\nu}}_{L^{\infty}([-1,1])} \leq c ( 1 + \nu)^{\gamma},\quad \forall \nu \in \bbN_0,
}
where $c = c(\tau)$ and $\gamma = \gamma(\tau)$ depend on $\tau$ only.
}
\prf{
Since $\Psi^{\alpha,\beta}_{0} \equiv 1$, the result holds for $\nu = 0$ for any $c \geq 1$ and $\gamma \geq 0$. Now suppose that $\nu \geq 1$. By the previous lemma,
\bes{
\nm{\Psi^{\alpha,\beta}_{\nu}}_{L^{\infty}([-1,1])} \leq \sqrt{ \frac{9 \cdot 2^{\alpha+\beta+1} \Gamma(\alpha+1) \Gamma(\beta+1)}{\Gamma(\alpha+\beta+2)}} (\sqrt{3} \nu)^{1+|\alpha|+|\beta|}.
}
On $(0,\infty)$ the Gamma function $\Gamma(z)$ is increasing as $z \rightarrow 0^+$ or $z \rightarrow \infty$ and has a unique minimum at $z_* = 1.461\ldots$ with $\Gamma_* = \Gamma(z_*)= 0.885\ldots > 0$. Using this and the fact that $|\alpha| , | \beta | \leq 1/\tau$, we obtain
\bes{
\nm{\Psi^{\alpha,\beta}_{\nu}}_{L^{\infty}([-1,1])} \leq 3^{1/\tau+3/2} \cdot 2^{1/\tau + 1/2} \cdot \frac{\max \{ \Gamma(\tau) , \Gamma(1/\tau+1) \}}{\sqrt{\Gamma_*}} \cdot (1+\nu)^{1+2 /\tau }.
}
This gives the result.
}

Next, we consider the following assumption, which we show is equivalent to Assumption \ref{main-ass} (recall Remark \ref{rem:Jac-cond-equiv}).
\begin{assumption}
\label{main-ass-alt}
The sequences $\bm{\alpha},\bm{\beta} > - \bm{1}$ are such that there is a polynomial $p$ depending on $\bm{\alpha}$ and $\bm{\beta}$ for which the univariate Jacobi polynomials satisfy
\bes{
\nm{\Psi^{\alpha_j,\beta_j}_{\nu_j}}_{L^{\infty}([-1,1])} \leq p(\nu_j),\quad \forall \nu_j \in \bbN_0,\ j \in \bbN.
}
\end{assumption}

\prop{
\label{prop:equiv-ass}
Assumption \ref{main-ass-alt} is equivalent to Assumption \ref{main-ass}. 
}
\prf{
Lemma \ref{l:jacobi-uniform-norm} shows that Assumption \ref{main-ass} implies Assumption \ref{main-ass-alt}. Conversely, suppose that Assumption \ref{main-ass-alt} holds. Recall that $P^{\alpha,\beta}_{\nu}(1) = {\nu + \alpha \choose \nu }$ and also that $P^{\alpha,\beta}_{\nu}(-1) = (-1)^{\nu} {\nu + \beta \choose \nu }$. Therefore
\bes{
\nm{\Psi^{\alpha,\beta}_{\nu}}^2_{L^{\infty}([-1,1])} \geq \max \{ |\Psi^{\alpha,\beta}_{\nu}(\pm 1) |^2 \} =\left ( \frac{h^{\alpha,\beta}_0}{h^{\alpha,\beta}_{\nu}} \right ) \max \left \{ {\nu + \alpha \choose \nu } , {\nu + \beta \choose \nu } \right \}^2.
}
Notice from \ef{eq:h_jacobi} that $h^{\alpha,\beta}_{\nu} \sim 2^{\alpha+\beta} \nu^{-1}$ and ${\nu + \alpha \choose \nu }  \sim \frac{\nu^{\alpha} \E^{-\alpha}}{\Gamma(\alpha+1)}$ as $\nu \rightarrow \infty$ for fixed $\alpha,\beta$. Therefore, Assumption \ref{main-ass-alt} can only hold if $\alpha_i,\beta_i$ are bounded above uniformly in $i$.

On the other hand, consider the quadratic polynomial $P^{\alpha,\beta}_{2}(x)$, which has explicit expression (see, e.g., \cite[eqn. (18.5.7)]{oliver2010nist})
\bes{
P^{(\alpha,\beta)}_{2}(y) = \frac{(\alpha+1)(\alpha+2)}{2} + (\alpha+2)(\alpha+\beta+3) \frac{y-1}{2} + \frac{(\alpha+\beta+3)(\alpha+\beta+4)}{2} \left ( \frac{y-1}{2} \right )^2.
}
This polynomial has a minimum at $y^* = \frac{\beta-\alpha}{4+\alpha+\beta} \in [-1,1]$. Substituting this value and using the fact that $\alpha,\beta > -1$, we get
\bes{
\nm{P^{(\alpha,\beta)}_2}_{L^{\infty}([-1,1])} \geq | P^{(\alpha,\beta)}_{2}(y^*) | = \frac{(\alpha+2)(\beta+2)}{2(\alpha+\beta+4)} \geq \frac{\max \{ \alpha,\beta \} + 2 }{2 ( 2\max \{ \alpha , \beta \} + 4 )} = \frac14.
}
We now apply \ef{eq:h_jacobi} and use the fact that $\alpha,\beta > -1$ once more to obtain
\bes{
\nms{\Psi^{(\alpha,\beta)}_2}^2_{L^{\infty}([-1,1])} \geq \frac{(\alpha+\beta+5)(\alpha+\beta+2)}{32(\alpha+2)(\alpha+1)(\beta+2)(\beta+1)} \geq \frac{2(\alpha+\beta+5)}{(\alpha+2)(\beta+2)} \max \left \{ \frac{1}{\alpha+1} , \frac{1}{\beta+1} \right \}.
}
The right-hand side tends to infinity if $\alpha$ or $\beta$ tends to $-1$. Hence Assumption \ref{main-ass-alt} can only hold if $\alpha_i,\beta_i$ are bounded away from $-1$ uniformly in $i$. This completes the proof.
}

\section{The relationship between $\overline{\theta_m}(p,\mathsf{M})$ and $\overline{\theta_m}(p)$}\label{app:theta_bar_m}

In this appendix, we elaborate upon the discussion in Remark \ref{rem:theta_bar_equate}. First, we show the following.

\lem{
Let $\varrho = \varrho_1 \times \varrho_2 \times \cdots$ be a tensor-product probability measure on $\cU$, where each $\varrho_i$ is a probability measure on $[-1,1]$, $\bm{b} = (b_i)_{i \in \bbN} \in [0,\infty)^{\bbN}$ and $\pi : \bbN \rightarrow \bbN$ be a bijection. Consider the quantity $\theta_m(\bm{b})$ defined in \ef{theta-m-b}. Then
\bes{
\theta_m(\bm{b}) = \Theta_m(\cH(\bm{b}) ; C(\cU ; \cV) , L^2_{\varrho}(\cU ; \cV) ) = \Theta_m(\cH(\bm{b}_{\pi}) ; C(\cU ; \cV) , L^2_{\varrho_{\pi}}(\cU ; \cV) ),
}
where $\Theta_m(\cdot ; \cdot , \cdot )$ is as in \ef{Theta_m-def}, $\bm{b}_{\pi} = (b_{\pi(i)})_{i \in \bbN}$ and $\varrho_{\pi} = \varrho_{\pi(1)} \times \varrho_{\pi(2)} \times \cdots$ is the permuted measure. In particular, if $\varrho_1 = \varrho_2 = \cdots$ then
\bes{
\theta_m(\bm{b}) = \theta_m(\bm{b}_{\pi}).
}
}
\prf{
The first equality is just the definition \ef{theta-m-b} of $\theta_m(\bm{b}) $. We now prove the second.
By \ef{Theta_m-def},
\bes{
\theta_m(\bm{b})  = \inf_{\cL,\cT}  \sup_{f \in \cH(\bm{b})} \nm{f - \cT (\cL(f)) }_{L^2_{\varrho}(\cU ; \cV)},
}
where the infimum is taken over all adaptive sampling operators $\cL : C(\cU ; \cV) \rightarrow \cV^m$ and reconstruction maps $\cT : \cV^m \rightarrow L^2_{\varrho}(\cU ; \cV)$. By \ef{Hb-permute}, we can write 
\bes{
\theta_m(\bm{b}) = \inf_{\cL,\cT} \sup_{g \in \cH(\bm{b}_{\pi})} \nm{g \circ \pi - \cT(\cL(g \circ \pi)) }_{L^2_{\varrho}(\cU ; \cV)}.
}
Given any $\cL$, let $\cL' : C(\cU ; \cV) \rightarrow \cV^m$ be defined by $\cL'(g) = \cL(g \circ \pi)$. Notice that $\cL'$ is also an adaptive sampling operator. Moreover, any adaptive sampling operator $\cL'$ can be expressed in this form for some adaptive sampling operator $\cL$. Similarly, given any $\cT$, define $\cT' : \cV^m \rightarrow L^2_{\varrho_{\pi}}(\cU ; \cV)$ by $\cT'(\bm{f})(\cdot) = \cT(\bm{f})(\pi(\cdot))$ for any $\bm{f} \in \cV^m$. Notice that $\cT'$ is a reconstruction map, and that any reconstruction map $\cT' : \cV^m \rightarrow L^2_{\varrho_{\pi}}(\cU ; \cV)$ can be expressed in the same way for some reconstruction map $\cT : \cV^m \rightarrow L^2_{\varrho}(\cU ; \cV)$. Therefore
\bes{
\theta_m(\bm{b}) = \inf_{\cL',\cT'}  \sup_{g \in \cH(\bm{b}_{\pi})} \nm{g \circ \pi(\cdot) - \cT'(\cL'(g))(\pi(\cdot)) }_{L^2_{\varrho}(\cU ; \cV)} =  \inf_{\cL',\cT'}  \sup_{g \in \cH(\bm{b}_{\pi})}\nm{g (\cdot) - \cT'(\cL'(g))(\cdot) }_{L^2_{\varrho_{\pi}}(\cU ; \cV)}.
}
This gives the first result. For the second result, we simply notice that $\varrho = \varrho_{\pi}$ whenever $\varrho$ is a tensor product of the same one-dimensional measure.
}

The previous lemma implies that $\theta_m(\bm{b})$ has no reason to coincide with $\theta_{m}(\bm{b}_{\pi})$, due to the change of measure from $\varrho$ to $\varrho_{\pi}$. However, this does hold whenever $\varrho$ is a tensor product of the same measure. We now show that the terms $\overline{\theta_m}(p)$ and $\overline{\theta_m}(p,\mathsf{M})$ in \ef{eq:def_theta} also coincide whenever this holds.
\prop{
Let $\varrho = \varrho_1 \times \varrho_1 \times \cdots $  be a tensor-product probability measure on $\cU$, where $\varrho_1$ is a probability measure on $[-1,1]$. Then the terms in \ef{eq:def_theta} coincide, i.e., $\overline{\theta_m}(p) = \overline{\theta_m}(p,\mathsf{M})$.
}
\prf{
We always have that $\overline{\theta_m}(p,\mathsf{M}) \leq \overline{\theta_m}(p)$. We now show the opposite inequality. Let $\bm{b} \in \ell^p(\bbN)$ with $\bm{b} \geq \bm{0}$ and $\nm{\bm{b}}_{p} \leq 1$. Let $\pi : \bbN \rightarrow \bbN$ be a bijection that gives a nonincreasing rearrangement of $\bm{b}$, i.e., $b_{\pi(1)} \geq b_{\pi(2)} \geq \cdots$. Then $\bm{b}_{\pi} \in \ell^p_{\mathsf{M}}(\bbN)$ and $\nm{\bm{b}_{\pi}}_{p,\mathsf{M}} = \nm{\bm{b}}_{p} \leq 1$. Using the previous lemma and the definition of $\overline{\theta_m}(p,\mathsf{M})$, we deduce that
\bes{
\theta_m(\bm{b}) = \theta_m(\bm{b}_{\pi}) \leq \overline{\theta_m}(p,\mathsf{M}).
}
Since $\bm{b}$ was arbitrary, we obtain the result.
}

\section{Notation}
Table \ref{tab:notation} summarizes main notation used in the paper. Note that we also consider scalar-valued versions of various spaces and norms, in which case we drop the $;\cV$ symbol used in the table.

     \begin{table}[t!]
     \caption{Table of notation.}
     
     \vspace{-0.5cm}
\begin{longtable}{ @{}  c | p{.75\textwidth} } 
            $\cU$ &  Domain of the function, equal to $[ -1,1]^{\bbN}$. \\ 
            $\cV$, $\ip{\cdot}{\cdot}_{\cV}$, $\nms{\cdot}_{\cV}$ &  The codomain of the function, a Hilbert space with inner product $\ip{\cdot}{\cdot}_{\cV}$ and norm $\nms{\cdot}_{\cV}$.\\ 
            $f$ & Function to approximate, $f : \cU \rightarrow \cV$.\\ 
            $\bm{y} _1,\ldots,\bm{y} _m$ & Sample points in $\cU$. \\ 
            $\bm{\alpha},\bm{\beta}$ & Sequences defining the tensor-product Jacobi measure and polynomials. See \S\ref{sec:jacobi-polys}. \\
	   $\Psi_{\bnu} = \Psi^{\bm{\alpha},\bm{\beta}}_{\bm{\nu}}$     &    The $\bnu$th tensor-product Jacobi polynomial. \\ 
            $\bm{c}_{\bnu} = \bm{c}^{\bm{\alpha},\bm{\beta}}_{\bnu}$ & The coefficient of a function corresponding to $\Psi_{\bnu}$.  \\ 
            $f_{S}$ & Truncated expansion of $f$ corresponding to multi-indices $\bnu \in S$. \\ 
  	$\cH(\bm{b})$ & The class of $(\bm{b},1)$-holomorphic functions with $L^{\infty}$-norm at most one. See~\S\ref{S:spaces}. \\ 
  			$\cH(p),\cH(p,\mathsf{M})$ & The union of $\cH(\bm{b},1)$ for  $\|\bm{b}\|_p \leq 1$ or $\|\bm{b}\|_{p,\mathsf{M}} \leq 1$, respectively, where $0<p<1$. See~\S\ref{S:spaces}. \\ 
            $\theta_m(\bm{b}),\overline{\theta_m}(p),\overline{\theta_m} (p,\mathsf{M})$ & Various $m$-widths for the known anisotropy case. See \eqref{theta-m-b}--\eqref{eq:def_theta}.\\
            
            $\theta_m(p),\theta_m(p,\mathsf{M})$ & Various $m$-widths for the unknown anisotropy case. See \eqref{theta-upsilon-unknown-aniso}.
            \\ 
            $\ell^p(\Lambda ; \cV)$, $\nms{\cdot}_{p;\cV}$     &   The space of $\ell^p$-summable $\cV$-valued sequences indexed over $\Lambda$ and its (quasi-)norm. See \S \ref{S:ellspaces}. \\
             $\ell^p_{\bm{w}}(\Lambda ; \cV)$, $\nms{\cdot}_{p,\bm{w};\cV}$     &      The space of weighted $\ell^p$-summable $\cV$-valued sequences indexed over $\Lambda$ and its (quasi-)norm. See \S \ref{S:ellspaces}. \\
              $\ell^p_{\mathsf{A}}(\Lambda;\cV)$, $\nms{\cdot}_{p,\mathsf{A};\cV}$ & The space of anchored $\ell^p$-summable $\cV$-valued sequences and its (quasi-) norm. See \S \ref{S:ellspaces}. \\
            $\sigma_s(\cdot)_{p;\cV}$ & The $\ell^p$-norm best $s$-term approximation error. See \eqref{sigma-s-def}. \\ 
            $\sigma_k(\cdot)_{p,\bm{w}}$ & The $\ell^p_{\bm{w}}$-norm weighted best $(k,\bm{w})$-term approximation error. See \eqref{weighted-k-w-term}.\\ 
            $\sigma_{s,\mathsf{A}}(\cdot)_{p;\cV}$ & The $\ell^p$-norm best $s$-term approximation error in anchored sets. See~\eqref{def:best_anch}.\\
$L^p_{\varrho}(\cU;\cV)$, $\nms{\cdot}_{L^p_{\varrho}(\cU ; \cV)}$  &      The Lebesgue-Bochner space and its norm. See \ef{def:normF}.\\
        $\nms{\cdot}_0$ & The $0$-norm of a multi-index or vector, equal to its number of nonzero entries. \\
        $\nms{\cdot}_{\mathsf{disc}}$ & The discrete seminorm given by $\|g\|_{\mathsf{disc}}^2=m^{-1}\sum_{i=1}^m\|g(\bm{y}_i)\|_{\cV}^2$.
\end{longtable}

\label{tab:notation}
\end{table}

\bibliographystyle{abbrv}
\small
\bibliography{optimalb_bib}

\end{document}